%% file: Kernel.tex
\documentclass[a4paper,10pt,oneside]{amsart}
\usepackage[a4paper, textwidth=15.5cm]{geometry}
\usepackage{amsmath,amsthm,amssymb,amsfonts,extarrows}
\usepackage{enumerate,amscd,amsxtra,MnSymbol}
\usepackage{mathrsfs}
\usepackage{color}
\usepackage[cmtip,all]{xy}
\usepackage{eucal}
\usepackage{bbold}
\usepackage{ upgreek }
\usepackage{paralist}
\usepackage{tikz-cd}
\usepackage{stackrel}
\usepackage{leftindex}

\usepackage[dvipsnames]{xcolor} 
\usepackage[backref=page]{hyperref}

\hypersetup{
    colorlinks=true,
    linkcolor=PineGreen, 
    citecolor=BurntOrange,
    urlcolor=CadetBlue,
    }

\usepackage{thmtools} 
\usepackage{quiver}
\usepackage{tikz-cd}
\usepackage{csquotes}
\usepackage{pdfpages}
\usepackage{graphicx}

\input{Comands}

\title{Kernel theorems for rigidly-compactly generated $\infty$-categories}
\author{Giovanni Rossanigo}

\begin{document}
\maketitle

\begin{abstract}
    We prove two representability results for rigidly-compactly generated $\infty$-categories and functors between them. 
    The first one represents contravariant linear functionals out of a category of perfect objects with values in a category of (pseudo)-coherent objects in terms of (pseudo)-coherent objects. 
    The second one represents covariant functionals out of coherent objects with values in a category of coherent objects in terms of  perfect objects. 
    The techniques used belong to the realm of  \enquote{functional analysis} of presentable stable categories and ultimately depend on the interaction between three notion of finiteness, namely compactness, dualizability and coherence. 
    These results apply to $\mb{E}_\infty$-ring spectra, quasi-proper maps of quasi-compact quasi-separated schemes and certain spectral algebraic spaces.
    We also reformulate Grothendieck duality in terms of internal left adjoints.
\end{abstract}

\begin{figure}[h]
\centering
\includegraphics[scale=0.2]{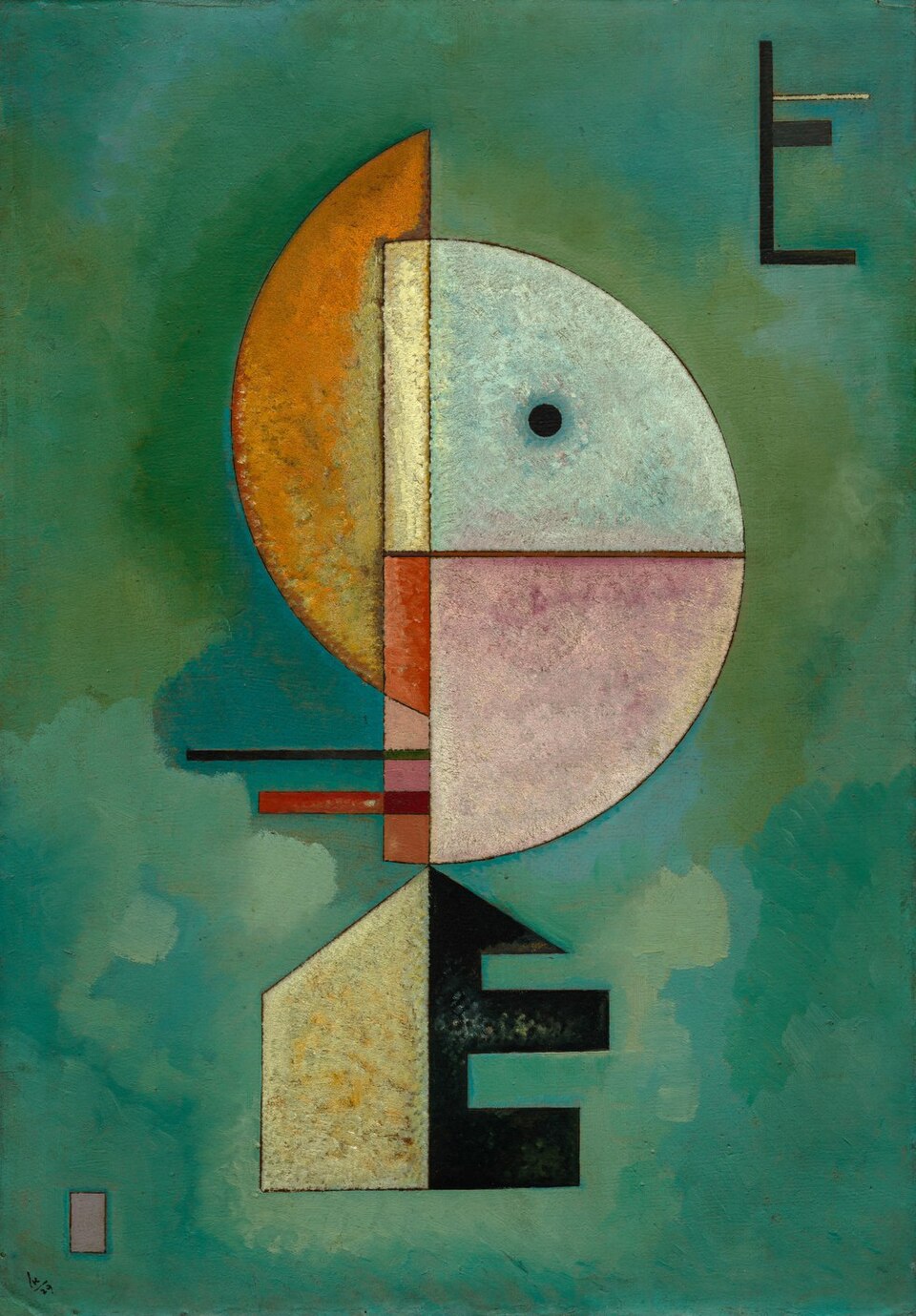}
\caption{\emph{Empor}, Vasilij Kandinskij, 1929.}
\end{figure}

\newpage
\tableofcontents

\section{Introduction}

Many of the stable categories arising in homotopy theory and derived algebraic geometry are \emph{rigidly-compactly generated}: they are compactly generated stable categories with a compatible symmetric monoidal structure such that the monoidal unit is compact and compact objects are dualizable.  
Typical examples include the $\infty$-category of spectra $\Sp$, module categories over connective $\mb{E}_\infty$-rings, and the derived $\infty$-category $\QCoh(X)$ of quasi-coherent sheaves on a quasi-compact quasi-separated scheme $X$. 
In such contexts it is fruitful to regard compact objects as the
\enquote{finite} part of the theory, while various $t$-structures single out further finiteness conditions (namely, almost compactness and bounded almost compactness in homotopy theory and pseudo-coherence and coherence in algebraic geometry).

A fundamental input of rigidly-compactly generated categories is that \emph{colimit-preserving} linear functors are controlled by \emph{kernels}.  
To be concrete, let $\Cat^\perf$ denote the $\infty$-category of small idempotent-complete stable $\infty$-categories and equip it with the symmetric monoidal structure constructed by Blumberg, Gepner, and Tabuada  in \cite{blumberg2013universal} and recall that a commutative algebra $A\in\CAlg(\Cat^\perf)$ is rigid if and only if every object is dualizable. 
Then the main \enquote{kernel theorem} asserts that for every  pair $\mathcal{M},\mathcal{N}\in\Mod_\mathcal{A}(\Cat^\perf)$ of small idempotent-complete $\mathcal{A}$-modules, the $\mathcal{A}$-linear Yoneda embedding  identifies $\Ind(\mathcal{M}^{\op}\otimes_\mathcal{A} \mathcal{N})$ with $\mathcal{A}$-linear exact functors $\mathcal{M}\to\Ind(\mathcal{N})$; here $\Ind$ denotes the ind-completion.
The goal of this paper is to establish two \emph{finitary} kernel theorems that refine the above equivalence when the source is restricted to \enquote{finite objects in the $t$-structure}.

These finite objects are defined for \emph{tor-finite categories} and the finitary kernel theorems are statements about functors between them.
To be concrete, a tor-finite category is a rigidly-compactly generated category $\mathcal{A}\in\CAlg^\rig(\Pr^{\L,\omega}_\st)$ equipped with a $t$-structure generated by a set of compact objects $S_\mathcal{A}\subseteq\mathcal{A}^\omega$.
This $t$-structure is required to be accessible, compatible with filtered colimits, right and left complete, and since for this data there exists a notion of \emph{finite tor-dimension} one requires that compact objects are of finite tor-dimension, so that tensoring with compacts has uniformly bounded $t$-exactness properties.
On the side of morphisms, a functor between tor-finite categories is just a rigid functor which is right $t$-exact and has a right $t$-exact up to a shift right adjoint.

As said before, in a tor-finite category $\mathcal{A}$ one has well-behaved notions of  \emph{pseudo-coherence} and \emph{coherence}.
We denote by $\PCoh{A}$ and $\Coh{A}$ the corresponding full
subcategories.
Pseudo-coherent objects  are defined to be the connective objects in the $t$-structure whose coconnective truncation defines a compact object in the coconnective part, and coherent objects are simply the bounded pseudo-coherent objects. 
Under mild assumptions, these objects relate to Neeman (bounded) pseudo-compact objects studied in \cite[Definition 0.20]{neeman2025triangulatedcategoriessinglecompact}.

As the Grothendieck–Neeman duality applies to rigid functors that have a compact object-preserving right adjoint (for us the \emph{quasi-perfect} functors), our finitary kernel theorems apply to functors of tor-finite categories that have a  pseudo-coherent object-preserving right adjoint, that is, to \emph{quasi-proper} functors in the terminology (which we adopt) of  Lipman-Neeman \cite[Page 3]{lipman2007quasi}.
In particular, we prove the first finitary kernel theorem, which asserts that, for a quasi-proper functor $f^\L:\mathcal{A}\to\mathcal{B}$, the exact $\mathcal{A}^\omega$-linear functors on compact objects with pseudo-coherent (respectively coherent) values are represented by pseudo-coherent (respectively coherent) kernels in $\mathcal{B}$.
\begin{theorem}[Functors out of $(\mathcal{B}^\omega)^\op$]\label{theorem: functors out of compacts}
   Let $f^\L:\mathcal{A}\to\mathcal{B}$ be a quasi-proper functor and assume that $S_\text{fin}\subseteq S_\mathcal{B}$ is a finite generating set. 
    Assume furthermore  that the family $(\text{ev}_s)_{s\in S_\mathcal{B}}$ detects coconnective objects. 
    Then there are equivalences of categories
    \[
    \PCoh{B}\to\Fun_{\mathcal{A}^\omega}^\ex((\mathcal{B}^\omega)^\op,\PCoh{A}),\qquad
    \Coh{B}\to\Fun_{\mathcal{A}^\omega}^\ex((\mathcal{B}^\omega)^\op,\Coh{A})
    \]
    induced by the $\mathcal{A}^\omega$-linear Yoneda embedding.
\end{theorem}
In the statment, the clause \enquote{$S_\text{fin}\subseteq S_\mathcal{B}$ is a finite generating set} forces certain projective classes generated by  $S_\text{fin}$ and $S_\mathcal{B}$ to agree, even if the $t$-structure geneated by them do not agree. 
It constitutes the main ingredient of the proof and, in applications, it ultimately reduces to produce a finite collection of compact connective generators (which are not required to compactly generate the given $t$-structure). 

A complementary problem is to represent exact $\mathcal{A}^\omega$-linear functors defined on coherent objects of $\mathcal{B}$.  
Here the appropriate representing objects are no longer coherent kernels, but rather compact kernels, and one must use the dual $A^\omega$-linear Yoneda embedding.
The obstruction is that $\Coh{B}$ is usually not generated, as a stable category, by compact objects alone.
We therefore introduce a descent condition designed to reduce the representability question to a setting where coherence and compactness coincide.  
Concretely, we define morphisms $h^\L: \mathcal{B}\to \mathcal{R}$ of \emph{universal descent} to a \emph{regular} tor-finite category $\mathcal{R}$.
Regularity is easy to explain, since it simply reduces to $\mathcal{R}^\omega=\Coh{R}$.
Universal descent, instead, is a technical assumption that produces an explicit finite-stage retract
of the identity, allowing one to descend many properties from the target to the source.
In particular, it allows us to prove the following finitary kernel theorem, which asserts that, for a quasi-proper functor $f^\L:\mathcal{A}\to\mathcal{B}$, the exact $\mathcal{A}^\omega$-linear functors on coherent objects with coherent values are represented by compact kernels. 
\begin{theorem}[Functors out of $\Coh{B}$]\label{theorem: functors out of coherents}
    Let $f^\L:\mathcal{A}\to\mathcal{B}$ be quasi-proper and assume that $\mathcal{B}$ admits a  morphism of universal descent $h^\L:\mathcal{B}\to\mathcal{R}$ to a regular category such that there exists a finite generating set $S_\text{fin}\subseteq S_\mathcal{R}$ for which $(\text{ev}_r)_{r\in S_\mathcal{R}}$ detects coconnective objects on $\mathcal{A}$.
    Then there exists an equivalence of categories
    \[
    (\mathcal{B}^\omega)^\op\to\Fun^\ex_{\mathcal{A}^\omega}(\Coh{B},\Coh{A})
    \]
    induced by the dual $\mathcal{A}^\omega$-linear Yoneda embedding.
\end{theorem}
These formal results have applications which the reader can find in \autoref{section: examples}.
The first  one is in the realm of $\mb{E}_\infty$-rings.
\begin{corollary}
    Let $f:A\to B$ be a map between connective $\mb{E}_\infty$-rings and assume that $\mathcal{B}$ is an almost compact $\mathcal{A}$-module.
    Then there are  equivalences of categories
    \[
    \text{PCoh}(B)\to \Fun^\ex_{\text{Perf}(A)}(\text{Perf}(B)^\op, \text{PCoh}(A)), \qquad
    \text{Coh}(B)\to \Fun^\ex_{\text{Perf}(A)}(\text{Perf}(B)^\op, \text{Coh}(A))
    \]
    induced by the $\text{Perf}(A)$-linear Yoneda embedding.
\end{corollary}
The second application is to schemes, and it provides the following generalization of \cite[Example 0.7]{neeman2025triangulatedcategoriessinglecompact}. 
To state it, let us denote by $\text{Perf}(-)\subseteq D^b_{\text{coh}}(-)\subseteq D^-_{\text{coh}}(-)$ the stable categories of perfect, bounded with coherent cohomology and bounded above with coherent cohomology complexes, respectively. Then:
\begin{corollary}
   Let $f:X\to Y$ be a proper map and assume that $Y$ is noetherian. 
   Then there are equivalences of categories
    \[
    D^-_{\text{coh}}(X)\to \Fun^\ex_{\text{Perf}(Y)}(\text{Perf}(X)^\op,D^-_{\text{coh}}(Y)), \qquad D^b_{\text{coh}}(X)\to \Fun^\ex_{\text{Perf}(Y)}(\text{Perf}(X)^\op,D^b_{\text{coh}}(Y)).
    \]
    induced by the  $\text{Perf}(Y)$-Yoneda embedding.  
    If furthermore $X$ admits a quasi-perfect regular alteration, then there is an equivalence of categories
    \[
    \text{Perf}(X)^\op\to\Fun_{\text{Perf}(Y)}^\ex(D^b_{\text{coh}}(X),D^b_{\text{coh}}(Y))
    \]
    induced by the  dual $\text{Perf}(Y)$-Yoneda embedding.
\end{corollary}
A more general result may be found in \autoref{section: examples}. 

We also prove a representability result for quasi-proper morphisms between arbitrary quasi-compact quasi-separeted schemes.
Still in the realm of algebraic geometry, we also prove a representability result for spectral algebraic spaces.
\begin{corollary}
    Let $f: X\to Y$ be a morphism of finite cohomological dimension of quasi-compact quasi-separated spectral algebraic spaces which is proper  and locally almost of finite presentation. 
    Assume also that $Y$ is noetherian. 
    Then there are equivalences of categories
    \[
    \text{PCoh}(X)\to \Fun^\ex_{\text{Perf}(Y)}(\text{Perf}(X)^\op,\text{PCoh}(Y)), \qquad \text{Coh}(X)\to \Fun^\ex_{\text{Perf}(Y)}(\text{Perf}(X)^\op,\text{Coh}(Y)).
    \]
    induced by the $\text{Perf}(Y)$-Yoneda embedding. 
\end{corollary}

\subsubsection*{Related work}
The present work was inspired by an informal question that appeared in \cite[Remark 0.9]{neeman2025triangulatedcategoriessinglecompact}.
It suggests an interesting parallelism between two representability results proved by Ben-Zvi, Nadler and Preygel in \cite{ben2017integral}, namely Theorem 1.1.3 (and more precisely Remark 1.1.6.(ii)) and Theorem 1.2.4, and two representability results proved by Neeman in the previous mentioned paper.  
In particular, \autoref{theorem: functors out of compacts} may be seen as a generalization of \cite[Theorem 0.4.(i)]{neeman2025triangulatedcategoriessinglecompact} and \autoref{theorem: functors out of coherents} as a partial (and actually incomplete) generalization of \cite[Theorem 0.4.(ii)]{neeman2025triangulatedcategoriessinglecompact}; both of our results appear then as a first step in understanding abstract relative finitary kernel theorems.

\subsubsection*{Linear overview} 
In the following we give an overview over the content.
In \autoref{section: rigid-categories} we recollect some  background on commutative algebra in $\Pr^\L_{\st}$ and  the construction of the linear Yoneda embedding as well as providing  a $2$-categorical formulation of Grothendieck-Neeman duality.  
In \autoref{section: Pseudo-Coherent and Coherent Objects} we introduce pseudo-coherent and coherent objects in right-complete presentable $t$-structures, study their formal properties and provide practical criteria to compute these subcategories.
We also compare them with Neeman pseudo-compact objects.
In \autoref{section: Neeman Dualities} we prove the main duality statements: we introduce tor-finite categories and quasi-proper
functors, prove \autoref{theorem: functors out of compacts}, develop universal descent, and prove \autoref{theorem: functors out of coherents}.
Finally, in \autoref{section: examples} we discuss applications to module categories, schemes and spectral algebraic spaces. 

\subsection{Acknowledgements}
This paper constitutes the author's master thesis. 
I \emph{thank} my advisors Denis-Charles Cisinski and Amnon Neeman for many helpful discussions and suggestions. 
I also thank Michele Riva for having read an early draft of these pages. 

\subsection{Notation and terminology}\label{subsection: conventions}
We will refer to \enquote{$\infty$-categories} simply as \enquote{categories}.
If we want to emphasise that a category has discrete mapping space, we will call it a \enquote{$1$-category}.

\begin{notation}
    Recall that a category $\mathcal{C}$ is called \emph{stable} if it is pointed, admits finite limits and finite colimits, and a commutative square is a pullback if and only if it is a pushout. 
    Equivalently, $\mathcal{C}$ is pointed and every morphism admits a fibre and a cofibre, which canonically agree after a shift.  
    Recall that a functor is \emph{exact} if it  preserves finite limits and colimits. 
\end{notation}
\begin{notation}
    Let $\mathcal{C}$ be a stable category. 
    A $t$-structure $\ts{C}$ on $\mathcal{C}$ will be graded homologically\footnote{Recall that to switch to the  cohomological convention it suffices to define $\mathcal{C}^{\geq n} = \mathcal{C}_{\leq -n}$  for every $n\in\Z$.}. 
    That is, we imagine $\mathcal{C}$ as linearized in the following way  $\dots \to\bullet_{n+1}\to\bullet_n\to\bullet_{n-1}\to\dots$. 
    We will think of objects in $\mathcal{C}_{\geq n}$ as existing  to the left of $n$, whereas  objects in $\mathcal{C}_{\leq n}$ will exist on the right. 
    In particular, we will call these objects \emph{$n$-connective} and \emph{$n$-coconnective}.
    The inclusions of the connective and coconnective objects admit a left and right adjoint respectively, that is 
    \[
    \tau_{\leq n}:\mathcal{C}\rightleftarrows\mathcal{C}_{\leq n}: i_{\leq n},
    \qquad 
    i_{\geq n}:\mathcal{C}_{\geq n}\rightleftarrows\mathcal{C}: \tau_{\geq n}.
    \]
    Let $\mathcal{C}^\heartsuit= \mathcal{C}_{\geq0}\cap\mathcal{C}_{\leq0}$ denote the \emph{heart of the $t$-structure}.
    We will denote by $\pi_n:\mathcal{C}\to\mathcal{C}^\heartsuit$ the functor $\tau_{\leq 0}\tau_{\geq0}[-n]$ for any $n\in\mb{Z}$, and refer to $\pi_n$ as the \emph{$n$-th homotopy group of the $t$-structure}. 
    We will also add a superscript (for example, $\pi_n^\mathcal{C}$) if more categories with $t$-structure are considered. 
    Finally, we will denote by 
    \[
    \mathcal{C}^-=\bigcup_{n>0}\mathcal{C}_{\geq -n},\qquad
    \mathcal{C}^+=\bigcup_{n>0}\mathcal{C}_{\leq n}, \qquad \mathcal{C}^b=\mathcal{C}^-\cap\mathcal{C}^+
    \]
    the full subcategories of $\mathcal{C}$ spanned by the \emph{eventually connective}, \emph{eventually coconnective} and \emph{bounded objects}.
\end{notation}
More notation will come in the following section. 

\section{Rigid geometry}\label{section: rigid-categories}

\subsection{Remarks on commutative algebra in $\Pr^\L$}
We begin with a list of observations. 
Most of the material appears in \cite[Section 4.8]{Lurie-HA}, in \cite[Section 4]{ben2010integral}, in \cite[Chapter I.1]{Gaitsgory2017ASI} and in \cite[Section 4]{hoyois2017highertracesnoncommutativemotives}.
\begin{remark}
    Recall that a category $\mathcal{C}$ is presentable if it is accessible (that is, there exists some  small regular cardinal $\kappa$ such that $\mathcal{C}\simeq\Ind_\kappa(\mathcal{D})$ is the ind-completion of  a small category $\mathcal{D}$) and admits small colimits.
    Let $\Pr^\text{L}$ denote the category of presentable categories and colimit-preserving functors (or equivalently, left adjoints). 
    Lurie constructed  a  symmetric monoidal structure on $\Pr^\text{L}$ characterized by the following universal property: given two presentable categories $\mathcal{C}$ and $\mathcal{D}$, their Lurie tensor product is a presentable category $\mathcal{C}\otimes\mathcal{D}$ with a functor $\mathcal{C}\times\mathcal{D}\to\mathcal{C}\otimes\mathcal{D}$ such that for every presentable category $\mathcal{E}$, precomposition with it induces an equivalence 
    \[
    \Fun^\text{L}(\mathcal{C}\otimes\mathcal{D},\mathcal{E})\to \Fun^{\text{L},\text{L}}(\mathcal{C}\times\mathcal{D},\mathcal{E}).
    \]
    Here the superscript $\L,\L$ denotes the full subcategory spanned by those functors $\mathcal{C}\times\mathcal{D}\to\mathcal{E}$ which preserve colimits separately in each variable.
    It follows from this universal property that the symmetric monoidal structure is also closed. To be precise,  for every pair of presentable categories $\mathcal{C}$ and $\mathcal{D}$, there exists a natural equivalence
    \[
    \Fun^\text{L}(\mathcal{C}\otimes\mathcal{D},\mathcal{E})\simeq \Fun^{\text{L},\text{L}}(\mathcal{C}\times\mathcal{D},\mathcal{E})\simeq \Fun^\text{L}(\mathcal{C},\Fun^\text{L}(\mathcal{D},\mathcal{E}))
    \]
    in $\mathcal{E}\in\Pr^\text{L}$,  so that $\Fun^\text{L}(-,-)$ exhibits the internal hom. 
    Since $\Fun^\text{L}(\Spc, \mathcal{C})\simeq\mathcal{C}$, it follows that the category of spaces $\Spc$ is the neutral element for the Lurie tensor product.
\end{remark}
We will need a stable version of this observation.
\begin{remark}
    Let $\Sp$ denote the category of spectra, defined as the colimit $\Sp=\colim (\Spc_*\xrightarrow{\susp}\Spc_*\xrightarrow{\susp}\dots)$ in $\Pr^\text{L}$, or dually, as the limit $\Sp\simeq\lim (\dots\xrightarrow{\Omega}\Spc_*\xrightarrow{\Omega}\Spc_*)$ in $\Pr^\text{R}$, the category of presentable categories and right adjoints. 
    Here $\Spc_*$ denotes the category of pointed spaces, whereas $\susp$ and $\Omega$ denote the suspension and loop functors, respectively.  
    Adjoining a point to a space and then \enquote{infinitely suspending it} produces a functor $\susp^\infty_+:\Spc\to\Sp$. 
    It follows that a presentable category $\mathcal{C}$ is stable if and only if the canonical map $\id_\mathcal{C}\otimes\susp^\infty_+:\mathcal{C}\to \mathcal{C}\otimes\Sp$ is an equivalence. 
    In particular, since $\Sp$ is a stable category,  the inverse of the  equivalence $\Sp\to\Sp\otimes\Sp$  makes $\Sp$ into a commutative algebra of $\Pr^\text{L}$. 
    It follows that the category of stable presentable categories and left adjoints can be realized as $ \Pr^\text{L}_\st \simeq \Mod_\Sp(\Pr^\text{L})$ and that taking spectrum objects $-\otimes\Sp:\Pr^\text{L}\to\Pr^\text{L}_\st$ defines a symmetric monoidal functor. 
\end{remark}

\begin{remark}
    Let $\Cat^\st$ denote the category of small stable categories and exact functors.
    Let $\Cat^\perf$ denote the full subcategory of $\Cat^\st$ spanned by the idempotent-complete stable categories. 
    Then the ind-completion  functor $\Ind: \Cat^\perf \to\Pr^\text{L}_\st$ induces an equivalence between $\Cat^\perf$ and the subcategory $\Pr^{\text{L},\omega}_\st$ of $\Pr^\text{L}_\st$ spanned by the compactly generated stable categories and by the left adjoint functors preserving compact objects. 
    The inverse is given by taking compact objects $(-)^\omega:\Pr^{\text{L},\omega}_\st\to\Cat^\perf$. 
    Now the symmetric monoidal structure on $\Pr^\text{L}_\st$ restricts to a symmetric monoidal structure on $\Pr^{\text{L},\omega}_\st$. 
    In particular, \cite[Proposition 4.4]{ben2010integral} implies that the above equivalence can enhanced to a symmetric monoidal equivalence. 
    More precisely, the tensor product of  $\mathcal{A},\mathcal{B}\in\Cat^\perf$ is given by a small idempotent-complete stable category $\mathcal{A}\otimes\mathcal{B}$ which is the recipient of the universal functor $\mathcal{A}\times\mathcal{B}\to\mathcal{C}$ which is exact in each variable. 
    In other words, there exists an equivalence
    \[
    \Fun^\ex(\mathcal{A}\otimes \otimes \mathcal{B}, \mathcal{C})\simeq \Fun^\ex(\mathcal{A}, \Fun^\ex(\mathcal{B}, \mathcal{C}))
    \]
    of categories for all $\mathcal{A},\mathcal{B},\mathcal{C}\in\Cat^\perf$. 
\end{remark}
We also remind the relevant features of commutative algebra objects and modules.
\begin{remark}
    It follows by construction that a commutative algebra object in $\CAlg(\Cat^\perf)$ can be identified with a small idempotent-complete stable category equipped with a symmetric monoidal structure whose tensor product is exact in each variable.
    Let $\mathcal{A}\in\CAlg(\Cat^\perf)$ be a commutative algebra object.
    Let  $\Mod_\mathcal{A}(\Cat^\perf)$ be the category of modules over $\mathcal{A}$ in $\Cat^\perf$ and regard it as a symmetric monoidal category with the induced relative tensor product $\otimes_\mathcal{A}$. This symmetric monoidal structure is also closed, and the internal hom may be identified with the pullback
    \[\begin{tikzcd}[cramped,sep=scriptsize]
	{\Fun^\ex_\mathcal{A}(-,-)} & {\Fun^\ex(-,-)} \\
	{\Fun_\mathcal{A}(-,-)} & {\Fun(-,-)}
	\arrow[from=1-1, to=1-2]
	\arrow[hook, from=1-1, to=2-1]
	\arrow[hook, from=1-2, to=2-2]
	\arrow[from=2-1, to=2-2]
    \end{tikzcd}\]
    where $\Fun_\mathcal{A}(-,-)$ denotes the category of $\mathcal{A}$-linear functors. Objects of $\Fun^\ex_\mathcal{A}(-,-)$ will be called \emph{$\mathcal{A}$-linear functors}.
    Similarly, let $\Mod_{\Ind(\mathcal{A})}(\Pr^\text{L}_\st)$ be the category of modules over $\Ind(\mathcal{A})$ in $\Pr^\text{L}_\st$, and equip it with the relative tensor product $\otimes_{\Ind(\mathcal{A})}$. Again, this symmetric monoidal structure is closed, with internal hom $\Fun^\text{L}_{\Ind(\mathcal{A})}(-,-)$ defined similarly. Objects of this category will be called $\Ind(\mathcal{A})$\emph{-linear functors}.
\end{remark} 
Since the ind-completion is symmetric monoidal, the definitions of linear functors immediately imply the following. 
\begin{lemma}\label{lemma: from L to ex}
    Let $\mathcal{A}\in\CAlg(\Cat^\perf)$ be a commutative algebra object and let $\mathcal{M}\in\Mod_\mathcal{A}(\Cat^\perf)$ be an $\mathcal{A}$-module.
    Then precomposition with the Yoneda embedding $\yo:\mathcal{M}\into\Ind(\mathcal{M})$ induces an equivalence 
    \[
    \yo^*:\Fun^\text{L}_{\Ind(\mathcal{A})} (\Ind(\mathcal{M}), \mathcal{N})\to\Fun^\ex_\mathcal{A} (\mathcal{M}, \mathcal{N})
    \]
    of $\mathcal{A}$-modules for every $\mathcal{N}\in\Mod_{\Ind(\mathcal{A})}(\Pr^\text{L}_\st)$.
\end{lemma}

\subsection{The linear Yoneda embedding} Following \cite[Section 4.2]{hoyois2017highertracesnoncommutativemotives} we now recall the construction of the enriched Yoneda embedding. 
\begin{construction}\label{construction: linear Yoneda}
    Let $\mathcal{A}\in\CAlg(\Cat^\perf)$ be a commutative algebra  and let $\mathcal{M}\in\Mod_\mathcal{A}(\Cat^\perf)$ be an $\mathcal{A}$-module. Let $m\in\mathcal{M}$ and consider the action functor $-\otimes m:\mathcal{A}\to\mathcal{M}$.
     By definition this functor preserves finite colimits and hence it admits an ind-right adjoint $\hom_\mathcal{M}(m,-):\mathcal{M}\to\Ind(\mathcal{A})$. 
    The functoriality of these cosntruction produces a functor
    \[
    \yo_\mathcal{A}: \mathcal{M}\to\Fun^\ex(\mathcal{M}^\op,\Ind(\mathcal{A}))
    \]
    which is given on objects by $m\mapsto\hom_\mathcal{M}(-,m)$. 
\end{construction}
In general the above functor is not fully-faithful; dualizability is exactly what is needed.
\begin{definition}
    A commutative algebra  in $\Cat^\perf$  is called \emph{rigid} if every object in it is dualizable. 
    We let  $\CAlg^\rig(\Cat^\perf)$ denote the full subcategory of $\CAlg(\Cat^\perf)$ spanned by the rigid commutative algebras.
\end{definition}

We recall the following folkore result. 
\begin{lemma}\label{lemma: action restricts}
    Let $\mathcal{A}\in\CAlg^\rig(\Cat^\perf)$ be a rigid commutative algebra.
    \begin{enumerate*}
        \item Taking duals defines a symmetric monoidal equivalence $(-)^\vee:\mathcal{A}\to\mathcal{A}^\op$.
        \item Let $\mathcal{M}\in\Mod_{\Ind(\mathcal{A})}(\Pr^{\L}_\st)$ be a presentable $\Ind(\mathcal{A})$-module. Then the action of $\mathcal{A}$ on $\mathcal{M}$ restricts to the full subcategory $\mathcal{M}^\omega\subseteq\mathcal{M}$ of compact objects.
    \end{enumerate*}
\end{lemma}
\begin{proof}
    Consider point $(1)$. Since $\mathcal{A}$ is rigid, every object $a\in A$ admits a (chosen) dual $a^\vee$ together with evaluation and coevaluation maps $\text{ev}_a: a\otimes a^\vee\to \mathbbm{1}_\mathcal{A}$ and $\text{coev}_a:\mathbbm{1}_\mathcal{A}\to a^\vee\otimes a$ satisfying the triangle identities. 
    Define the functor $(-)^\vee: \mathcal{A} \to \mathcal{A}^{\op}$ by taking duals. Since for dualizable objects there are canonical equivalences $(a\otimes b)^\vee \simeq b^\vee\otimes a^\vee\simeq a^\vee\otimes b^\vee$, the functor $(-)^\vee$ is symmetric monoidal.
    Moreover, $(-)^\vee$ is adjoint to itself: the unit $\eta_a:a\to (a^\vee)^\vee$ and counit $\epsilon_a:(a^\vee)^\vee\to a$ are the usual double-dual maps. 
    Dualizability implies $\eta_a$ (equivalently $\epsilon_a$) is an equivalence for every $a$, hence $(-)^\vee$ is an equivalence of symmetric monoidal categories.

    Consider point $(2)$.
    Fix $a\in \mathcal{A}$. The action endofunctor $a\otimes - : \mathcal{M}\to \mathcal{M}$ preserves colimits (because the $\Ind(\mathcal{A})$-action does), hence is a left adjoint. 
    Since $a$ is dualizable in $\mathcal{A}$, the functor $a\otimes -$ has right adjoint $a^\vee\otimes -$. 
    In particular, $a^\vee\otimes -$ preserves filtered colimits (indeed it preserves all colimits), so $a\otimes -$ preserves compact objects.
    In particular, the $\mathcal{A}$-action restricts to $\mathcal{M}^\omega\subseteq \mathcal{M}$.
\end{proof}

The following result has been proved many times in the literature. 
We recall the proof for completeness.
\begin{lemma}\label{lemma: compactly generated categories are dualizable}
    Let $\mathcal{A}\in\CAlg^\rig(\Cat^\perf)$ be a rigid  algebra  and let $\mathcal{M}\in\Mod_{\Ind(\mathcal{A})}(\Pr^{\text{L},\omega}_\st)$ be an $\Ind(\mathcal{A})$-compactly generated module. 
    Then $\mathcal{M}$  is dualizable, and the dual can be identified with $\mathcal{M}^\vee \simeq\Ind((\mathcal{M}^\omega)^\op)$.
\end{lemma}
\begin{proof}
    Set $\mathcal{M}^\vee = \Ind((\mathcal{M}^\omega)^{\op})$.
    First of all, point $(2)$ of \autoref{lemma: action restricts} implies that the $\Ind(\mathcal{A})$-action   on $\mathcal{M}$ restricts to an $\mathcal{A}$-action on the small idempotent-complete stable category $\mathcal{M}^\omega$.
    Hence $(\mathcal{M}^\omega)^{\op}$ is canonically an $\mathcal{A}$-module as well, and therefore $\Ind((\mathcal{M}^\omega)^{\op})$ inherits a canonical $\Ind(\mathcal{A})$-module structure.
    In particular, $\mathcal{M}^\vee\in \Mod_{\Ind(\mathcal{A})}(\Pr^{\L,\omega}_{\st})$.
    To prove dualizability of $\mathcal{M}$ with dual $\mathcal{M}^\vee$, it suffices to construct, for every module $\mathcal{B}\in \Mod_{\Ind(\mathcal{A})}(\Pr^{\L}_{\st})$, a natural equivalence
    \[
    \mathcal{M}^\vee \otimes_{\Ind(\mathcal{A})} \mathcal{B}\simeq \Fun^{\L}_{\Ind(\mathcal{A})}(\mathcal{M},\mathcal{B}),
    \]
    where $\otimes_{\Ind(\mathcal{A})}$ denotes the relative Lurie tensor product.
    Because $\Mod_{\Ind(\mathcal{A})}(\Pr^{\mathrm L}_{\st})$ is closed symmetric monoidal, for every $\mathcal{B}$ and $\mathcal{B}$ there is a natural equivalence
    \begin{equation}\label{eq:closed}
    \Fun^{\mathrm L}_{\Ind(\mathcal{A})}(\mathcal{M}^\vee\otimes_{\Ind(\mathcal{A})} \mathcal{B},\mathcal{C})
    \simeq
    \Fun^{\mathrm L}_{\Ind(\mathcal{A})}(\mathcal{M}^\vee,\; \Fun^{\mathrm L}_{\Ind(\mathcal{A})}(\mathcal{B},\mathcal{C})).
    \end{equation}
    Since taking opposites is compatible with $\Ind(\mathcal{A})$-linearity, there is a natural equivalence
    \[
    \Fun^{\mathrm L}_{\Ind(\mathcal{A})}(\mathcal{B},\mathcal{C})\simeq \Fun^{\mathrm L}_{\Ind(\mathcal{A})}(\mathcal{C}^{\op},\mathcal{B}^{\op})^{\op}.
    \]
    Plugging this into \eqref{eq:closed} and using the closed structure again (and currying-uncurrying) gives a chain of natural equivalences
    \begin{align*}
    \Fun^{\mathrm L}_{\Ind(\mathcal{A})}(\mathcal{M}^\vee\otimes_{\Ind(\mathcal{A})} \mathcal{B}, \mathcal{C})
    &\simeq
    \Fun^{\mathrm L}_{\Ind(\mathcal{A})}(\mathcal{M}^\vee, \Fun^{\mathrm L}_{\Ind(\mathcal{A})}(\mathcal{C}^{\op},\mathcal{B}^{\op})) \\
    &\simeq
    \Fun^{\mathrm L}_{\Ind(\mathcal{A})}(\mathcal{C}^{\op},\Fun^{\mathrm L}_{\Ind(\mathcal{A})}(\mathcal{M}^\vee,\mathcal{B}^{\op})) \\
    &\simeq
    \Fun^{\mathrm L}_{\Ind(\mathcal{A})}(\Fun^{\mathrm L}_{\Ind(\mathcal{A})}(\mathcal{M}^\vee,\mathcal{B}^{\op})^{\op}, \mathcal{C}).
    \end{align*}
    Yoneda then implies that $\mathcal{M}^\vee\otimes_{\Ind(\mathcal{A})} \mathcal{B}\simeq \Fun^{\mathrm L}_{\Ind(\mathcal{A})}(M^\vee,\mathcal{B}^{\op})^{\op}$.
    By using \autoref{lemma: from L to ex} twice it follows that
    \[
    \Fun^\L_{\Ind(\mathcal{A})}(\mathcal{M}^\vee, \mathcal{B}^\op)^\op 
    \simeq
    \Fun^\ex_\mathcal{A}((\mathcal{M}^\omega)^\op,\mathcal{B}^\op)^\op 
    \simeq 
    \Fun^\ex_\mathcal{A}((\mathcal{M}^\omega),\mathcal{B}) 
    \simeq
    \Fun^\L_{\Ind(\mathcal{A})}(\mathcal{M},\mathcal{B})
    \]
    thus concluding the proof.
\end{proof}

\begin{corollary}\label{corollary: equivalence ind and exact functors}
    Let $\mathcal{A}\in\CAlg^\rig(\Cat^\perf)$ be a rigid  algebra  and let $\mathcal{M},\mathcal{N}\in\Mod_{\mathcal{A}}(\Cat^\perf)$ be two $\mathcal{A}$-modules. Then there exists a natural equivalence
    \[
    \Ind(\mathcal{M}^\op \otimes_\mathcal{A}\mathcal{N}) \simeq\Fun^\ex_\mathcal{A}(\mathcal{M}, \Ind(\mathcal{N}))
    \]
    in $\Mod_{\Ind(\mathcal{A})}(\Pr^{\text{L},\omega}_\st)$. 
\end{corollary}
\begin{proof}
    There are equivalences
    \[\begin{aligned}
    \Ind(\mathcal{M}^{\op}\otimes_\mathcal{A} \mathcal{N})
    &\simeq\Ind(\mathcal{M}^{\op})\otimes_{\Ind(\mathcal{A})} \Ind(\mathcal{N})\\
    &\simeq \Ind(\mathcal{M})^\vee\otimes_{\Ind(\mathcal{A})}\Ind(\mathcal{N})\\
    &\simeq\Fun^\L_{\Ind(\mathcal{A})}(\Ind(\mathcal{M}),\Ind(\mathcal{N}))\\
    &\simeq\Fun^{\ex}_\mathcal{A}(\mathcal{M},\Ind(\mathcal{N})).
    \end{aligned}\]
    Indeed the first one follows since  $\Ind:\Mod_\mathcal{A}(\Cat^\perf)\to \Mod_{\Ind(\mathcal{A})}(\Pr^{\L,\omega}_{\st})$ is symmetric monoidal, the second one by \autoref{lemma: compactly generated categories are dualizable} since $\mathcal{A}$ is rigid, the third one by definition of internal hom and the last one by \autoref{lemma: from L to ex}.
\end{proof}

\begin{remark}
    Let $\mathcal{A}\in\CAlg^\rig(\Cat^\perf)$ be a rigid  algebra and let $\mathcal{M}\in\Mod_{\mathcal{A}}(\Cat^\perf)$ be an $\mathcal{A}$-module. 
    Then \autoref{corollary: equivalence ind and exact functors} implies that the functor $\yo_\mathcal{A}:\mathcal{M}\into\Fun^\ex_\mathcal{A}(\mathcal{M}^\op,\Ind(\mathcal{A}))$ of \autoref{construction: linear Yoneda} is fully-faithful and exhibits the target as the ind-completion of the source. 
\end{remark}
This motivates the following.
\begin{definition}
    Let $\mathcal{A}\in\CAlg^\rig(\Cat^\perf)$ be a rigid  algebra and let $\mathcal{M}\in\Mod_{\mathcal{A}}(\Cat^\perf)$ be an $\mathcal{A}$-module. 
    We will refer to the functor  $\yo_\mathcal{A}:\mathcal{M}\into\Fun^\ex_\mathcal{A}(\mathcal{M}^\op,\Ind(\mathcal{A}))$ as the \emph{$\mathcal{A}$-linear Yoneda embedding of $\mathcal{M}$}.
\end{definition}

\subsection{Grothendieck duality} 
A central feature of rigidly-compactly generated categories and of functors between them is that Grothendieck–Neeman duality admits a clean axiomatic formulation in this setting.
One such formulation is due to Balmer, Dell’Ambrogio, and Sanders \cite[Theorem 1.7]{balmer2016grothendieck}. 
We will now recast this duality in the language of $(\infty,2)$-categories, which we will henceforth refer to simply as $2$-categories.

\begin{remark}
    To work $2$-categorically with presentable stable categories, one uses the standard $\Cat$-enriched enhancement: for $\mathcal{A},\mathcal{B}\in \Pr^\L_{\st}$ the mapping category is
    \[
    \underline{\Hom}_{2\Pr^\L_{\st}}(\mathcal{A},\mathcal{B}) = \Fun^\L(\mathcal{A},\mathcal{B}),
    \]
    and composition is induced by composition of functors.
    More intrinsically (in the sense of complete Segal objects, see \cite[Chapter I.1, 6.1.8]{Gaitsgory2017ASI}), one may package this as an $2$-category $2\Pr^\L_{\st}$ whose underlying $1$-category $\iota_12\Pr^\L_{\st}$ is the usual symmetric monoidal $1$-category of presentable stable categories and colimit-preserving functors, and whose hom-$1$-categories are $\Fun^\L(\mathcal{A},\mathcal{B})$. 
    In particular, adjunctions internal in this $2$-category coincide with ordinary adjunctions of functors.
    Finally, the Lurie tensor product of presentable categories upgrades (functorially) to this $2$-categorical enhancement, yielding a symmetric monoidal $2$-category structure on $2\Pr^\text{L}_{\st}$.
\end{remark}

Every $2$-category supports the notion of \enquote{internal left adjoint}, bootstrapped from the usual notion of adjunction of categories. 
We now specialize it to modules in $2\Pr^\text{L}_\st$.
\begin{definition}
    Let $\mathcal{B}\in\CAlg(\Pr^\text{L}_\st)$ be a commutative algebra. 
    Let $f^\text{L}:\mathcal{M}\to\mathcal{N}$ be a map in $\Mod_\mathcal{B}(\Pr^\text{L}_\st)$. We will say that $f^\text{L}$ is an \emph{internal left adjoint} if its right adjoint $f^\text{R}$ preserves colimits and the square
    \[\begin{tikzcd}[cramped]
	{\mathcal{B}\otimes\mathcal{M}} & {\mathcal{B}\otimes\mathcal{N}} \\
	{\mathcal{M}} & {\mathcal{N}}
	\arrow["{{\id_{\mathcal{B}}\otimes f^\text{L}}}", from=1-1, to=1-2]
	\arrow["{\text{act}_\mathcal{M}}"', from=1-1, to=2-1]
	\arrow["{\text{act}_\mathcal{N}}", from=1-2, to=2-2]
	\arrow["{{f^\text{L}}}"', from=2-1, to=2-2]
    \end{tikzcd}\]
    is horizontally right adjointable.
\end{definition}
In other terms, the functor $f^\text{L}:\mathcal{M}\to\mathcal{N}$ is an internal left adjoint if and only if the right adjoint $f^\text{R}$ preserves colimits and the canonical right-lax $\mathcal{B}$-module structure $\text{act}_\mathcal{M}(b, f^\text{R}(n))\to f^\text{R}(\text{act}_{\mathcal{M}}(b,n))$ is an equivalence for every $b\in\mathcal{B}$ and $n\in\mathcal{N}$.

We now procede towards Grothendieck–Neeman duality.
\begin{remark}
    Let $\mathcal{B}\in\CAlg(\Pr^\text{L}_\st)$ be a commutative algebra and let $f^\text{L}:\mathcal{M}\to\mathcal{N}$ be an internal left adjoint. Then $f^\text{R}$ is a $\mathcal{B}$-module map. Furthermore, $f^\text{R}$ admits a right adjoint $f^{(1)}$ which has the structure of a right-lax $\mathcal{B}$-module structure. This lax structure can be constructed explicitly as 
    \[
    \text{act}_\mathcal{N}(b, f^{(1)}(m))
    \to f^{(1)}f^\text{R}(\text{act}_\mathcal{N}(b, f^{(1)}(m)))
    \simeq f^{(1)}( \text{act}_\mathcal{M}(b, f^\text{R}f^{(1)}(m)))
    \to f^{(1)}(\text{act}_\mathcal{M}(b, m))
    \]
    for $b\in\mathcal{B}$ and $m\in\mathcal{M}$. Here the first map and last map are  given by the unit and counit of $f^\text{R}\dashv f^{(1)}$, respectively.
\end{remark}

\begin{remark}
    Let $\mathcal{M}\in\CAlg(\Pr^{\text{L}}_\st)$ be a commutative algebra and let $\mathcal{N}\in\CAlg_{\mathcal{M}/}(\Pr^{\text{L}}_\st)$ be an algebra object under $\mathcal{M}$. 
    Under the equivalence $\CAlg_{\mathcal{M}/}(\Pr^{\text{L}}_\st)\simeq \CAlg(\Mod_\mathcal{M}(\Pr^\text{L}_\st))$ the category $\mathcal{N}$ inherits the structure of a $\mathcal{M}$-module. In particular, if $f^\text{L}:\mathcal{M}\to\mathcal{N}$ is the structure map, then the $\mathcal{M}$-module structure $\mathcal{M}\times\mathcal{N}\to\mathcal{N}$ on $\mathcal{N}$ is given by $(m,n)\mapsto f^\text{L}(m)\otimes_\mathcal{N}n$. 
\end{remark}

\begin{remark}
    Let $f^\text{L}:\mathcal{M}\to\mathcal{N}$ be a map of commutative algebras in $\Pr^\text{L}_\st$ which is an internal left adjoint  of $\mathcal{M}$-modules. Then the functor $f^\text{R}$ is colimit preserving and $\mathcal{M}$-linear and the right adjoint $f^{(1)}$ has a right lax $\mathcal{M}$-module structure. 
    Consider now the equivalences
    \[\begin{aligned}
    \Hom_\mathcal{M}(f^\text{R}f^{(1)}(x)\otimes_\mathcal{M} y, x\otimes_\mathcal{M} y) 
    &\simeq \Hom_{\mathcal{M}}(f^\text{R}( f^{(1)}(x)\otimes_\mathcal{M} f^\text{L}(y)), x\otimes_\mathcal{M} y)\\
    &\simeq \Hom_{\mathcal{N}}(f^{(1)}(x)\otimes_\mathcal{M} f^\text{L}(y), f^{(1)}(x\otimes_\mathcal{N} y))
    \end{aligned}\]
    where the first one follows by linearity of $f^\text{R}$ and the second one by adjunction. If $\epsilon: f^\text{R}f^{(1)}\to\id_\mathcal{M}$ is the counit of $f^\text{R}\dashv f^{(1)}$, then the object $\epsilon_x\otimes_\mathcal{M}\id_y$ produces a canonical morphism
    \[
    f^{(1)}(-)\otimes_\mathcal{N} f^\text{L}(-)\to f^{(1)}(-\otimes_\mathcal{M}-). 
    \]  
    The Grothendieck-Neeman duality is about understanding this morphism.
\end{remark}

\begin{definition}
     Let $f^\text{L}:\mathcal{M}\to\mathcal{N}$ be a map of commutative algebras in $\Pr^\text{L}_\st$ which is an internal left adjoint  of $\mathcal{M}$-modules. We will say that \emph{Grothendieck–Neeman duality holds for} $f^\text{L}$ if the canonical map $f^{(1)}(-)\otimes_\mathcal{N} f^\text{L}(-)\to f^{(1)}(-\otimes_\mathcal{M}-)$ is an equivalence. 
\end{definition}
We have then the following $2$-categorification of Grothendieck-Neeman duality. 
\begin{theorem}\label{theorem: Grothendieck Duality}
    Let $f^\text{L}:\mathcal{M}\to\mathcal{N}$ be a map of commutative algebras in $\Pr^\text{L}_\st$ which is an internal left adjoint  of $\mathcal{M}$-modules. Then the right adjoint $f^\text{R}$ is an internal left adjoint of $\mathcal{M}$-modules if and only if Grothendieck–Neeman duality holds. 
\end{theorem}
\begin{proof}
    Assume that $f^\text{R}$ is an internal left adjoint. Then $f^{(1)}$ is $\mathcal{M}$-linear. Since the lax structure on $f^{(1)}$ is given by the canonical map $f^{(1)}(-)\otimes_\mathcal{N} f^\text{L}(-)\to f^{(1)}(-\otimes_\mathcal{M}-)$ the claim follows. 
    Conversely, if this map is an equivalence, to show that $f^\text{R}$ is an internal left adjoint it suffices to show that $f^{(1)}$ preserves colimits. 
    Since $f^\text{L}$ is colimit preserving, evaluating at the unit in the first argument shows that $f^{(1)}$ is colimit preserving.  
\end{proof}

The following observations are a translation of this language for rigidly-compactly generated categories.
\begin{remark}\label{remark: ramzi corollary 3.8}
    Let $\mathcal{B}\in\CAlg(\Pr^\text{L}_\st)$ be a commutative algebra and assume that $\mathcal{B}$ is generated by the dualizable objects.
    Then by \cite[Corollary 3.8]{ramzi2024locallyrigidinftycategories} every right-lax $\mathcal{B}$-linear functor is $\mathcal{B}$-linear. 
    Therefore a $\mathcal{B}$-module map  $f^\text{L}:\mathcal{M}\to\mathcal{N}$ is an internal left adjoint if and only if its right adjoint is colimit preserving.
\end{remark}
\begin{example}\label{example: internal left adjoint for rigidly-compactly generated}
    Let $f^\text{L}:\mathcal{M}\to\mathcal{N}$ be a morphism of commutative algebras in $\Pr^\text{L}_\st$. 
    Assume that $\mathcal{M}$ and $\mathcal{N}$ are rigidly-compactly generated. 
    Then $f^\text{L}$, being symmetric monoidal, preserves dualizable objects, thus it preserves compact objects. 
    But every functor preserving compact objects between compactly-generated categories has a filtered colimit preserving right adjoint. 
    In the stable case this implies that the right adjoint preserves colimits. 
    Hence \autoref{remark: ramzi corollary 3.8} implies that $f^\text{L}$ is an internal left adjoint.
    To spell it out: every functor $f^\text{L}:\mathcal{M}\to\mathcal{N}$ in $\CAlg^\rig(\Pr^{\text{L},\omega}_\st)$ is an internal left adjoint in $\Mod_\mathcal{M}(\Pr^\text{L}_\st)$.

    Grothendieck–Neeman duality asks the functor $f^\R$ to be an internal left adjoint of $\mathcal{M}$-modules. 
    Once again, \autoref{remark: ramzi corollary 3.8} implies that $f^\R$ is an internal left adjoint if and only if the right adjoint $f^{(1)}$ preserves colimits (and thus it has a further right adjoint). 
    Since $f^{(1)}$ is  exact, this is equivalent to  $f^{(1)}$ preserving filtered colimits, which in turn is equivalent to $f^\R$ preserving compact objects. 
    This is Neeman's criterion in \cite[Theorem 1.7]{balmer2016grothendieck}.
    Notice that if this is the case, the functor $f^\L$ admits a left adjoint. 
\end{example}
\begin{definition}
    Let $p^\L:\mathcal{M}\to\mathcal{N}$ be a map of commutative algebras in $\Pr^\text{L}_\st$ which is an internal left adjoint of $\mathcal{M}$-modules. The object $\omega_{p}=p^{(1)}(\mb{1}_\mathcal{M})$ is called the \emph{dualizing object of} $p^\L:\mathcal{M}\to\mathcal{N}$. 
\end{definition} 

\section{Pseudo-coherent and coherent objects}\label{section: Pseudo-Coherent and Coherent Objects}

\subsection{Right-complete presentable $t$-structures}
Recall that a $t$-structure on a presentable stable category is accessible if the connective aisle is presentable, and it is compatible with filtered colimits if the coconnective aisle is closed under filtered colimits.
\begin{definition}\label{definition: geometric t-structure}
    Let $\mathcal{C}\in\Pr^\text{L}_\st$ be a  presentable stable category equipped with a $t$-structure  $\ts{C}$. We will say that the $t$-structure $\ts{C}$ is \emph{right-complete presentable} if it is accessible, compatible with filtered colimits and right complete.
    We will furthermore say that the $t$-structure is \emph{excellent} if it is also left-complete.
\end{definition}

\begin{lemma}\label{lemma: NP 6.1.1}
    Let $\mathcal{C}\in\Pr^\text{L}_\st$ be a  presentable stable category with accessible $t$-structure. Then the following are equivalent.
    \begin{enumerate*}
        \item The category $\mathcal{C}_{\leq0}$ is compatible with filtered colimits.
        \item The functor $i_{\leq0} : \mathcal{C}_{\leq0}\to \mathcal{C}$ preserves filtered colimits.
        \item The functor $i_{\leq0}\tau_{\leq0}: \mathcal{C}\to\mathcal{C}$ preserves filtered colimits.
        \item The functor $i_{\geq0}\tau_{\geq0}:\mathcal{C}\to\mathcal{C}$ preserves filtered colimits.
        \item The functor $\tau_{\geq0}: \mathcal{C}\to\mathcal{C}_{\geq0}$ preserves filtered colimits.
    \end{enumerate*}
    These equivalent conditions imply the following.
    \begin{enumerate*}
        \item[(6)] The functor $i_{\geq0} : \mathcal{C}_{\geq0}\to\mathcal{C}$ preserves compact objects.
        \item[(7)] The functor $\tau_{\leq0}: \mathcal{C}\to\mathcal{C}_{\leq0}$ preserves compact objects.
    \end{enumerate*}
    Furthermore,
    \begin{enumerate**}
        \item If $\mathcal{C}$ is compactly-generated, then so is $\mathcal{C}_{\leq0}$ (with compact objects being retracts of objects of the form $\tau_{\leq0}c$ for  $c\in\mathcal{C}^\omega$). In this case, the above conditions are equivalent to $\tau_{\leq0}$ preserving compact objects.
        \item If $\mathcal{C}$ and $\mathcal{C}_{\geq0}$ are compactly-generated, then the above conditions are equivalent to $i_{\geq0}$ preserving compact objects.
    \end{enumerate**}
\end{lemma}
\begin{proof}
    The equivalence $(1)\iff(2)$ is by definition. 
    The implication $(2)\Rightarrow(3)$ follows since $\tau_{\leq0}$ is a left adjoint whereas the converse $(3)\Rightarrow(2)$ follows since the $t$-structure is accessible. A similar argument (coupled with the fact that also $\mathcal{C}_{\geq0}$ is presentable) shows the equivalence $(4)\iff(5)$. 
    To conclude,  note that $(2)\iff(5)$ follows by considering the cofibre sequence of functors $\tau_{\geq0}\to id\to\tau_{\leq-1}\simeq \susp^{-1}\tau_{\leq0}$ on $\mathcal{C}$. 
    Assume now the equivalent conditions $(1)$-$(5)$. 
    The implication $(2)\Rightarrow(7)$ follows since a right adjoint preserving filtered colimits has a left adjoint preserving compact objects. The same argument imples  $(5)\Rightarrow(6)$. 
    Consider the \enquote{furthermore} part. Point (a) follows since a left Bousfield localization of a compactly generated category is again compactly generated. For point (b), if $\mathcal{C}$ and $\mathcal{C}_{\geq0}$ are compactly generated, then $(6)$ implies $(5)$ since a left adjoint that preserves compact objects between compactly-generated categories has a right adjoint that preserves filtered colimits.
\end{proof}

We also need to incorporate a monoidal structure.
\begin{definition}
    Let $\mathcal{C}\in\CAlg(\Pr^\text{L}_\st)$ be a commutative algebra  with a $t$-structure $\ts{C}$. 
    We will say that the $t$-structure is \emph{compatible with $\otimes$} if the inclusion functor   $\mathcal{C}_{\geq0}\into\mathcal{C}$ is symemtric monoidal.
\end{definition}
In other terms, $t$-structure is compatible with $\otimes$ if the monoidal unit is connective and connective objects are closed under tensor products.
\begin{remark}\label{remark: lurie construction of t}
    Let $\mathcal{C}\in\Pr^\text{L}_\st$ be a  presentable stable category. 
    Then \cite[Proposition 1.4.4.11]{Lurie-HA} implies that any small collection of objects $S\subseteq\mathcal{C}$  determines an accessible $t$-structure on $\mathcal{C}$, which is minimal among  $t$-structures such that  $S\subseteq\mathcal{C}_{\geq0}$.
    Assume now that $\mathcal{C}\in\CAlg(\Pr^\text{L}_\st)$ is a commutative algebra object. 
    If $S$ is closed under tensor products and contains the monoidal unit, then the resulting $t$-structure is also compatible with $\otimes$. More generally, to check that the $t$-structure is compatible with $\otimes$ it suffices for the monoidal unit to be in $\mathcal{C}_{\geq0}$ and $S\otimes\mathcal{C}_{\geq0}\subseteq \mathcal{C}_{\geq0}$. 
\end{remark}
\begin{corollary}
    Let $\mathcal{C}\in\CAlg(\Pr^\text{L}_\st)$. Then the monoidal unit $\mb{1}_\mathcal{C}$ determines an accessible $t$-structure whose connective aisle is the smallest full subcategory of $\mathcal{C}$ closed under colimits, retracts and tensor products  containing $\mb{1}_\mathcal{C}$. 
\end{corollary}
In general, the $t$-structure determined by a collection of objects is  neither compatible with filtered colimits nor right complete. 
We now look for conditions that fix these issues.
\begin{lemma}\label{lemma: conditions of t-structure}
    Let $\mathcal{C}\in\Pr^\text{L}_\st$ be a  presentable stable category and let $S\subseteq\mathcal{C}^\omega$ be a small collection of compact objects that generates, in the sense that for every nonzero object $x\in\mathcal{C}$ there exists $s\in S$ and $n\in\mb{Z}$ for which the abelian group $\pi_0\Hom_\mathcal{C}(\susp^n s,x)$ is nonzero\footnote{Notice that this forces $\mathcal{C}$ to be compactly geneated.}.
    Let $\mathcal{C}_{\leq0}$ be  the full subcategory of $\mathcal{C}$ spanned by those objects $x\in\mathcal{C}$ for which the groups $\pi_0\Hom_\mathcal{C}(\susp^ns,x)$ vanish for all $s\in S$ and $n>0$.
    \begin{enumerate*}
        \item Then the $t$-structure generated by $S$ is compatible with filtered colimits.
        \item Then the $t$-structure is right-complete. 
        \item The connective half $\mathcal{C}_{\geq0}$ is compactly generated. More precisely, let $\mathcal{G}$ denote  the smallest full subcategory of $\mathcal{C}$ which contains the generators and is  closed under finite colimits and extensions. Then the inclusion $\mathcal{G}\into\mathcal{C}$ extends to an  equivalence of categories $\Ind(\mathcal{G})\to\mathcal{C}_{\geq0}$.
    \end{enumerate*} 
    In other words, the $t$-structure is right-complete presentable and compactly generated.
\end{lemma}
\begin{proof}
    This is \cite[Proposition C.6.3.1]{Lurie-SAG}.
\end{proof}
The previous result is particularly useful if coupled with the following.  
\begin{lemma}\label{lemma: thick, compact are eventually connective}
    Let $\mathcal{C}\in\Pr^{\L,\omega}_\st$ be a compactly generated stable category with a $t$-structure. 
    Let $S\subseteq \mathcal{C}_{\geq0}\cap\mathcal{C}^\omega$ be a set of connective compact objects such that $\mathcal{C}$ is the smallest  localizing subcategory containing $S$. 
    Then every compact object $\mathcal{C}^\omega\subseteq\mathcal{C}^-$ is eventually connective.
\end{lemma}
\begin{proof}
    Since $S\subseteq\mathcal{C}^-$ and since $\mathcal{C}^-$ is thick, the thick closure of $S$ lies in $\mathcal{C}^-$. 
    Since $\mathcal{C}$ is the smallest  localizing subcategory containing $S$, the  thick closure of $S$ coincides with $\mathcal{C}^\omega$. Thus $\mathcal{C}^\omega\subseteq\mathcal{C}^-$ .
\end{proof}

Another assumption that will be useful later is left completeness. Recall that a $t$-structure $\ts{C}$ on a stable category $\mathcal{C}$  is \emph{left complete} if the canonical functor $\mathcal{C}\to\lim_{n\in\N}(\mathcal{C}_{\leq n},\tau_{\leq n})$ is an equivalence of categories. 
\begin{remark}
    Let $\mathcal{C}$ be a stable category equipped with a  left complete $t$-structure $\ts{C}$. Let $x\in\cap_{n\in\N}\mathcal{C}_{\geq n}$. Then $\tau_{\leq n}x\simeq0$ for all $n\in\N$ so that $x\simeq 0$ by left completeness.
    In other words, the $t$-structure is \emph{left separated}.
    The converse holds if $\mathcal{C}$ has countable products and the connective half is closed under such products; see \cite[Proposition 1.2.1.19]{Lurie-HA}.
\end{remark}

We summarize the situation.
\begin{remark}\label{remark: summary 1}
    Let $\mathcal{C}\in\CAlg(\Pr^{\L,\omega}_\st)$ be a commutative algebra and let $S\subseteq \mathcal{C}$ be a small collection of compact objects. Let $\ts{C}$ be the accessible $t$-structure generated by $S$. Convinient assumptions are then the followings.
    \begin{enumerate*}
        \item The collection $S$ generates. 
        \item The $t$-structure is left complete (and hence left separated).
        \item The $t$-structure is compatible with $\otimes$. 
    \end{enumerate*}
    Then $(1)$ coupled with \autoref{lemma: conditions of t-structure} implies that the $t$-structure generated by $S$ is right-complete presentable and compactly generated. 
    It does not say anything about left completeness and compatibility with $\otimes$.
    Nonetheless, point $(3)$  is equivalent to requiring the monoidal unit to be in $\mathcal{C}_{\geq0}$ and $S\otimes\mathcal{C}_{\geq0}\subseteq \mathcal{C}_{\geq0}$. 
    In particular, if $S$ is closed under tensor products and contains the monoidal unit, then point $(3)$ is satisfied.
    Another convinient assumption is:
    \begin{enumerate*}
        \item[(4)] Compact objects $\mathcal{C}^\omega\subseteq\mathcal{C}^-$ are eventually connective.
    \end{enumerate*}
    Then \autoref{lemma: thick, compact are eventually connective} gives an easy condition for $(4)$ to be satisfied, namely that $\mathcal{C}$ is the smallest localizing subcategory containing $S$.
\end{remark}
It may seem hard to construct such $S$ in applications. 
We give an idea here (which we will study more in detail in \autoref{section: examples}).
\begin{example}\label{example: moral}
    Let $X$ be a geometric object of some kind such that  there is a category of \enquote{quasi-coherent sheaves} $\QCoh(X)\in\CAlg^\rig(\Pr^{\L,\omega}_\st)$ and assume that there exists  a \enquote{standard $t$-structure} compatible with $\otimes$.
    Let $\Perf(X)$ denote the full subcategory on compact objects.  
    Then the essentially small category $\Perf(X)_{\geq0} = \Perf(X)\cap\QCoh(X)_{\geq0}$ is symmetric monoidal. 
    If $X$ satisfies \enquote{connective perfect generation}, meaning that the connective half is generated by $\Perf(X)_{\geq0}$, then by picking a small skeleton\footnote{
        In general the small skeleton $S$ does not inherit  a symmetric monoidal structure from $\Perf(X)_{\geq0}$ so that it is not clear that the $t$-structure generated by $S$ is compatible with $\otimes$. 
        Fortunately, the $t$-structure generated via \autoref{lemma: conditions of t-structure} depends only on the skeleton of the generators, so that the $t$-structure generated by $S$ and the one by $\Perf(X)_{\geq0}$ coincide. 
        In practice, it is useful to close $S$ under $\otimes$ and unit to an essentially small collection $S^\otimes$. 
        The symmetric monoidal category $S^\otimes\into\Perf(X)_{\geq0}$ is not a skeleton anymore, but it has the same skeleton of $S$, so that the resulting $t$-structures are the same.
    } 
    $S\subseteq\Perf(X)_{\geq0}$ and by applying \autoref{lemma: conditions of t-structure}, the resulting $t$-structure is the same one we started from.  
    It follows formally that the $t$-structure is right-complete presentable and compactly generated (and compatible with $\otimes$). This leaves us to check left completeness.   
    For $(4)$, it is often the case that $\QCoh(X)$ admits a single \enquote{perfect} generator which is eventually connective, thus shifting it and including it in $S$ allows us to deduce that compact objects are eventually connective.
\end{example}
Another case where $S$ is easy to construct is when the monoidal unit is a compact generator.  
\begin{example}
    Let $\mathcal{C}\in\CAlg^\rig(\Pr^{\L,\omega}_\st)$ and assume that the monoidal unit generates. 
    Then the $t$-structure generated by $S=\{\mb{1}\}\subseteq \mathcal{C}^\omega$ produces a right-complete presentable $t$-structure compatible with $\otimes$.  Furthermore,  $\mathcal{C}^\omega\subseteq\mathcal{C}^-$.
    Again, left completeness is not automatic.
\end{example}
\begin{example}
    The left completeness in the previous example really needs to be checked. Consider for example  the complex $K$-theory spectrum $KU$ and consider the category of modules $\Mod_{KU}(\Sp)$. 
    Then the connective half $t$-structure generated by $KU$ is  all of $\Mod_{KU}(\Sp)$, since $KU$ a $2$-periodic spectrum. This $t$-structure is obviously not left complete. 
\end{example}

\subsection{$t$-exact functors}
We add to the previous discussion the input of functors.
\begin{definition}
    Let $f:\mathcal{B}\to\mathcal{C}$ be an exact functor between stable categories with $t$-structures $\ts{B}$ and $\ts{C}$. We will say that:
    \begin{enumerate*}
    \item The functor $f$ is \emph{right $t$-exact up to a shift $r\in\N$} if it sends $0$-connective objects to $(-r)$-connective objects. That is, $f(\mathcal{B}_{\geq0})\subseteq \mathcal{C}_{\geq -r}$.
    \item The functor $f$ is \emph{left $t$-exact up to a shift $l\in\N$} if it sends $0$-coconnective objects to $l$-coconnective objects. That is, $f(\mathcal{B}_{\leq0})\subseteq \mathcal{B}_{\leq l}$.
    \end{enumerate*}
    In the case where $r,l$ are zero, we will say that $f$ is right $t$-exact, left $t$-exact, respectively.
\end{definition}
The next observation explains the discrepancies between the indexes. 
\begin{remark}
    Let $f^\text{L}:\mathcal{B}\rightleftarrows\mathcal{C}:f^\text{R}$ be adjoint functors between stable categories equipped with $t$-structures $\ts{B}$ and $\ts{C}$.
    Let $r\in\N$. Then a compuation with hom-spaces shows that $f^\text{L}$ is right $t$-exact up to a shift $r$ if and only if $f^\text{R}$ is left $t$-exact up to a shift $r$. 
\end{remark}

\begin{remark}\label{remark: t-exact functors on the heart}
    Let $f^\text{L}:\mathcal{B}\rightleftarrows\mathcal{C}:f^\text{R}$ be adjoint functors between stable categories equipped with $t$-structures $\ts{B}$ and $\ts{C}$.
    Let $i_\mathcal{B}:\mathcal{B}^\heartsuit\into \mathcal{B}$ and $i_\mathcal{C}:\mathcal{C}^\heartsuit\into \mathcal{C}$ be the  inclusions of the hearts.
    We denote by $f^\text{L}_\heartsuit$ and $f^\text{R}_\heartsuit$ the compositions
    \[
    f^\text{L}_\heartsuit: \mathcal{B}^\heartsuit\xrightarrow{i_\mathcal{B}}\mathcal{B}\xrightarrow{f^\text{L}}\mathcal{C}\xrightarrow{\pi_0^\mathcal{C}}\mathcal{C}^\heartsuit,\qquad 
    f^\text{R}_\heartsuit:
    \mathcal{C}^\heartsuit\xrightarrow{i_\mathcal{C}}\mathcal{C}\xrightarrow{f^\text{R}}\mathcal{B}\xrightarrow{\pi_0^\mathcal{B}}\mathcal{B}^\heartsuit
    \]
    and call them the \emph{induced functors of} $f^\text{L}\dashv f^\text{R}$. 
\end{remark}
We conclude with a technical result. 

\begin{lemma}\label{lemma: properties of functors on the heart}
    Let $f^\text{L}:\mathcal{B}\rightleftarrows\mathcal{C}:f^\text{R}$ be adjoint functors between stable categories equipped with $t$-structures $\ts{B}$ and $\ts{C}$.
    Assume that  $f^\text{L}$ is right $t$-exact up to a shift $r\in\N$.
    \begin{enumerate*}
        \item Given $x\in\mathcal{B}_{\geq0}$ and $y\in\mathcal{C}_{\leq0}$, there are a natural isomorphisms
        \[
        f^\text{L}_\heartsuit(\pi_0^\mathcal{B}x)\to \pi_{-r}^\mathcal{C}(f^\text{L}(x)),\qquad
        \pi_{-r}^\mathcal{B}(f^\text{R}(y))\to f^\text{R}_\heartsuit(\pi_0^\mathcal{C}(y))
        \]
        in $\mathcal{C}^\heartsuit$ and $\mathcal{B}^\heartsuit$, respectively.
        \item  The adjunction $f^\text{L}\dashv f^\text{R}$ determines an adjunction $f^\text{L}_\heartsuit\dashv f^\text{R}_\heartsuit$. 
    \end{enumerate*}
\end{lemma}
\begin{proof}
    Consider $(1)$. By duality it suffices to prove the first isomorphism.
    Consider $x\in\mathcal{B}_{\geq0}$ and consider the cofibre sequence $\tau_{\geq1}^\mathcal{B}x\to x\to \tau_{\leq0}^\mathcal{B}x$ in $\mathcal{B}$. Since $x$ is $0$-connective, we can identify  $\tau_{\leq0}^\mathcal{B}x$ with $\pi_0^\mathcal{B}(x)$. 
    Apply $f^\text{L}$ to obtain the cofibre sequence $f^\text{L}(\tau_{\geq1}^\mathcal{B}x)\to f^\text{L}(x)\to f^\text{L}(\pi_0^\mathcal{B}x)$ in $\mathcal{C}$.
    Since $f^\text{L}$ is right $t$-exact up to a shift $r$, it sends connective objects to $(-r)$-connective objects. In particular, the natural map $f^\text{L}(\pi_0^\mathcal{B}x)\to\pi_{-r}^\mathcal{C}f^\text{L}(x)$ has to be an equivalence.
    The proof of point $(2)$ is a chain of equivalences. 
    Take $x\in\mathcal{B}^\heartsuit$ and $y\in\mathcal{C}^\heartsuit$ and compute
    \[\begin{aligned}
    \Hom_{\mathcal{C}^\heartsuit}(f^\text{L}_\heartsuit(x),y)
    &= \Hom_{\mathcal{C}^\heartsuit}(\pi_0^\mathcal{C}f^\text{L}i_\mathcal{B}(x),y)\\
    &\iso\Hom_{\mathcal{C}}(f^\text{L}i_\mathcal{B}(x),i_\mathcal{C}(y))\\
    &\iso\Hom_{\mathcal{B}}(i_\mathcal{B}(x),f^\text{R}i_\mathcal{C}(y))\\
    &\iso\Hom_{\mathcal{B}^\heartsuit}(x, \pi_0^\mathcal{B}f^\text{R}i_\mathcal{C}(y)) 
    = \Hom_{\mathcal{B}^\heartsuit}(x,f^\text{R}_\heartsuit(y)).
    \end{aligned}\]
    Here the first equality is the definition of $f^\text{L}_\heartsuit$, the second by our assumption (since $f^\text{L}(\mathcal{B}_{\geq0})\subseteq\mathcal{C}_{\geq0}$ and since $i_\mathcal{C}(y)$ is $0$-coconnective, the space of maps $f^\text{L}i_\mathcal{B}(x)\to i_\mathcal{C}(y)$ is discrete by using the standard argument with the cofibre sequence given by truncating), the third one by adjunction $f^\text{L}\dashv f^\text{R}$, and the  fourth one again by assumption (since $f^\text{R}(\mathcal{C}_{\leq0})\subseteq\mathcal{B}_{\leq0}$ and since $i_\mathcal{B}(x)\in\mathcal{B}_{\geq0}$ is $0$-connective, the space of maps $i_\mathcal{B}(x)\to f^\text{R}i_\mathcal{C}(y)$ is discrete, again by the same argument above). 
    The last equality is the definition of $f^\text{R}_\heartsuit$.
\end{proof}

\subsection{Finite objects in a $t$-structure}
Let $\mathcal{C}\in\Pr^\text{L}$ be  a presentable category. Recall that an object $x\in\mathcal{C}$ is \emph{almost compact} if for every integer $n\geq0$ the truncation $\tau_{\leq n}x$ is a compact object of $\tau_{\leq n}\mathcal{C}$.
Here $\tau_{\leq n}\mathcal{C}$ denotes the  full subcategory of $\mathcal{C}$ spanned by the $n$-truncated objects, that is, by those objects $x\in\mathcal{C}$ such that the mapping space $\Hom_\mathcal{C}(y,x)$ is $n$-truncated for all $y\in\mathcal{C}$.

\begin{definition}\label{definition: pseudo-coherent}
    Let $\mathcal{C}\in\Pr^\text{L}_\st$ be a presentable stable category equipped with an accessible $t$-structure. 
    We will say that an object $x\in\mathcal{C}$ is \emph{pseudo-coherent} if it is $n$-connective $x\in\mathcal{C}_{\geq n}$ for some $n\in\mb{Z}$ and almost compact as an object of $\mathcal{C}_{\geq n}$. We will furthermore say that $x$ is \emph{coherent} if pseudo-coherent and coconnective, that is $x\in\mathcal{C}_{\leq m}$ for some $m\in\mb{Z}$.
    
    We let $\Coh{C}\subseteq\PCoh{C}$ denote the full subcategories of $\mathcal{C}$ spanned by the coherent and pseudo-coherent objects, respectively. 
\end{definition}
\begin{remark}
    The nomenclature we have chosen here comes from algebraic geometry. 
    Pseudo-coherent objects were first introduced by Illusie in \cite{berthelot1971generalites} via a slightly different definition. 
    There, pseudo-coherent complexes on a scheme  are defined as complexes which, locally, are (quasi-)isomorphic to bounded above complexes which admit projective resolutions by finitely generated projectives.
    If the scheme appears to be noetherian, then pseudo-coherent complexes are precisely the bounded-above complexes whose cohomology is coherent. 
    A similar statement can be deduced from \autoref{proposition: pseudo coherent and pi_n} and from \autoref{section: examples}.
\end{remark}

\begin{remark}
    Let $\mathcal{C}$ be a stable category equipped with a $t$-structure $\ts{C}$. Then an object $x\in\mathcal{C}_{\geq n}$ is $k$-truncated  as an object of $\mathcal{C}_{\geq n}$ if and only if  $x\in\mathcal{C}_{\leq n+k}$.
    In particular, $\tau_{\leq k}\mathcal{C}_{\geq n} =\mathcal{C}_{\geq n}\cap \mathcal{C}_{\leq n+ k}$.
\end{remark}

The first result of this section shows some basic properties of these subcategories.
\begin{lemma}\label{lemma: stupid properties of Pcoh and coh}
     Let $\mathcal{C}\in\Pr^\text{L}_\st$ be a presentable stable category equipped with an accessible $t$-structure. 
     Then $\Coh{C}\subseteq\PCoh{C}$ are idempotent-complete stable subcategories of $\mathcal{C}$. 
\end{lemma}
\begin{proof}
    To show that $\Coh{C}\subseteq\PCoh{C}$ are stable subcategories of $\mathcal{C}$ it suffices to prove that they are closed under finite colimits and desuspensions. This follows immediately since almost compact objects are closed under finite colimits and desuspensions (because compact objects are), together with the stability of connective and bounded objects.  
    A similar argument shows also that $\Coh{C}\subseteq\PCoh{C}$ are closed under retracts, since compact objects are. It follows that $\Coh{C}$ and $\PCoh{C}$ are idempotent-complete, since a stable thick subcategory of a presentable category is.
\end{proof}
\begin{warning}
    In the assumptions of \autoref{lemma: stupid properties of Pcoh and coh} we cannot conclude that $\Coh{C}$ and $\PCoh{C}$ are essentially small. There might exist objects $x\in\mathcal{C}^-$ such that $\tau_{\leq n}x\simeq 0$ for every $n\in\mb{Z}$. 
    To fix this issue it is sufficient to assume that the $t$-structure is left complete.
    Indeed, since the $t$-structure is accessible, each $\mathcal{C}_{\geq n}$ is presentable so that each truncation $\tau_{\leq m}\mathcal{C}_{\geq n}$ is presentable. In particular, its subcategory of compact objects is small. 
    Left completeness then identifies pseudo-coherent objects with the category of towers $x_0\to x_1\to\dots$ where each  $x_m\in(\tau_{\leq m}\mathcal{C}_{\geq n})^\omega$ and such that the maps are Postkinov compatible, and this category is  small.
\end{warning}
\begin{corollary}\label{corollary: smallnes of coherent/pseudo}
    Let $\mathcal{C}\in\Pr^\text{L}_\st$ be a presentable stable category equipped with an accessible left complete $t$-structure. Then $\Coh{C},\PCoh{C}\in\Cat^\perf$.  
\end{corollary}
We can also analyse the interaction between (pseudo)-coherent objects and compact objects. 
\begin{remark}\label{remark: compact are pseudo-coherent}
    Let $\mathcal{C}\in\Pr^\text{L}_\st$ be a presentable stable category equipped with an accessible and compatible with filtered colimits $t$-structure. 
    \begin{enumerate*}
        \item If $\mathcal{C}^\omega\subseteq\mathcal{C}^-$, then $\mathcal{C}^\omega\subseteq\PCoh{C}$. 
        \item If $\mathcal{C}^\omega\subseteq\mathcal{C}^b$, then $\mathcal{C}^\omega\subseteq\Coh{C}$. 
    \end{enumerate*} 
    Indeed, being the $t$-structure compatible with filtered colimits, \autoref{lemma: NP 6.1.1} implies that each truncation functor $\tau_{\leq n}:\mathcal{C}\to\mathcal{C}_{\leq n}$ preserves compact objects, leaving us to prove that compact objects are bounded below for point $(1)$ and bounded for point $(2)$.  
\end{remark}

We now analyze the behaviour of pseudo-coherent and coherent objects under the formation of geometric realization.

\begin{lemma}\label{lemma: 0-connected pseudo-coheren are closed under the formation of geometric realization}
    Let $\mathcal{C}\in\Pr^\text{L}_\st$  be a presentable stable category equipped with an accessible $t$-structure. Then the full subcategory $\PCoh{C}\cap \mathcal{C}_{\geq0}\subseteq\mathcal{C}$ is closed under the formation of geometric realization of simplicial objects.
\end{lemma}
\begin{proof} 
    Let $x_\bullet$ be a simplicial object of $\PCoh{C}\cap \mathcal{C}_{\geq0}$ such that  each $x_k$ is a compact object of $\tau_{\leq n}(\PCoh{C}\cap \mathcal{C}_{\geq0})$.
    To show that the geometric realization $|x_\bullet|$ can be computed in $\PCoh{C}\cap \mathcal{C}_{\geq0}$ and that it is preserved by the inclusion $\PCoh{C}\cap \mathcal{C}_{\geq0}\subseteq\mathcal{C}$ it suffices to notice that $\tau_{\leq n}(\PCoh{C}\cap \mathcal{C}_{\geq0})$ is equivalent to a $(n+1)$-category, so that the equivalence
    \[
    |x_\bullet|\simeq \colim_{[k]\in\Delta^\op_{\leq n+1}}x_k
    \]
    exhibits the geometric realization $|x_\bullet|$ as a finite colimit, which is then preserved by the inclusion $\PCoh{C}\cap \mathcal{C}_{\geq0}\subseteq\mathcal{C}$.
\end{proof}

Our next goal is to prove a converse of this result. 
\begin{remark}[\protect{\cite[Lemma 1.2]{neeman2025triangulatedcategoriessinglecompact}}]\label{remark: Neeman lemma 1.2}
    Let $\mathcal{C}$ be a stable category  equipped with a $t$–structure. If $x\in\mathcal{C}^-$ and $\pi_l(x)=0$ for $l<i$, then $x\in\mathcal{C}_{\geq i}$.
    Indeed, since $x$ belongs to $\mathcal{C}^-$, it belongs to some $\mathcal{C}_{\geq -n}$ for some $n>0$. Thus the canonical map $\tau_{\geq-n}x\to x$ is an equivalence. 
    Now, since $\pi_l(x)=0$ for every $l<i$, the cofibre sequence $\tau_{\geq 1+l}\to \tau_{\geq l}x\to \susp^{-l}\pi_l(x)$  implies that, as long as $l<i$, the map $\tau_{\geq 1+l}\to \tau_{\geq l}x$ is  an equivalence. The chain of equivalences  $\tau_{\geq i}x\to \tau_{\geq i-1}x\to \dots \to \tau_{\geq -n}x\to x$ shows then the claim. 
\end{remark}
\begin{remark}[\protect{\cite[Lemma 1.3]{neeman2025triangulatedcategoriessinglecompact}}]\label{remark: Neeman lemma 1.3}
    Let $\mathcal{C}$ be a stable category and equipped with a left separated $t$–structure. 
    Then: 
    \begin{enumerate*}
        \item Every object $x \in \mathcal{C}^-$, with $\pi_n(x)\iso 0$ for all $n\in\mb{Z}$, must vanish.
        \item If $f : x\to y$ is a morphism in $\mathcal{C}^-$ such that $\pi_n(f)$ is an isomorphism for every $n\in\mb{Z}$, then $f$ is an equivalence.
    \end{enumerate*}
    Indeed,  $x$ belongs to $\cap_l \mathcal{C}_{\geq l}$ by \autoref{remark: Neeman lemma 1.2} and left separatedness implies $x\simeq 0$. 
    Point $(2)$  follows by applying $(1)$ to the cofibre of $f$.
\end{remark}

Recall that a full subcategory $S\subseteq\mathcal{C}_{\geq0}$ is a generating  subcategory if for every object $x\in\mathcal{C}_{\geq0}$ there exists a morphism $\oplus_{i\in I}c_i\to x$ which induces an epimorphism on $\pi_0$ and each $c_i$  is $S$.

\begin{lemma}\label{lemma: assumption S}
    Let $\mc{C}\in\Pr^{\L,\omega}_\st$ be a compactly generated stable category equipped with a right complete and left separated presentable $t$-structure. 
    Assume also that:
    \begin{enumerate}
    \item[(S)] There exists a generating subcategory $S\subseteq \mathcal{C}_{\geq0}$ which is closed under finite colimits and extensions and  consists of pseudo-coherent connective objects of $\mathcal{C}$.
    \end{enumerate} 
    Let $x\in\PCoh{C}\cap\mc{C}_{\geq0}$ be a connective pseudo-coherent object.
    Then $x$ can be obtained as the geometric realization of a simplicial object $x_\bullet$ such that each $x_n$ is in $S$.
\end{lemma}
\begin{proof}
    By the (prestable) Dold-Kan correspondence \cite[Theorem C.1.3.1]{Lurie-SAG}, it suffices to construct a sequential diagram
    \[
    D(0)\xrightarrow{f_1}D(1)\to \dots
    \]
    mapping to $x$ such that the cofibre of each map $f_n$ is a $n$-fold suspension of an object of $S$ and such that the canonical map $\colim_{n\in\N}D(n)\to x$ is an equivalence. 
    Equivalently, it suffices to produce a cofibre sequence 
    \[
    D(n)\to x \to \susp^n c(n)
    \]
    for each $n\in\N$ such that $D(n)$ is in $S$ and $c(n)$ is connective (and pseudo-coherent). 
    The argument goes by induction. 
    Set, by convention, $D(0)$ to be the zero, so that the cofibre of $0\to x$ is $x$. 
    The inductive hypothesis constructs a diagram   
    \[
    D(n)\to  x\to \susp^n c(n)
    \]
    with $D(n)$ in $S$  and $c(n)$ connective and pseudo-coherent (thus $\susp^nc(n)\in\mathcal{C}_{\geq n}$). 
    In particular $\susp^nc(n)$ is in  $\PCoh{C}\cap\mathcal{C}_{\geq n}$.
    Thus $c(n)\in\PCoh{C}\cap\mathcal{C}_{\geq 0}$ is connective and pseudo-coherent. 
    Since $S$ is a generating subcategory of $\mathcal{C}_{\geq0}$, there exists a morphism $\oplus_{i\in I} p_i\to c(n)$ with $p_i\in S$ for every $i\in I$ inducing a $\pi_0$-epimorphism. 
    Since $\pi_0c(n)$ is compact, being $c(n)$ pseudo-coherent, there exists a finite set $I_0\subseteq I$ such that $\oplus_{i\in I_0} p_i\to c(n)$ induces the required $\pi_0$-epimorphism. 
    Set $P=\oplus_{i\in I_0} p_i$. 
    Since $S$ is closed under finite colimits, it follows that $P\in S$. 
    Define $D(n+1)$ via the pullback square
    \[\begin{tikzcd}[cramped]
	{D(n+1)} & {\susp^nP} \\
	x & {\susp^nc(n)}
	\arrow[from=1-1, to=1-2]
	\arrow[from=1-1, to=2-1]
	\arrow[from=1-2, to=2-2]
	\arrow[from=2-1, to=2-2]
    \end{tikzcd}\]
    Then the cofibre sequence $D(n)\to D(n+1)\to \susp^nP$ exhibits $D(n+1)$ in $S$, being $S$ closed under extensions. 
    Furthermore, the equivalence $\cofib(D(n+1)\to x)\simeq \susp^{n+1}\fib(P\to c(n))$ shows that the map $D(n+1)\to x$ induces an  isomorphism on $\pi_i$ for $0\leq i \leq n-2$. It follows by \autoref{remark: Neeman lemma 1.3} that the canonical map $\colim_{n\in N} D(n)\to x$ is an equivalence. Since, by construction, the cofiber of each map $D(n)\to D(n+1)$ is an $n$-fold suspension of  an object of $S$, the claim follows.
\end{proof}
\begin{remark}\label{remark: assumption S}
    Let $\mathcal{C}$ be as in the assumptions. Denote by $\mathcal{C}^\omega_{\geq0}$  the full subcategory of the connective half $\mathcal{C}_{\geq0}$ spanned by the compact objects. 
    Then \autoref{lemma: NP 6.1.1} implies that $\mathcal{C}^\omega\cap\mathcal{C}_{\geq0} = \mathcal{C}^\omega_{\geq0}$. 
    It follows that if $\mathcal{C}_{\geq0}$ is compactly generated, then  assumption (S)  is satisfied.
\end{remark}

There is another notion of finiteness that we should explore.
\begin{definition}\label{definition: tor dimension}
    Let $\mathcal{C}$ be a stable symmetric monoidal category equipped with a $t$-structure compatible with $\otimes$. 
    Let $a\leq b\in\mb{Z}$.
    We will say that an object $x\in\mathcal{C}$ has \emph{tor-amplitude in $[a,b]$} if $x\otimes-:\mathcal{C}\to\mathcal{C}$ is right $t$-exact up to $-a$ and left $t$-exact up to $b$. 
    We will say that an object has \emph{finite tor-dimension} if it has tor-amplitude in $[a,b]$ for some $a\leq b\in\mb{Z}$.
    We let  $\Tor{C}$ denote the full subcategory of $\mathcal{C}$ spanned by the objects of finite tor-dimension.
\end{definition}
In other words, $x\in\mathcal{C}$ has tor-amplitude in $[a,b]$ if tensoring with $x$ restricts to functors $x\otimes-:\mathcal{C}_{\geq0}\to\mathcal{C}_{\geq a}$ and  $x\otimes-:\mathcal{C}_{\leq 0}\to\mathcal{C}_{\leq b}$. 


The next observation connects our definition of finite tor-amplitude with the classical one. 
\begin{remark}\label{remark: tor amplitude classical}
    Let $\mathcal{C}$ be a stable symmetric monoidal category equipped with a $t$-structure compatible with $\otimes$. Let $x\in\mathcal{C}$ and let $n\in\mb{Z}$. Consider the following properties.
    \begin{enumerate*}
        \item For every $y\in\mathcal{C}_{\geq0}$ the object $x\otimes y\in\mathcal{C}_{\geq n}$ is $n$-connective. 
        Equivalently, $x\in\mathcal{C}_{\geq n}$.
        \item For every $m\in\N$ and for every $y\in\mathcal{C}_{\geq0}\cap\mathcal{C}_{\leq m}$ the object $x\otimes y\in\mathcal{C}_{\geq n}$ is $n$-connective.
        \item For every $y\in\mathcal{C}^\heartsuit$ the object $x\otimes y\in\mathcal{C}_{\geq n}$ is $n$-connective.
    \end{enumerate*}
    Clearly $(1)\Rightarrow(2)\Rightarrow(3)$. 
    The implication  $(3)\Rightarrow(2)$ follows by induction.
    Indeed, if $y\in\mathcal{C}_{\geq0}\cap\mathcal{C}_{\leq m}$, then tensoring the cofibre sequence $\susp^k\pi_k(y)\to\tau_{\leq k}y\to\tau_{\leq k-1}y $ with $x$ produces a cofibre sequence $x\otimes\susp^k\pi_k(y) \to x\otimes\tau_{\leq k}y\to x\otimes\tau_{\leq k-1}y$. 
    The first term belongs to $\mathcal{C}_{\geq n+k}\subseteq\mathcal{C}_{\geq n}$, and the claim follows since the connective half of a $t$-structure is closed under extensions. 
    The implication $(2)\Rightarrow(1)$ is not true in general.
    It holds for instance if the monoidal unit is (connective and) eventually coconnective.
    More generally, it holds if $x\in\mathcal{C}_{\geq a}$ is eventually connective: indeed it follows by choosing $m$ with $a+m+1\geq n$ and using the triangle $\tau_{\geq m+1}\to y\to\tau_{\leq m}y$ with $y\in\mathcal{C}_{\geq0}$.
\end{remark}
\begin{lemma}\label{lemma: tor amplitude vs classical}
    Let $\mathcal{C}$ be a stable symmetric monoidal category equipped with a $t$-structure compatible with $\otimes$. Let $x\in\mathcal{C}$.
    \begin{enumerate*}
    \item If $x$ has tor-amplitude in $[a,b]$, then $\pi_n(x\otimes y)\simeq 0$ for every $y\in\mathcal{C}^\heartsuit$ and $n\notin [a,b]$.
    \end{enumerate*}
    Assume that $\mathcal{C}\in\CAlg(\Pr^\L_\st)$ is presentable, that $\otimes$ preserves colimits in each variable and that the $t$-structure is right-complete presentable. Assume also the monoidal unit is eventually coconnective (or that $x\in\mathcal{C}_{\geq A}$).
    \begin{enumerate*}
    \item[(2)]  If $x$ is such that $\pi_n(x\otimes y)\simeq 0$ for every $y\in\mathcal{C}^\heartsuit$ and $n\notin [a,b]$, then $x$ has tor-amplitude in $[a,b]$.
    \end{enumerate*} 
\end{lemma}
\begin{proof}
    Assume that $x$ has tor-amplitude in $[a,b]$  and let $y\in\mathcal{C}^\heartsuit$ be discrete. Then $x\otimes y\in\mathcal{C}_{\geq A}\cap\mathcal{C}_{\leq b}$ so that $\pi_n(x\otimes y)\simeq 0$ for $n\notin [a,b]$. This proves $(1)$.

    Consider $(2)$. Since the monoidal unit is connective and eventually coconnective (or since $x\in\mathcal{C}_{\geq A}$), \autoref{remark: tor amplitude classical} coupled with the assumption of homotopy groups implies that tensoring with $x$ produces a functor $\mathcal{C}_{\geq0}\to\mathcal{C}_{\geq a}$. To show that $x\otimes-$ is left $t$-exact up to $b$, consider $y\in\mathcal{C}_{\leq0}$. 
    Since the $t$-structure is right complete, the canonical map $\colim_{n\in\N}\tau_{\geq-n}y\to y$ is an equivalence. 
    Since $\tau_{\leq0}$ is a left adjoint, the above map produces an equivalence $\colim_{n\in\N}\tau_{\leq0}\tau_{\geq-n}y\to y$. 
    Since $\mathcal{C}_{\leq0}$ is comatible with filtered colimits it suffices to show that $x\otimes y\in\mathcal{C}_{\leq b}$ for $y\in\mathcal{C}_{\geq -n}\cap\mathcal{C}_{\leq0}$. 
    This can be proved by induction, similarly to \autoref{remark: tor amplitude classical}, on the cofibre sequence $\tau_{\geq-k}y\to \tau_{\geq -k-1}y\to\susp^{-k-1}\pi_{-k-1}y$ for $0\leq k\leq n$. 
\end{proof}

We now study the categorical properties of objects of finite tor-amplitude.
\begin{lemma}\label{lemma: stupid properties of tor}
    Let $\mathcal{C}$ be a stable symmetric monoidal category equipped with a $t$-structure. Then $\Tor{C}$ is a stable subcategory, closed under retracts. In other words, it is a thick subcategory. 
    If $\mathcal{C}$ is also presentable, then $\Tor{C}$ is idempotent-complete.
\end{lemma}
\begin{proof}
    It suffices to show that $\Tor{C}$ is pointed (which is trivial), stable under cofibres and desuspensions. Closure under desuspensions follows since $\otimes$ is exact in each variable separately.
    Closure under cofibres follows since the connective half is closed under cofibres, and the same happens to coconnective halves, after a possible shift.
    Finally,  since the connective and coconnective half of a $t$-structure are closed under retracts, the claim follows for $\Tor{C}$. 
    The last claim is a general fact of stable thick subcategories of presentable stable categories.
\end{proof}

Objects of finite tor-amplitude control many \enquote{geometric} properties.
Let $\mathcal{C}\in\CAlg(\Pr^{\text{L},\omega}_\st)$ be a commutative algebra with a right-complete presentable $t$-structure compatible with $\otimes$. 
Then we may define $\Perf(\mathcal{C})$ to be the intersection $\PCoh{C}\cap\Tor{C}$, and call its objects \emph{perfect objects}. 
The next result is an instance of \enquote{perfect object controlling the geoemtry}.  
\begin{lemma}\label{lemma: tor and compact}
    Let $\mathcal{C}\in\CAlg(\Pr^{\text{L},\omega}_\st)$ be a commutative algebra with a right-complete presentable $t$-structure compatible with $\otimes$. 
    Assume that $\mathcal{C}^\omega \subseteq \Tor{C}$. 
    \begin{enumerate*}
    \item Then $\mathcal{C}^\omega\subseteq\PCoh{C}$.
    \item If the unit of the monoidal structure is eventually coconnective, then $\mathcal{C}^\omega\subseteq\Coh{C}$.
    \end{enumerate*} 
\end{lemma}
\begin{proof} 
    This is an application of \autoref{remark: compact are pseudo-coherent}. Indeed point $(1)$ follows since the monoidal unit is connective so that every compact object is eventually connective, whereas point $(2)$ follows since the unit $\mb{1}_\mathcal{C}$ is eventually coconnective (and the usual argument by induction shows that every compact object is eventually coconnective). 
\end{proof}

We can now analyze the behaviour of pseudo-coherent and coherent objects under tensor products. 
\begin{lemma}\label{lemma: pcoh and coh are Cc submodules}
    Let $\mathcal{C}\in\CAlg^\rig(\Pr^{\text{L},\omega}_\st)$ be a rigid commutative algebra equipped with a right-complete presentable $t$-structure compatible with $\otimes$.
    Assume that $\mathcal{C}^\omega\subseteq \Tor{C}$. Then $\PCoh{C}$ and $\Coh{C}$ are  $\mathcal{C}^\omega$-modules.
\end{lemma}
\begin{proof}
    Since $\mb{1}_\mathcal{C}\in\mathcal{C}_{\geq0}$, it follows that $\Tor{C}\subseteq \mathcal{C}^-$. In particular, $\mathcal{C}^\omega\subseteq\mathcal{C}^-$. 
    Let $x\in\PCoh{C}$ be pseudo-coherent and let $c\in\mathcal{C}^\omega$ be compact. It follows that $x\otimes c$ is in $\mathcal{C}^-$, being the  $t$-structure compatible with $\otimes$.
    To be precise, if  $x\in\mathcal{C}_{\geq N}$ and $c\in\mathcal{C}_{\geq M}$ then $c\otimes x\in\mathcal{C}_{\geq NM}$.
    Fix $n\in\N$ and consider a filtered diagram $y:I\to\tau_{\leq n}(\mathcal{C}_{\geq NM})$. Then
    \[\begin{aligned}
        \colim_{i\in I}\Hom_{\tau_{\leq n}(\mathcal{C}_{\geq NM})}(\tau_{\leq n}(c\otimes x), y_i)
        &\simeq \colim_{i\in I}\Hom_{\mathcal{C}}(c\otimes x, y_i)\\
        &\simeq \colim_{i\in I}\Hom_\mathcal{C}(x, c^\vee\otimes y_i)\\
        &\simeq \Hom_{\mathcal{C}}(x, \colim_{i\in I} c^\vee\otimes y_i)\\
        &\simeq \Hom_\mathcal{C}(x, c^\vee\otimes (\colim_{i\in I}y_i))\\
        &\simeq \Hom_{\mathcal{C}}(c\otimes x, \colim_{i\in I}y_i)\\
        &\simeq \Hom_{\tau_{\leq n}(\mathcal{C}_{\geq NM})}(\tau_{\leq n}(c\otimes x), \colim_{i\in I}y_i).
    \end{aligned}\] 
    Here the first equivalence follows by fully-faithfullness and the second one by rigidity.
    For the third one note that, since $c^\vee\in\mathcal{C}^\omega$ is of finite tor-dimension, the objects $c^\vee\otimes y_i$ are in some truncation of a shift of the connective half, so that pseudo-coherence of $x$ allows to commute the colimit. 
    The fourth equivalence follows since $\otimes$ commutes with small colimits in each variable, the fifth one by rigidity and the last one by fully-faithfullness.
    The case of coherent objects follows then from above plus the fact that tensoring with a compact object (hence a tor-finite object) preserves bounded objects.
\end{proof}
We summarize (again) the situation.
\begin{remark}\label{remark: pseudo-coherent objects in the standard $t$-structure}
    Let $\mathcal{C}\in\CAlg^\rig(\Pr^{\text{L},\omega}_\st)$ be a rigid commutative algebra and let $S\subseteq\mathcal{C}^\omega$ a collection of compact objects. Let $\ts{C}$ be the accessible $t$-structure generated by $S$. Assume that:
    \begin{enumerate*}
        \item The collection $S$ generates.
        \item The t-structure is left complete (and hence left separated).
        \item The $t$-structure is compatible with $\otimes$.
        \item Compact objects $\mathcal{C}^\omega\subseteq\Tor{C}$ are of finite tor-dimension.
    \end{enumerate*} 
    Then \autoref{remark: summary 1} implies that  the $t$-structure is right-complete presentable, compactly generated and compatible with $\otimes$. 
    Regarding pseudo-coherence, point $(4)$ together with \autoref{lemma: pcoh and coh are Cc submodules} imply that $\Coh{C},\PCoh{C}$ are $\mathcal{C}^\omega$-modules. 
    It also implies that $\mathcal{C}^\omega\subseteq\PCoh{C}$, and, if the unit is eventually coconnective, also $\mathcal{C}^\omega\subseteq\Coh{C}$.
    In any case, point $(2)$ together with \autoref{corollary: smallnes of coherent/pseudo} imply that $\Coh{C},\PCoh{C}\in\Mod_{\mathcal{C}^\omega}(\Cat^\perf)$.  
    Furthermore, assumption $(S)$ of \autoref{lemma: assumption S} is satisfied: it follows that every connective pseudo-coherent object can be realized as the geometric realization of a simplicial object whose pieces are in $\mathcal{C}_{\geq0}^\omega=\mathcal{C}^\omega\cap\mathcal{C}_{\geq0}$.
\end{remark}
We will come back to these type of categories in \autoref{definition: tor finite categories}.

\subsection{Coherent categories}\label{subsection: Coherent categories}
Even if pseudo-coherent and coherent objects behaved well categorically, it is not entirely clear how to compute them explicitly.  
The goal of this section is to fill this gap by showing that, under some coherence assumption, $\PCoh{C}$ can be computed in terms of the homotopy groups of the $t$-structure. 

\begin{definition}\label{definition: coherent t-structure}
    Let $\mathcal{C}\in\Pr^{\text{L},\omega}_\st$ be a compactly generated stable category equipped with a right complete presentable $t$-structure. We will say  that the $t$-structure is \emph{coherent} if it restricts to $\PCoh{C}$.
\end{definition}

Before exploiting all the properties of coherent $t$-structures we point out a few discrepancies between our definition and the one already available in the literature. 
\begin{warning}
    Our definition of coherent category is \emph{not} the same of Ben-Zvi, Nadler and Preygel \cite[Definition 6.2.7]{ben2017integral}. 
    Lemma 6.2.5 in loc. shows how to compute pseudo-coherent objects in terms of their homotopy groups, proving the same statement of \autoref{proposition: pseudo coherent and pi_n}.
    The author believes that this claim is actually false. Indeed, if $R$ is a connective $\mb{E}_\infty$-ring spectrum with $\pi_0(R)$  coherent, then $\Mod_R$ is coherent in the sense of Definition 6.2.7. 
    In particular, since $R\in(\Mod_R)^\omega\subseteq\PCoh{\Mod_R}$, Lemma 6.2.5 implies that every $\pi_n(R)$ is finitely presented as a  $\pi_0(R)$-module. This is true if and only if $R$ is  coherent in the sense of \cite[Definition 7.2.4.16]{Lurie-HA}. 
\end{warning}
\begin{warning}
    Furthermore, our definition is, in general, different form the Canonaco-Neeman-Stellari \cite[Definition 10.1.1]{canonaco2024passage}.
    It agrees under the assumptions of \autoref{proposition: PCoh = C-c}.
\end{warning}

We now study the heart of a coherent $t$-structure. 

\begin{lemma}\label{lemma: restriction of t-structure on pcoh}
    Let $\mathcal{C}\in\Pr^{\text{L},\omega}_\st$ be a compactly generated  stable category equipped with a right-complete presentable $t$-structure. 
    Then the following conditions are equivalent.
    \begin{enumerate*}
        \item The $t$-structure on $\mathcal{C}$ restricts to one on $\PCoh{C}$, that is the $t$-structure is coherent. 
        \item The inclusion $i_{\leq0}^{\leq1} : \mathcal{C}_{\leq0} \into \mathcal{C}_{\leq1}$ preserves compact objects.
        \item Desuspending $\susp^{-1}: \mathcal{C}_{\leq0}\to\mathcal{C}_{\leq0}$ preserves compact objects.
    \end{enumerate*}
    Assume  that $\mathcal{C}^\omega\subseteq\mathcal{C}^-$. Then in this case $\Coh{C}^\heartsuit =\PCoh{C}^\heartsuit = (\mathcal{C}^\heartsuit)^\omega$. 
    In particular, 
    \begin{enumerate*}
        \item[(4)] The subcategory of compact objects in the heart $(\mathcal{C}^\heartsuit)^\omega \subseteq\mathcal{C}^\heartsuit$ is abelian.
    \end{enumerate*}
\end{lemma}
\begin{proof}
    Consider $(1)\iff(2)$. 
    First of all, recall that the $t$-structure restricts to $\PCoh{C}$ if and only if for every $x\in\PCoh{C}$ the truncation $\tau_{\leq n}x\in\mathcal{C}_{\leq n}$ is again pseudo-coherent.
    Since $\tau_{\leq m}\circ\tau_{\leq n}\simeq \tau_{\leq\text{min}(n,m)}$, the claim then follows by noting that $\tau_{\leq n}x \in\mathcal{C}_{\leq n}$ is compact if and only if  $\tau_{\leq m}x \in\mathcal{C}_{\leq m}$ is compact for all $m\geq n$ and $i_{\leq m}^{\leq n}:\mathcal{C}_{\leq m}\to\mathcal{C}_{\leq n}$ preserves compact objects for $m\geq n$.
    The equivalence $(2)\iff(3)$ follows since $\mathcal{C}_{\leq1}= \susp^{1}\mathcal{C}_{\leq0}$.  
    Assume now the equivalent conditions $(1)$-$(2)$ and $(3)$. 
    First of all, note that the $t$-structure also restricts to $\Coh{C}$, thus producing hearts $\Coh{C}^\heartsuit$ and $\PCoh{C}^\heartsuit$ so that the equality $\Coh{C}^\heartsuit =\PCoh{C}^\heartsuit$ is essentially by definition. 
    To prove that $\PCoh{C}^\heartsuit= (\mathcal{C}^\heartsuit)^\omega$, note that $(\subseteq)$ is always true and that $(\supseteq)$ follows since $i_{\leq0}^{\leq1}$ preserves compact objects. 
    Finally, since the heart of a $t$-structure is always abelian, the above equalities  prove $(4)$.
\end{proof}
Recall that a \emph{locally coherent abelian $1$-category} is a compactly generated Grothendieck abelian $1$-category  whose compact objects form an abelian category. Examples include modules over a coherent ring and quasi-coherent sheaves over a coherent scheme. 
\begin{lemma}
    Let $\mathcal{C}\in\Pr^{\text{L},\omega}_\st$ be a compactly generated  stable category equipped with a coherent $t$-structure such that $\mathcal{C}^\omega\subseteq\mathcal{C}^-$.
    Then $\mathcal{C}^\heartsuit$ is a locally coherent abelian $1$-category.
\end{lemma}
\begin{proof}
    Since $(\mathcal{C}^\heartsuit)^\omega$ forms an abelian subcategory of $\mathcal{C}^\heartsuit$ by \autoref{lemma: restriction of t-structure on pcoh}, it suffices to show that $\mathcal{C}^\heartsuit$ is compactly generated. 
    Since $\mathcal{C}$ is compactly generated, \autoref{lemma: NP 6.1.1} implies that $\mathcal{C}_{\leq0}$ is compactly generated.
    Since \autoref{lemma: restriction of t-structure on pcoh} implies that $i_{\leq0}^{\leq1} : \mathcal{C}_{\leq0} \into \mathcal{C}_{\leq1}$ preserves compact objects, \cite[Corollary 5.5.7.3]{Lurie-HTT} implies that $\mathcal{C}^\heartsuit\simeq\mathcal{C}_{\leq1}/\mathcal{C}_{\leq 0}$ is compactly generated. 
\end{proof}
We may wonder if condition $(4)$ of  \autoref{lemma: restriction of t-structure on pcoh} is equivalent to the other coniditions. 
This is not always the case as we need the following. 
\begin{lemma}
    Let $\mathcal{C}\in\Pr^{\text{L},\omega}_\st$ be a compactly generated  stable category equipped with a right-complete presentable $t$-structure such that $\mathcal{C}^\omega\subseteq\mathcal{C}^-$. 
    Consider the following condition for $n\geq0$.
    \begin{enumerate}
        \item[$(\text{HC})_n$] The category $\mathcal{C}_{\geq0}\cap\mathcal{C}_{\leq n}$ is compactly generated and the compact objects are closed under finite limits.
    \end{enumerate}
    Then the $t$-structure is coherent if and only if $(\text{HC})_n$ holds for every $n\geq0$.
\end{lemma}
\begin{proof}
    Assume that the $t$-structure is coherent. 
    Fix $n\geq0$ and consider the truncation $\tau_{\geq 0}:\mathcal{C}_{\leq n}\to\mathcal{C}_{\geq0}\cap\mathcal{C}_{\leq n}$ with kernel $\mathcal{C}_{\leq -1}$.
    Since $\mathcal{C}_{\geq0}$ is compactly generated by \autoref{lemma: NP 6.1.1} and since the inclusion $\mathcal{C}_{\leq0}\into\mathcal{C}_{\leq 1}$ preserves compacts, an application of \cite[Corollary 5.5.7.3]{Lurie-HTT} implies that $\mathcal{C}_{\geq0}\cap\mathcal{C}_{\leq n}$ is compactly generated. 
    To show that the comapct objects of $\mathcal{C}_{\geq0}\cap\mathcal{C}_{\leq n}$ are closed under finite limits, since $\tau_{\geq0}$ is right exact, it suffices to show that the compact objects of $\mathcal{C}_{\leq n}$ are closed under finite limits. This follows by \autoref{lemma: restriction of t-structure on pcoh}. 

    Conversely, assume that $(\text{HC})_n$ holds for every $n\geq0$. 
    To show that the $t$-structure is coherent it suffices to show that the it restricts to $\PCoh{C}$. 
    Let $x\in\PCoh{C}$ and assume without loss of generality that $x\in\mathcal{C}_{\geq0}$. 
    Fix $k\in\mb{Z}$. 
    Then $\tau_{\leq k}x$ is clearly connective and almost compact in $\tau_{\leq k}\mathcal{C}_{\geq0}$, thus $\tau_{\leq k}x\in\PCoh{C}$.
    For the other truncation, consider the fibre sequence $\tau_{\geq1}x\to x\to \pi_0(x)$. 
    Apply the truncation $\tau_{\leq n}$ for every $n\geq0$ to produce a fibre sequence $\tau_{\leq n}\tau_{\geq1}x\to\tau_{\leq n}x\to\pi_0(x)$. 
    Since all the terms belong to $\mathcal{C}_{\geq0}\cap\mathcal{C}_{\leq n}$, it follows that $\tau_{\leq n}\tau_{\geq1}x\simeq \fib(\tau_{\leq n}x\to\pi_0(x))$ in $\mathcal{C}_{\geq0}\cap\mathcal{C}_{\leq n}$. 
    Since $x$ and $\pi_0(x)$ belong to $\PCoh{C}\cap \mathcal{C}_{\geq0}$, the first one by assumption and the second one being $\pi_0(x)\simeq \tau_{\leq0}x$ the truncation of a pseudo-coherent object, it follows by assumption $(\text{HC})_n$ that $\tau_{\leq n}\tau_{\geq1}x$ is compact in $\mathcal{C}_{\geq0}\cap\mathcal{C}_{\leq n}$. 
    Since this holds for every $n\geq0$, it follows that $\tau_{\geq1}x\in\PCoh{C}$.
    In particular, the $t$-structure restricts to $\PCoh{C}$.
\end{proof}

As we will see, even though condition $(\text{HC})_n$ seems very hard to check in practice, many interesting categories satisfy it.

Then we have the following.
\begin{theorem}\label{proposition: pseudo coherent and pi_n}
   Let $\mathcal{C}\in\Pr^{\text{L},\omega}_\st$ be a compactly generated stable category with a coherent $t$-structure. 
   Assume that the $t$-structure is left separated and $\mathcal{C}^\omega\subseteq\mathcal{C}^-$.
   Assume also that:
   \begin{enumerate}
    \item[(S)] There exists a generating subcategory $S\subseteq\mathcal{C}_{\geq0}$ which is closed under finite colimits and extensions and consists of pseudo-coherent connective objects of $\mathcal{C}$.
   \end{enumerate}
   Then:
    \begin{enumerate*}
        \item  The heart $\Coh{C}^\heartsuit= \Coh{C}\cap\mathcal{C}^\heartsuit$ consists precisely of the compact objects of $\mathcal{C}^\heartsuit$.
        \item An object $x\in\PCoh{C}$ if and only if $\pi_nx\in(\mathcal{C}^\heartsuit)^\omega$ and $\pi_n x=0$ for $n<<0$.
        \item An object $x\in\Coh{C}$ if and only if $\pi_nx\in(\mathcal{C}^\heartsuit)^\omega$ and $\pi_n x=0$ for all but finitely many $n$.
    \end{enumerate*}
    In particular, if the $t$-structure  is left complete, then $\PCoh{C}$ is the left $t$-completion of $\Coh{C}$.
\end{theorem}
\begin{proof}
    Point $(1)$ follows by \autoref{lemma: restriction of t-structure on pcoh}. 
    Since point $(3)$ follows immediately from point $(2)$ by noticing that coherent objects are defined to be bounded above, it suffices to show $(2)$. 
    Let $x\in\mathcal{C}$. 
    Since both conditions imply that $x$ is eventually connective, it is safe to assume that $x\in\mathcal{C}_{\geq0}$ is connective. 
    Suppose first that $x\in\PCoh{C}\cap\mathcal{C}_{\geq0}$ is pseudo-coherent and connective. 
    Since $x$ is connective, $\pi_n(x)\simeq 0$ for $n<0$.
    Since the $t$-structure restricts to $\PCoh{C}$ by \autoref{lemma: restriction of t-structure on pcoh}, it follows that $\tau_{\geq n}x$ belongs to $\PCoh{C}\cap\mathcal{C}_{\geq n}$ for every $n\geq0$. 
    The cofibre sequence $\tau_{\geq n+1}x\to\tau_{\geq n}x\to\susp^n\pi_n(x)$ implies then that $\susp^n\pi_n(x)$ belongs to $\PCoh{C}\cap\mathcal{C}_{\geq n}$. 
    Since $\susp^n\pi_n(x)$ belongs to $\mathcal{C}_{\leq n}$, it follows that $\susp^n\pi_n(x)\in \PCoh{C}\cap \susp^n\mathcal{C}^\heartsuit$.
    By desuspending and applying point $(4)$ of \autoref{lemma: restriction of t-structure on pcoh} it follows that $\pi_n(x)\in\PCoh{C}\cap\mathcal{C}^\heartsuit=(\mathcal{C}^\heartsuit)^\omega$ is compact for every $n\geq0$. 

    Assume now that $x\in\mathcal{C}_{\geq0}$ is connective and that $\pi_n(x)\in(\mathcal{C}^\heartsuit)^\omega$ for every $n\geq0$.
    By \autoref{lemma: 0-connected pseudo-coheren are closed under the formation of geometric realization} it suffices to prove that $x$ can be obtained as geometric realization of a simplicial object $x_\bullet$  such that every $x_i$ is in $S$.
    As in the proof of \autoref{lemma: assumption S}, it will suffice to construct a sequential diagram
    \[
    D(0) \xrightarrow{f_1} D(1) \xrightarrow{f_1} D(2) \to\dots
    \]
    mapping to $x$ such that the cofibre of each map $f_n$ is a $n$-fold suspension of an object of $S$ and such that the canonical map $\colim_{n\in\N} D(n) \to x$ is an equivalence. 
    Equivalently, it suﬀices to produce a cofibre sequence $D(n)\to x\to \susp^nc(n)$ for each $n \in\N$ such that $D(n)$ is in $S$ and $c(n)$ is connective. 
    The proof will show that $c(n)$ can be constructed to have compact homotopy groups. 
    The argument goes by induction. 
    Set, by convention, $D(0)$ to be the zero, so that the cofibre of $0 \to x$ is $x$, hence connective with compact homotopy groups. 
    The inductive hypothesis constructs a diagram $D(n)\to x\to \susp^n c(n)$ with $D(n)$ in $S$ and $c(n)$ connective with compact homotopy groups. 
    In particular, since $\pi_0c(n)\in(\mathcal{C}^\heartsuit)^\omega$ and since $S$ is generating, there exists a finite set $I$ and a morphism $\oplus_{i\in I} p_i\to c(n)$ with $p_i\in S$ for every $i\in I$ inducing a $\pi_0$-epimorphism. 
    Set $P=\oplus_{i\in I_0} p_i$. 
    Since $S$ is closed under finite colimits, it follows that $P\in S$. 
    Define $D(n+1)$ via the pullback square
    \[\begin{tikzcd}[cramped]
	{D(n+1)} & {\susp^nP} \\
	x & {\susp^nc(n)}
	\arrow[from=1-1, to=1-2]
	\arrow[from=1-1, to=2-1]
	\arrow[from=1-2, to=2-2]
	\arrow[from=2-1, to=2-2]
    \end{tikzcd}\]
    Then the cofibre sequence $D(n)\to D(n+1)\to \susp^nP$ exhibits $D(n+1)$ in $S$, being $S$ closed under extensions. 
    Furthermore, the equivalence $\cofib(D(n+1)\to x)\simeq \susp^{n+1}\fib(P\to c(n))$ identifies $c(n+1)$ with $\fib(P\to c(n))$. 
    Since $c(n)$ has compact homotopy groups and since $P$ is pseudo-coherent (thus it has compact homotopy groups by the previous implication) and since $(\mathcal{C}^\heartsuit)^\omega$ is abelian, it follows that the fibre of $P\to c(n)$ has also compact homotopy groups $\pi_i$ which vanish for $i\leq0$ (being $P\to c(n)$ a $\pi_0$-epimorphism), thus concluding the induction hypotheses.
    The above equivalence shows also that the map $D(n+1)\to x$ induces an  isomorphism on $\pi_i$ for $0\leq i \leq n-2$. 
    It follows by \autoref{remark: Neeman lemma 1.3} that the canonical map $\colim_{n\in N} D(n)\to x$ is an equivalence. 
    Since, by construction, the cofiber of each map $D(n)\to D(n+1)$ is an $n$-fold suspension of  an object of $S$, the claim follows by  the (prestable) Dold-Kan correspondence.
\end{proof}

\subsection{Pseudo-compact objects}\label{subsection: pseudo-compact}
We now introduce Neeman's pseudo-compact objects. 
\begin{definition}
    Let $\mathcal{C}$ be a stable category equipped with a $t$-structure $\ts{C}$. An object $x\in\mathcal{C}$ is called \emph{pseudo-compact} if for every $n>0$ there exists a cofibre sequence $c\to x\to d$ where $c\in\mathcal{C}^\omega$ is compact and $d\in\mathcal{C}_{\geq n}$ is $n$-connective.
    We will denote by $\mathcal{C}^-_c$ the full subcategory of $\mathcal{C}$ spanned by the pseudo-compact objects.
    We also denote by $\mathcal{C}^b_c$ the intersection $\mathcal{C}^-_c\cap \mathcal{C}^b$,
\end{definition}
Let  $\mathcal{C}$ be compactly generated and equipped with a compact generator $G\in\mathcal{C}^\omega$ such that $G\in\mathcal{C}_{\geq-N}$ and $\pi_0\Hom_\mathcal{C}(G,\mathcal{C}_{\geq N})=0$ for some integer $N>0$. Then \cite[Lemma 2.9 and Proposition 2.10]{neeman2025triangulatedcategoriessinglecompact} show that $\mathcal{C}^-_c$ and $\mathcal{C}^b_c$ are stable thick subcategories of $\mathcal{C}$. 
In particular, if $\mathcal{C}\in\Pr^\text{L}_\st$ is presentable, then $\mathcal{C}^-_c$ and $\mathcal{C}^b_c$ are idempotent-complete.
\begin{warning}
    A priori there is no reason to expect $\mathcal{C}^-_c$ to be a small category. This issue will be fixed in \autoref{lemma: lemma 7.5 Neeman}.
\end{warning}

\begin{lemma}
    Let $\mathcal{C}\in\CAlg(\Pr^{\text{L},\omega}_\st)$ be commutative algebra equipped with a right-complete presentable $t$-structure compatible with $\otimes$. 
    Let $G\in\mathcal{C}^\omega$ be a compact generator which is $(-N)$-connective $G\in\mathcal{C}_{\geq -N}$ and $\pi_0\Hom_\mathcal{C}(G,\mathcal{C}_{\geq N})=0$  for $N\in\N$. 
    \begin{enumerate*}
        \item Then $\mathcal{C}^-_c$ is a $\mathcal{C}^\omega$-module. 
        \item If $\mathcal{C}^\omega\subseteq \Tor{C}$, then $\mathcal{C}^b_c$ is a $\mathcal{C}^\omega$-module. 
    \end{enumerate*}
    Notice that, contrary to \autoref{lemma: pcoh and coh are Cc submodules}, there is no need to assume that $\mathcal{C}$ is rigid.   
\end{lemma}
\begin{proof}
    Let $x\in\mathcal{C}^-_c$ and $c\in\mathcal{C}^\omega$. 
    Since $\mathcal{C}$ is generated by a single compact object which is eventually connective, it follows that $c\in\mathcal{C}_{\geq -m}$ for some $m>0$. 
    Fix $n>0$. Being $x\in\mathcal{C}^-_c$  pseudo-compact,   the definition with  $n+m>0$ produces a cofibre sequence $c'\to x\to d$ with $c'\in\mathcal{C}^\omega$ and $d\in\mathcal{C}_{\geq n+m}$.
    Tensoring with $c$ produces then the cofibre sequence $c'\otimes_\mathcal{C} c\to x\otimes_\mathcal{C} c\to d\otimes_\mathcal{C} c$ where $c'\otimes_\mathcal{C} c$ is again compact and $d\otimes_\mathcal{C} c \in \mathcal{C}_{\geq n+m}\otimes_\mathcal{C}\mathcal{C}_{\geq -m}\subseteq\mathcal{C}_{\geq n}$.  This proves point $(1)$. Point $(2)$ follows similarly.
\end{proof} 

In the following it will be useful to work with an equivalence class of $t$-structures rather than a single one. 
\begin{remark}
    Let $\mathcal{C}$ be a stable category with a pair of $t$–structures $(\mathcal{C}_{\geq0}^1,\mathcal{C}_{\leq0}^1)$ and $(\mathcal{C}_{\geq0}^2,\mathcal{C}_{\leq0}^2)$.
    We will say that the $t$–structures are \emph{equivalent} if there exists an integer $A\in\N$ such that  $\mathcal{C}_{\geq A}^1\subseteq \mathcal{C}_{\geq0}^2\subseteq \mathcal{C}_{\geq-A}^1$.
    Equivalently, if $\mc{C}_{\leq -A}^1\subseteq \mc{C}_{\leq0}^2\subseteq \mc{C}_{\leq A}^1$.
    Being equivalent is an equivalence relation on the (possibly large or even very large) set of $t$-structures on a stable category.
    It is a common belief that all the meaningful constructions associated to a $t$-structure should only depend on the equivalence class. 
    For example, the full subcategories $\mathcal{C}^-,\mathcal{C}^+$ and $\mathcal{C}^b$ are insensitive to  equivalent $t$-structures.

    Notice, however, that not every property of $t$-structures is preserved by passing to equivalence classes.
    For example:
    \begin{enumerate*}
        \item Being accessible is not preserved.
        \item Being compatible with filtered colimits is not preserved.
        \item Being compactly generated is not preserved. 
        \item The heart of the $t$-structure depends on a representative. 
    \end{enumerate*}
    On the other hand, being right or left complete is instead preserved by the passing to the quotient.
\end{remark}

\begin{remark}
    Let $\mathcal{C}\in\Pr^{\L,\omega}_\st$ be a compactly generated stable category, and assume that $\mathcal{C}$ admits a single compact generator. 
    Then this compact generator determines an equivalence class of $t$-structures.
    In particular, \cite[Lemma 0.13]{neeman2025triangulatedcategoriessinglecompact} and \cite[Example 0.17]{neeman2025triangulatedcategoriessinglecompact} imply that this equivalence class is insensitive to compact generators, in the sense that two different compact generators generate the same equivalence class of $t$-structure. 
    We will refer to this equivalence class as the \emph{preferred equivalence class}. 
\end{remark}
The next observation is a useful criterion to produce left complete $t$-structures.
\begin{remark}\label{lemma: excellent t-structure are preserved by equivalence class}
    Let $\mathcal{C}\in\Pr^{\L,\omega}_\st$ be a compactly generated stable category equipped with a compact generator $G\in\mathcal{C}^\omega$ such that $\pi_0\Hom_\mathcal{C}(G,\susp^iG)=0$ for $i>>0$.
    Then \cite[Lemma 3.6]{Burke_2023} implies every $t$-structure in the preferred equivalence class is left and right complete. 
\end{remark}

\begin{definition}{\protect{\cite[Definition 8.3]{neeman2025triangulatedcategoriessinglecompact}}}
    Let $\mc{C}\in\Pr^{\L,\omega}_\st$ be a compactly generated stable category equipped with right-complete presentable $t$-structure. Assume that $\mc{C}$ comes equipped with a compact generator $G\in\mc{C}^\omega$ such that the $t$-structure is in the preferred equivalence class and there exists an integer $N>0$ such that $G$ is $(-N)$-connective, $G\in\mc{C}_{\geq-N}$ and $\pi_0\Hom_\mc{C}(G,\mc{C}_{\geq N})=0$. 
    A \emph{strong  $\langle G\rangle_n$–approximating system} is a sequence of objects and morphisms $D(1)\to D(2)\to\dots$ such that: 
    \begin{enumerate*}
        \item Each $D(i)$  belongs to $\langle G\rangle_n$.
        \item The map $\pi_l(D(i)) \to \pi_l(D(i+1))$ is an isomorphism in $\mathcal{C}^\heartsuit$ whenever $l\leq i$.
    \end{enumerate*}
    In this definition we also allow $n = \infty$ by declaring $\langle G\rangle_\infty= \mathcal{C}^\omega$. 
    Given an object $x\in\mathcal{C}$, a \emph{strong $\langle G\rangle_n$-approximating system for $x$} is a strong  $\langle G\rangle_n$–approximating system $D(1)\to D(2)\to\dots$ with a map to $x$ that induces isomorphism $\pi_l(D(i))\to \pi_l(x)$  whenever $i\geq l$.    
\end{definition}
The notation $\langle G\rangle_n$ is borrowed from \cite[Reminder 0.12]{neeman2025triangulatedcategoriessinglecompact}. 
It denotes the full thick subcategory of $\mathcal{C}$ whose objects are finite coproducts of at most $n$-extensions of the objects  $\susp^i G$, for $i\in\mb{Z}$. 

Our next  goal is to show that every pseudo-compact object  admits a $\mathcal{C}^\omega$-strong approximating system.
This is a key result in Neeman's paper, and we now present his argument for completeness. But first we recall a couple of technical results.

\begin{remark}[\protect{\cite[Lemma 1.4 and Remark 1.5]{neeman2025triangulatedcategoriessinglecompact}}]\label{remark: Neeman lemma 1.4 and remark 1.5}
    Let $\mathcal{C}\in\Pr^\text{L}_\st$ be a presentable stable category. Assume that $\mathcal{C}$ comes  equipped with a $t$-structure with both $\mathcal{C}_{\geq0}$ and $\mathcal{C}_{\leq0}$ closed under  small colimits in $\mathcal{C}$. Let $D(1)\to D(2)\to\dots$ be a sequence of objects and morphisms in $\mathcal{C}$. 
    Then, for every $l\in\mb{Z}$, there is an exact sequence 
    \[
    0\to \colim_i \pi_l(D(i))\to \pi_l(\colim_i D(i))\to \colim^1_i\pi_{l-1}(D(i))\to 0
    \]
    in the heart $\mathcal{C}^\heartsuit$. Here $\colim^1$ is the derived functor of the colimit. 
    In particular, if the sequences $\pi_l(D(1))\to \pi_l(D(2))\to\dots$ eventually stabilize for every $l$,  then the $\colim^1$ terms all vanish, and the natural map $\colim_i \pi_l(D(i))\to \pi_l(\colim_i D(i))$ is an isomorphism.
\end{remark}
\begin{remark}[\protect{\cite[Lemma 2.8]{neeman2025triangulatedcategoriessinglecompact}}]\label{remark: neeman Lemma 2.8}
    Let $\mathcal{C}\in\Pr^\text{L}_\st$ be a presentable stable category. Let  
    $\ts{C}$ be a $t$–structure on $\mathcal{C}$. 
    Assume that $\mathcal{C}$ comes equipped with  a compact generator $G\in\mathcal{C}^\omega$ and an integer $N>0$ such that $\pi_0\Hom_\mathcal{C}(G, \mathcal{C}_{\geq N})=0$.
    Then for any compact object $c\in\mathcal{C}^\omega$ there exists an integer $n > 0$, depending on $c$, with $\pi_0\Hom_\mathcal{C}(c, \mathcal{C}_{\geq n})=0$.
    Indeed, since $G$ is a compact generator, it follows that $c\in\mathcal{C}^\omega=\langle G\rangle$ must belong to some $\langle G\rangle^{[-m,m]}$. By picking  $n=m+N$ the claim follows.
\end{remark}

We can now prove the following.
\begin{lemma}[\protect{\cite[Lemma 8.5]{neeman2025triangulatedcategoriessinglecompact}}]\label{lemma: lemma 7.5 Neeman}
    Let $\mc{C}\in\Pr^{L,\omega}_\st$ be a compactly generated stable category equipped with a right-complete presentable $t$-structure. Assume that $\mc{C}$ comes equipped with a compact generator $G\in\mc{C}^\omega$ such that the $t$-structure is in the preferred equivalence class and there exists an integer $N>0$ such that $G$ is $(-N)$-connective, $G\in\mc{C}_{\geq-N}$ and $\pi_0\Hom_\mc{C}(G,\mc{C}_{\geq N})=0$. 
    Then:
     \begin{enumerate*}
        \item Every $\langle G\rangle_n$-strong approximating system $D(1)\to D(2) \to \dots$ is a strong $\langle G\rangle_n$-approximating system for the sequential colimit $\colim_i D(i)$. Moreover $\colim_i D(i)$ belongs to $\mathcal{C}^-_c$.
        \item Given $x\in\mathcal{C}^-$ and a strong $\langle G\rangle_n$-approximating system $D$ for $x$, then the canonical map $\colim_{i} D(i) \to x$ is an equivalence.
         \item Every object $x\in\mathcal{C}^-_c$ admits a strong $\mathcal{C}^\omega$–approximating system.  
     \end{enumerate*}
\end{lemma}
\begin{proof} 
    We recall the proof for completeness.
    Consider $(1)$. Take  a $\langle G\rangle_n$-strong approximating system $D(1)\to D(2) \to \dots$ and let $d=\colim_i D(i)$ be the colimit. 
    We first show that $d$ is in $\mathcal{C}_{\geq -n}$ for some $n>0$. 
    Since every factor $D(i)$ is in $\langle G\rangle_n$, it is also in $\mathcal{C}^\omega$. 
    In particular, being $G\in\mathcal{C}_{\geq-N}$ , every  $D(i)$ is in $\mathcal{C}^-$.

    Choose now $n>0$ such that $D(1)\in\mathcal{C}_{\geq-n}$. Since the map $\pi_l(D(1))\to\pi_l(D(m))$ is an isomorphism for all $l\leq 1$, it follows that $\pi_l(D(m))=0$  for all $l>-n$ and all $m$. 
    \autoref{remark: Neeman lemma 1.2} implies now that the $D(m)$ all lie in $\mathcal{C}_{\geq -n}$. 
    Hence the  colimit $d$ also belongs to $\mathcal{C}_{\geq -n}$.
    Now \autoref{remark: Neeman lemma 1.4 and remark 1.5} implies that the map $\colim_i \pi_l(D(i))\to \pi_l(d)$ is an isomorphism for every $l\in\mb{Z}$. On the other side, the cofibre sequence $D(m)\to D\to \cofib_m$ lies in $\mathcal{C}^-$, and since the  map $\pi_l(D(m))\to\pi_l(d)$ is an isomorphism for $m\geq l$ it follows that $\pi_l(\cofib_m) = 0$ for all $m\geq l$.
    Once again \autoref{remark: Neeman lemma 1.2} implies that $\cofib_m\in\mathcal{C}_{\geq m+1}$, and as $D(m)\in\langle G\rangle_n\subseteq\mathcal{C}^\omega$ it follows that $d\in\mathcal{C}^-_c$

    Consider $(2)$.
    Since the map $\colim_i D(i)\to x$  is a morphism from an object of $\mathcal{C}^-_c$ to an object in $\mathcal{C}^-$, it must be a morphism in $\mathcal{C}^-$. \autoref{remark: Neeman lemma 1.4 and remark 1.5} implies then that the natural map $\pi_l(\colim_i D(i))\to \pi_l(x)$ is an isomorphism in $\mathcal{C}^\heartsuit$.
    Apply now \autoref{remark: Neeman lemma 1.3} to deduce the claim.

    Consider $(3)$. The proof is by induction.
    Since $x\in\mathcal{C}^-_c$, there exists a cofibre sequence $D(1)\to x\to \cofib_1$  with $D(1)\in\mathcal{C}^\omega$ and $\cofib_1\in\mathcal{C}_{\geq3}$. 
    When $l \leq 1$, the exact sequence 
    \[
    \pi_{l+1}(\cofib_1) \to\pi_l(D(1))\to\pi_l (x)\to\pi_l(\cofib_1)
    \]
    has $\pi_{l+1}(\cofib_1)=0=\pi_l(\cofib_1)$, starting the construction of $D(n)$. 
    Assume that the sequence has been constructed up to an integer $n > 0$, so that there exists a map $f_m : D(m)\to x$, with $D(m)\in\mathcal{C}^\omega$, and so that $\pi_l(f_m)$ is an isomorphism for all $m\leq l$. 
    \autoref{remark: neeman Lemma 2.8} allows one to choose an integer $N > 0$ so that $\pi_0\Hom_\mathcal{C}(D(m),\mathcal{C}_{\geq N})=0$. 
    Because $x$ belongs to $\mathcal{C}^-_c$, there exists  a cofibre sequence $D(m+1)\to x\to \cofib_{m+1}$  with $D(m+1)\in\mathcal{C}^\omega$ and $\cofib_{m+1}\in\mathcal{C}_{\geq N+m+3}$.  As in the paragraph above, the map $\pi_l(D(m+1))\to\pi_l(x)$  is an isomorphism for all $m+1\geq l$. Since the composite $D(m)\to x\to D(m+1)$ is null-homotopic, the map $f_m$ must factor as $D(m)\to D(m+1)\to x$, concluding the proof.
\end{proof}

We can now prove our comparison result between pseudo-coherent and pseudo-compact objects by means of \autoref{lemma: assumption S}.
\begin{proposition}\label{proposition: PCoh = C-c}
    Let $\mc{C}\in\Pr^{L,\omega}_\st$ be a compactly generated stable category equipped with right-complete presentable $t$-structure. 
    Assume that $\mathcal{C}$ comes equipped with a compact generator $G\in\mathcal{C}^\omega$ such that the $t$-structure is in the preferred equivalence class and such that $\pi_0\Hom_\mathcal{C}(G,\susp^i G)=0$ for $i>>0$.
    Assume also that:
    \begin{enumerate}
    \item[(S)] The subcategory $\mathcal{C}^\omega\cap\mathcal{C}_{\geq0}\subseteq \mathcal{C}_{\geq0}$ is a generating subcategory of $\mathcal{C}_{\geq0}$.
    \end{enumerate} 
    Then $\PCoh{C}=\mathcal{C}^-_c$.
\end{proposition}
\begin{proof}
    Let $x\in\mathcal{C}$.
    First of all, since $\mathcal{C}$ has an eventually connective compact generator,  every compact object is  bounded below. 
    In particular, every object in $\mathcal{C}^-_c$ is an extension of bounded below objects, and hence bounded  below. Since  every object in $\PCoh{C}$  is bounded below, it is safe to assume that $x\in\mathcal{C}_{\geq0}$ is $0$-connective. 

    Assume first that $x\in\PCoh{C}\cap\mathcal{C}_{\geq0}$. The proof of \autoref{lemma: assumption S} showed that  $x$ can be written as a sequential colimit  $x\simeq\colim_n D(n)$ where each $D(n)\in\mathcal{C}^\omega\cap\mathcal{C}_{\geq0}$ is compact and connective.  
    Fix now $n>0$. Since $D(n+1)$ is in $\mathcal{C}^\omega\cap\mathcal{C}_{\geq0}$ and the $t$-structure is compatible with filtered colimits, point $(6)$ of \autoref{lemma: NP 6.1.1} implies that $D(n+1)$ is compact in $\mathcal{C}$. 
    In particular, the map $D(n+1)\to x$ is a morphism from a compact object  to $x$, and it is a $\pi_i$-equivalence for every $i\leq n$, hence its cofibre lies in $\mathcal{C}_{\geq n}$, thus showing $x\in\mathcal{C}^-_c$.  

    Assume now that $x\in\mathcal{C}^-_c\cap\mathcal{C}_{\geq0}$. We will prove that $x$ is pseudo-coherent by showing that it can be written as a sequential colimit $x\simeq \colim_i D(i)$ of $0$-connective and pseudo-coherent objects $D(i)$. 
    Then the $\infty$-Dold-Kan correspondence coupled with \autoref{lemma: 0-connected pseudo-coheren are closed under the formation of geometric realization} will conclude. 
    Now we can easily modify the proof of  point $(3)$ of \autoref{lemma: lemma 7.5 Neeman} and pick the $D(i)$'s in $\mathcal{C}_{\geq0}\cap\mathcal{C}^\omega$. Since $x$ is $0$-connective (hence in $\mathcal{C}^-$), point $(2)$ allows us to apply \autoref{remark: Neeman lemma 1.3}. We deduce that the canonical map $\colim_i D(i)\to x$ is an equivalence, thus concluding the proof. This implication does not use assumption (S).
\end{proof}

\section{Dualities}\label{section: Neeman Dualities}
We now take the first steps towards the two main results of this paper. 
\subsection{Tor-finite categories}\label{subsection: Tor-finite categories}
We start by axiomatizing the categories which are \enquote{tor-finite}.
\begin{construction}
    Let $\mathcal{C}\in\CAlg^\rig(\Pr^{\L,\omega}_\st)$ be a rigidly-compactly generated stable category. 
    We let $\Acc{C}$ denote the (possibly very very large) set of  small collection of compact objects $S\subseteq \mathcal{C}^\omega$ for which the inclusion $S\into\mathcal{C}^\omega$ is symmetric monoidal. 
    We regard $\Acc{C}$ as a poset ordered by inclusion. 
    Then every rigid functor $f^\L:\mathcal{C}\to\mathcal{D}$ in $\Pr^{\L,\omega}_\st$ determines a functor $f:\Acc{C}\to\Acc{D}$. 
    In particular, the assignment $\mathcal{C}\mapsto \Acc{C}$ may be regarded as a covariant functor $\CAlg^\rig(\Pr^{\L,\omega}_\st)\to\widehat{\widehat{\text{Cat}}}$ to very very large categories. 
    The cocartesian unstraightening produces then a cocartesian fibration $\Pr^{S,\otimes}\to\CAlg^\rig(\Pr^{\L,\omega}_\st)$. 
    Objects of $\Pr^{S,\otimes}$ may be identified with a pair $(\mathcal{C},S_\mathcal{C})$ consisting of a rigidly-compactly generated stable category $\mathcal{C}$ and a symmetric monoidal embedding $S_\mathcal{C}\subseteq \mathcal{C}^\omega$; morphisms in $\Pr^{S,\otimes}$ are then rigid functors $f^\L:\mathcal{C}\to\mathcal{D}$ in $\Pr^{\L,\omega}_\st$ such that $f^\L(S_\mathcal{C})\subseteq S_\mathcal{D}$.
\end{construction}
\begin{remark}
    Let $(\mathcal{C},S_\mathcal{C})$ be in $\Pr^{S,\otimes}$. 
    Accordingly to \autoref{remark: lurie construction of t}, picking a  small skeleton of $S_\mathcal{C}$ produces an accessible $t$-structure on $\mathcal{C}$. 
    The construction of this $t$-structure is insensitive to the choice of the skeleton generating it. 
    In particular, the symmetric monoidal embedding $S_\mathcal{C}\into\mathcal{C}^\omega$ shows that the $t$-structure on $\mathcal{C}$  is compatible with $\otimes$. 
    It follows that $(\mathcal{C},S_\mathcal{C})$ may be identified with a rigidly-compactly generated stable category $\mathcal{C}$ together with an accessible $t$-structure on $\mathcal{C}$ which is compatible with $\otimes$. 
    In the following we will always tacitly assume this type of construction.
\end{remark}

\begin{remark}
    Let $(\mathcal{C},S_\mathcal{C})$ be in $\Pr^{S,\otimes}$ and assume that $S_\mathcal{C}$ generates, meaning that there exists a small skeleton of generators inside $S_\mathcal{C}$. 
    Then \autoref{lemma: conditions of t-structure} implies that the accessible $t$-structure generated by $S_\mathcal{C}$ is right-complete and compatible with filtered colimits. 
    In other words it is right-complete presentable and compatible with $\otimes$. 
\end{remark}

\begin{definition}\label{definition: tor finite categories}
    We let $\Pr^\text{tor}$ denote the non-wide non-full subcategory of $\Pr^{S,\otimes}$ on those $(\mathcal{C},S_\mathcal{C})$ such that $S_\mathcal{C}$ generates and the resulting the $t$-structure is left complete and $\mathcal{C}^\omega\subseteq \Tor{C}$ and on those functors whose right adjoint  is right $t$-exact up to a shift. 
    We will refer to objects and functors in $\Pr^\text{tor}$ as \emph{tor-finite categories} and \emph{functors of tor-finite categories}, respectively.
\end{definition}
\begin{remark}
    To spell it out:
    \begin{enumerate*}
        \item An object in $\Pr^\text{tor}$ consists of a rigidly-compactly generated category $\mathcal{C}$ together with a $t$-structure generated by a collection of compact objects $S_\mathcal{C}$. 
        The $t$-structure is accessible, compatible with $\otimes$, compatible with filtered colimits, left-complete and right-complete. 
        Furthermore, $\mathcal{C}^\omega\subseteq \Tor{C}$.
        \item A functor in $\Pr^\text{tor}$ consists of a rigid functor $f^\L:\mathcal{C}\to\mathcal{D}$ which is right $t$-exact and such that its right adjoint $f^\R:\mathcal{D}\to\mathcal{C}$ is left $t$-exact and right $t$-exact up to a shift. 
        Notice that $f^\L$ is an internal left adjoint in $\Pr^\L_\st$, so that $f^\R$ preserves all colimits.
    \end{enumerate*}
    Furthermore, since the monoidal unit is connective, \autoref{lemma: tor and compact} implies that  $\mathcal{C}^\omega\subseteq\PCoh{C}\cap \Tor{C}$. 
    More importantly, \autoref{corollary: smallnes of coherent/pseudo} and \autoref{lemma: pcoh and coh are Cc submodules} imply that $\Coh{C},\,\PCoh{C}\in\Mod_{\mathcal{C}^\omega}(\Cat^\perf)$ are small idempotent-complete stable modules over $\mathcal{C}^\omega$. 
\end{remark}
The next observation links tor-finite categories to the theory of approximation of Neeman.
\begin{remark}
    Let $\mathcal{B}\in\Pr^\L_\st$ be a compactly generated stable category and let $S_\mathcal{B}\subseteq \mathcal B^\omega$ be a small collection of compact object that generates. 
    Then for every connective object $x\in\mathcal{B}_{\geq0}$ there exists a cofibre sequence $s\to x\to c$ where $s$ is a sum of objects of $S_\mathcal{B}$ and $c\in\mathcal{B}_{\geq1}$ is $1$-connective.

    Indeed since $S_\mathcal{B}$ is generating there exists a $\pi_0$-epimorphism $s\to x$ with $s$  a sum of objects of $S_\mathcal{B}$. The long exact sequence in homotopy associated to the cofibre sequence $s\to x\to c$ shows that not only $c\in\mathcal{B}_{\geq0}$, but $c\in\mathcal{B}_{\geq1}$ since $\pi_0c\simeq 0$.
\end{remark}

\begin{remark}\label{remark: thick generation}
    Let $\mathcal{B}\in\Pr^\L_\st$ be a compactly generated stable category and let $S\subseteq \mathcal B^\omega$ be a small collection of compact object that generates. 
    In other terms, the smallest localizing subcategory $\text{Loc}_{\mathcal B}(S)$ containing $S$ is all of $\mathcal B$.
    Then $S$ \emph{thickly generates the compact objects}, in the sense that  $\mathcal{B}^\omega = \text{Thick}_{\mathcal{B}}(S)$, where $\text{Thick}_{\mathcal B}(S)\subseteq \mathcal B^\omega$ denotes the smallest idempotent-complete stable subcategory containing $S$.
    Indeed, let $\text{Loc}_{\mathcal B}(S)\subseteq \mathcal B$ be the localizing subcategory generated by $S$.
    Since $S$ consists of compact objects, the compact objects of $\text{Loc}_{\mathcal B}(S)$ are exactly the thick closure of $S$, that is $(\text{Loc}_{\mathcal B}(S))^\omega=\text{Thick}_{\mathcal B}(S)$.
    In particular, the equality $\text{Loc}_{\mathcal B}(S)=\mathcal B$,  yields $\mathcal B^\omega=\text{Thick}_{\mathcal B}(S)$.
\end{remark}

\subsection{Quasi-perfect and quasi-proper functors}\label{subsection: Quasi-Perfect and Quasi-Proper Functors}
The goal of this section is to study the behaviour of (pseudo-)coherent object under morphisms of commutative algebras.

\begin{lemma}\label{lemma: f^l preserves pseudo-coherent}
    Let $f^\L:\mathcal{A}\to\mathcal{B}$ be an internal left adjoint in $\Pr^\L_\st$ and assume that $\mathcal{A}$ and $\mathcal{B}$ are equipped with accessible  and compatible with filtered colimits $t$-structures such that $f^\L$ is right $t$-exact and $f^\R$ is right $t$-exact up to a shift.
     Then $f^\L$ restricts to a functor $f^\L:\PCoh{A}\to\PCoh{B}$. 
\end{lemma}
\begin{proof}
    Let $x\in\PCoh{A}$ be pseudo-coherent, and assume without loss of generality, that $x\in\mathcal{A}_{\geq0}$. Since $f^\L$ is right $t$-exact, $f^\L(x)$ is in $\mathcal{B}_{\geq0}$. Let $y_i: I\to\tau_{\leq n}(\mathcal{B}_{\geq0})$ be a filtered diagram. Then:
    \[\begin{aligned}
        \colim_{i\in I}\Hom_{\tau_{\leq n}(\mathcal{B}_{\geq0})}(\tau_{\leq n}f^\L(x), y_i)
        &\simeq \colim_{i\in I}\Hom_{\mathcal{B}_{\geq0}}(\tau_{\leq n}f^\L(x), y_i)\\
        &\simeq \colim_{i\in I}\Hom_{\mathcal{B}}(\tau_{\leq n}f^\L(x), y_i)\\
        &\simeq\colim_{i\in I}\Hom_{\mathcal{B}}(f^\L\tau_{\leq n}(x), y_i)\\
        &\simeq\colim_{i\in I}\Hom_{\mathcal{A}}(\tau_{\leq n}(x),f^\R y_i)
    \end{aligned}\]
    Here the first and second equivalence follows by fully-faithfullness, the third one by \cite[Proposition 5.5.6.28]{Lurie-HTT} and the last one by adjunction. 
    Since $f^\R$ is left $t$-exact and right $t$-exact up to a shift $N\in\N$, it follows that $f^\R(y_i)\in\tau_{\leq n}\mathcal{A}_{\geq-N}$.
    Thus the claim will be proved (after translating back through the preceding equivalences) if the inclusion $\mathcal{A}_{\geq0}\into\mathcal{A}_{\geq-N}$ preserves compact objects.  This follows by \autoref{lemma: NP 6.1.1}, being the $t$-structure compatible with filtered colimits. 
\end{proof}

The situation is much easier with pseudo-compact objects.
\begin{lemma}\label{remark: pullback preserves pseudo-compact}
    Let $f^\L:\mathcal{A}\to\mathcal{B}$ be in $\CAlg^\rig(\Pr^{\L,\omega}_\st)$. 
    Assume that  $\mathcal{A}$ and $\mathcal{B}$ are equipped with $t$-structures such that $f^\L$ is right $t$-exact.
    Then $f^\L$ preserves pseudo-compact objects. 
\end{lemma}
\begin{proof}
    If $x\in\mathcal{A}^-_c$, then for every $m>0$ there exists a cofibre sequence $b\to x\to e$ where $b\in\mathcal{A}_c$ is compact and $e\in\mathcal{A}_{\geq m}$. Exactness of $f^\L$ produces then the cofibre sequence $f^\L(b)\to f^\L(x)\to f^\L(e)$, which exhibits $f^\L(x)$ as a pseudo-compact object of $\mathcal{B}$ since $f^L$ preserves compacts being rigid.
\end{proof}
We are more interested in understanding when the right adjoint $f^\R:\mathcal{B}\to\mathcal{A}$ preserves pseudo-coherent objects. 
Note that, when $f^\R$ sends pseudo-coherent to pseudo-coherent, then it also preserves coherent objects, being left $t$-exact. 
Anyway, let us give a name to these functors.
\begin{definition}
    Let $f^\L:\mathcal{A}\to\mathcal{B}$ be a functor of tor-finite categories. 
    We will say that $f^\L$ is \emph{quasi-proper} if the right adjoint $f^\R:\mathcal{B}\to\mathcal{A}$ preserves pseudo-coherent objects. 
\end{definition}
The nomenclature comes from algebraic geometry.  
Indeed, a map $f:X\to Y$ between quasi-compact quasi-separated schemes is called quasi-proper if the (derived) pushforward $f_*$ preserves pseudo-coherent complexes. 
We suggest Lipman and Neeman's article \cite{lipman2007quasi} for a nice review of all the categorical properties of the (derived) pushforward. 
For us what matters is that they also review the notion of quasi-perfect maps, that is those maps of quasi-compact quasi-separated schemes whose (derived) pushforward  preserves perfect complexes.  
This notion has an immediate generalization in our context.
\begin{definition}
    Let $f^\L:\mathcal{A}\to\mathcal{B}$ be a functor of tor-finite categories. 
    We will say that $f^\L$ is \emph{quasi-perfect} if the right adjoint $f^\R:\mathcal{B}\to\mathcal{A}$ preserves compact objects. 
\end{definition}
\begin{remark}
     Let $f^\L:\mathcal{A}\to\mathcal{B}$ be a functor of tor-finite categories. 
     Then $f^\L$ is quasi-perfect if and only the right adjoint $f^\R$ is an internal left adjoint in $\Mod_\mathcal{A}(\Pr^{\L}_\st)$ by \autoref{example: internal left adjoint for rigidly-compactly generated}. In particular, the Grothendieck-Neeman duality, that is \autoref{theorem: Grothendieck Duality}, applies.
\end{remark}

The main result of the above mentioned article is \cite[Theorem 1.2]{lipman2007quasi}. It says that a quasi-compact quasi-separated map $f:X\to Y$ between quasi-compact quasi-separated schemes is quasi-perfect if and only if it is quasi-proper and the twisted inverse image $f^{(1)}$ is $t$-bounded (or, equivalently, that $f$ is quasi-proper and of finite tor-dimension). We can prove a (partial) analogue of this statement in the abstract setting we are developing.

\begin{lemma}\label{lemma: equivalent condition for quasi-proper}
    Let $f^\L:\mathcal{A}\to\mathcal{B}$ be a functor of tor-finite categories. 
    Then $f^\L$ is quasi-proper if and only if $f^\R$ carries compact objects to pseudo-coherent objects.
\end{lemma}
\begin{proof}
    The implication $(\Rightarrow)$ is trivial since $\mathcal{B}^\omega\subseteq\PCoh{B}$ in a tor-finite category. 
    Consider $(\Leftarrow)$. 
    Let $N\in\N$ be the error of right $t$-exactness of $f^\R$ so that $f^\R(\mathcal{B}_{\geq0})\subseteq \mathcal{A}_{\geq -N}$.
    Let $x\in\PCoh{B}$ be pseudo-coherent. 
    After a shift, it is safe to assume that $x\in\mathcal{B}_{\geq0}$ is connective. 
    In paritcular, \autoref{lemma: assumption S} implies that $x$ can be written as a sequential colimit $x\simeq \colim_{n\in N}D(n)$ where the cofibre of each map $D(n-1)\to D(n)$ is a $n$-fold suspension of an object of $\mathcal{B}^\omega_{\geq0} = \mathcal{B}^\omega\cap\mathcal{B}_{\geq0}$, thus of a compact and connective object of $\mathcal{B}$. 
    Since $f^\L$ is an internal left adjoint, the functor $f^\R$ preserves colimits, thus $f^\R(x)\simeq f^\R(\colim_{n\in N}D(n))\simeq \colim f^\R(D(n))$. 
    Now the cofibre of each $f^\R(D(n-1))\to f^\R(D(n))$ is given by the $n$-fold suspension of an object of $\PCoh{A}\cap\mathcal{A}_{\geq-N}$ since $f^\R$ carries compact objects to pseudo-coherent objects. 
    By the $\infty$-Dold-Kan correspondence, this sequential colimit corresponds to a simplicial object whose pieces are pseudo-coherent and $(-N)$-connective. Thus, after a shift,  \autoref{lemma: 0-connected pseudo-coheren are closed under the formation of geometric realization} implies that $f^\R(x)$ is pseudo-coherent. 
\end{proof}

\begin{lemma}\label{lemma: right adjoint preserves tor finite}
    Let $f^\L:\mathcal{A}\to\mathcal{B}$ be a functor of tor-finite categories which is left $t$-exact up to a shift. 
    Then $f^\R$ restricts to a functor $\Tor{B}\to\Tor{A}$. 
\end{lemma}
\begin{proof}
    Let $x\in\Tor{B}$. To show that tensoring with $f^\R(x)$ defines a left $t$-exact and right $t$-exact functor up to a shift it suffices to notice that $f^\R(x)\otimes-\simeq f^\R(x\otimes f^\L(-))$ is a composition of left $t$-exact and right $t$-exact functors up to a shift, since $f^\L$ is left $t$-exact up to a shift.
\end{proof}
In particular, we deduce the following result. 
\begin{corollary}\label{corollary: quasi-perfect implies quasi-proper}
    Let $f^\L:\mathcal{A}\to\mathcal{B}$ be a functor of tor-finite categories. 
    \begin{enumerate*}
        \item  If $f^\L$ is quasi-perfect, then it is quasi-proper. 
    \end{enumerate*}
    Assume that $\mathcal{A}^\omega = \PCoh{A}\cap\Tor{A}$.
    \begin{enumerate*}
        \item[(2)] 
        If $f^\L$ is quasi-proper and left $t$-exact up to a shift, then it is quasi-perfect.
    \end{enumerate*}
\end{corollary}
\begin{proof}
    For point $(1)$, if $f^\L$ is quasi-perfect, then its right adjoint $f^\R$ preserves compact objects. Then the second condition of \autoref{lemma: equivalent condition for quasi-proper} applies.
    For point $(2)$, since $f^\R$ preserves pseudo-coherent objects and since $f^\L$ is left $t$-exact up to a shift, \autoref{lemma: right adjoint preserves tor finite} implies that $f^\R$ restricts to a functor $\mathcal{B}^\omega\subseteq\PCoh{B}\cap\Tor{B}\to\PCoh{A}\cap\Tor{A}=\mathcal{A}^\omega$. Thus $f^\R$ preserves compact objects, hence $f^\L$ is quasi-perfect.
\end{proof}

\subsection{Functors out of $(\mathcal{B}^\omega)^\op$}\label{subsection: Functors out of compact objects}

Let $f^\L:\mathcal{A}\to\mathcal{B}$ be a functor of tor-finite categories.
Then $f^\L$ restricts to a functor between compact objects and exhibits $\mathcal{B}^\omega$ as a small idempotent-complete stable commutative algebra under $\mathcal{A}^\omega$. In particular, we may regard $\mathcal{B}^\omega$ as an $\mathcal{A}^\omega$-module in $\Cat^\perf$. 
Thus an application of \autoref{corollary: equivalence ind and exact functors} produces an equivalence
\[
\mathcal{B}\to\Fun_{\mathcal{A}^\omega}^\ex((\mathcal{B}^\omega)^\op,\mathcal{A})
\]
in $\Mod_\mathcal{A}(\Pr^\L_\st)$. 
The equivalence is explicitly given by $x\mapsto \mathcal{B}(-,x)$ where $\mathcal{B}(-,-)= f^\R\IHom_\mathcal{B}(-,-)$ is the linear Yoneda. 
Our goal is now to study the above Yoneda embedding when the source is a $\mathcal{B}^\omega$-module $\mathcal{B}_0\subseteq\mathcal{B}$, such as $\Coh{B}\subseteq\PCoh{B}$. 
In general, by restricting the source to a $\mathcal{B}^\omega$-module $\mathcal{B}_0\subseteq\mathcal{B}$, the linear Yoneda embedding will still be fully-faithful. 
However, it  will cease to be an equivalence, since not every exact and $\mathcal{A}^\omega$-linear functor $(\mathcal{B}^\omega)^\op\to\mathcal{A}$ arises from an object of $\mathcal{B}_0$. 
For this reason, it is helpful to approach the problem from the other side. 
Instead of restricting the source of the linear Yoneda embedding, it is easier to restrict its target, and then identify its kernel.

We now study the case of (pseudo-)coherent objects in tor-finite categories.

\begin{lemma}\label{lemma: quasi-proper functor and pseudo}
    Let $f^\L:\mathcal{A}\to\mathcal{B}$ be a quasi-proper functor. 
    Then the $\mathcal{A}^\omega$-linear Yoneda embedding produces functors
    \[
    \PCoh{B}\to\Fun^\ex_{\mathcal{A}^\omega}((\mathcal{B}^\omega)^\op,\PCoh{A}),
    \qquad
    \Coh{B}\to\Fun^\ex_{\mathcal{A}^\omega}((\mathcal{B}^\omega)^\op,\Coh{A})
    \]
    which are fully-faithful.
\end{lemma}
\begin{proof}
    Indeed, being $\PCoh{B}$ a small idempotent-complete stable $\mathcal{B}^\omega$-module,  the tensor product of $\mathcal{B}$ restricts to a functor $\PCoh{B}\times\mathcal{B}^\omega\to\PCoh{B}$. 
    Since $f^\L$ is quasi-proper, the right adjoint $f^\R$ sends $\PCoh{B}$ to $\PCoh{A}$. 
    It follows that the first functor is fully-faithful, being the restriction of a fully-faithful functor.
    The claim for coherent objects follows since coherent objects are the coconnective pseudo-coherent objects, and $f^\R$ is left $t$-exact. 
\end{proof}

We begin with some simple observations. 
\begin{remark}
    Let $f^\L:\mathcal{A}\to\mathcal{B}$ be a functor of tor-finite categories. 
    Then for every $b\in\mathcal{B}^\omega$ the evaluation at $b$ admits a left adjoint 
    \[
    b_!:\mathcal{A}\rightleftarrows \Fun^\ex_{\mathcal{A}^\omega}((\mathcal{B}^\omega)^\op,\mathcal{A}): \text{ev}_b
    \]
    and a right adjoint.
    Indeed, since limits and colimits in the functor category are computed pointwise and are preserved by $\text{ev}_b$, the usual yoga with Kan-extension implies that $\text{ev}_b$ has adjoints on both sides. 
\end{remark}
\begin{remark}\label{remark: component of counit has a section}
    Let $f^\L:\mathcal{A}\to\mathcal{B}$ be a functor of tor-finite categories. 
    The counit $\epsilon_b:b_!\circ\text{ev}_b\to \id$ evaluated at an exact $\mathcal{A}^\omega$-linear functor $F:(\mathcal{B}^\omega)^\op\to\mathcal{A}$ and at an object $b\in\mathcal{B}^\omega$ admits a section given by the unit $\mb{1}_\mathcal{A}\to\mathcal{B}(b,b)$ as a consequence of the triangle identities.
    In particular, $\text{ev}_b(\epsilon_b(F)): \mathcal{B}(b,b)\otimes_\mathcal{A}F(b)\to F(b)$ is a split epimorphism. 
\end{remark}
\begin{remark}\label{remark: t-structure on functor category}
    Let $f^\L:\mathcal{A}\to\mathcal{B}$ be a functor of tor-finite categories. 
    Denote by $F_{\geq0}$ the smallest full subcategory of $\Fun_{\mathcal{A}^\omega}^\ex((\mathcal{B}^\omega)^\op,\mathcal{A})$ closed under small colimits and extensions and spanned by $b_!(a)$ for $a\in S_\mathcal{A}$ and $b\in S_\mathcal{B}$. 
    Regard it as a $t$-structure via \autoref{remark: lurie construction of t}. 
    Since $f^\L(S_\mathcal{A})\subseteq S_\mathcal{B}$ and since the monoidal unit of $\mathcal{A}$ is connective, the equivalence $\mathcal{B}\to\Fun_{\mathcal{A}^\omega}^\ex((\mathcal{B}^\omega)^\op,\mathcal{A})$ extends to a $t$-exact equivalence. 
    It follows that the $t$-structure generated by $F_{\geq0}$ inherits all the properties enjoyed by the $t$-structure on $\mathcal{B}$.
    Moreover, $b_!$ is right $t$-exact and $\text{ev}_b$ is left $t$-exact for every $b\in S_\mathcal{B}$ (and hence for every $b\in\mathcal{B}_{\geq0}$). 
\end{remark}
Let $f^\L:\mathcal{A}\to\mathcal{B}$ be a functor of tor-finite categories. The previous observation implies that the $\mathcal{A}^\omega$-linear Yoneda restricts to equivalence
\[
\PCoh{B}\to\text{PCoh}(\Fun^\ex_{\mathcal{A}^\omega}((\mathcal{B}^\omega)^\op,\mathcal{A})),
\qquad
\Coh{B}\to\text{Coh}(\Fun^\ex_{\mathcal{A}^\omega}((\mathcal{B}^\omega)^\op,\mathcal{A}))
\]
between the pseudo-coherent and coherent objects. 
What is left now to show is that $\text{PCoh}$ and $\text{Coh}$ can \enquote{enter into the second member of the functor category}.
There are three ingredients.
The first one is the following.
\begin{lemma}\label{lemma: b_! sends pseudo to pseudo}
    Let $f^\L:\mathcal{A}\to\mathcal{B}$ be a functor of tor-finite categories. 
    Let $a\in\PCoh{A}$ and $b\in\mathcal{B}^\omega$. 
    Then $b_!(a)$ defines an exact and $\mathcal{A}^\omega$-linear functor $(\mathcal{B}^\omega)^\op\to \mathcal{A}$ which is pseudo-coherent.
    If $f^\L$ is left $t$-exact up to a shift, the same holds if $a\in\Coh{A}$ is coherent.
\end{lemma}
\begin{proof}
    By definition $b_!(a)\simeq a\otimes_\mathcal{A}\mathcal{B}(-,b)\simeq \mathcal{B}(-,f^\L(a)\otimes_\mathcal{B}b)$ since $(-)\in\mathcal{B}^\omega$ is compact (hence dualizable). 
    Since $a\in\PCoh{A}$ and $b\in\mathcal{B}^\omega$ it follows that $f^\L(a)\otimes_\mathcal{B}b$ is in $\PCoh{B}$ by \autoref{lemma: f^l preserves pseudo-coherent}. 
    The equivalence $\PCoh{B}\simeq \text{PCoh}(\Fun^\ex_{\mathcal{A}^\omega}((\mathcal{B}^\omega)^\op,\mathcal{A}))$ of \autoref{remark: t-structure on functor category} implies then that $b_!(a)$ is pseudo-coherent.
    The coherent case is analogus. 
\end{proof}

The second ingredient is a particular comonad on the functor category.
\begin{remark}\label{remark: comonad}
    Let $f^\L:\mathcal{A}\to\mathcal{B}$ be a functor of tor-finite categories and let $S_\mathcal{B}\subseteq\mathcal{B}^\omega$ be generating.
    Then there exists an adjunction
    \[
    L:\mathcal{A}^{S_\mathcal{B}}\rightleftarrows \Fun^\ex_{\mathcal{A}^\omega}((\mathcal{B}^\omega)^\op,\mathcal{A}): R
    \]
    where the left adjoint $L$ is given by $L((a_s)_{s\in S})= \bigoplus_{s\in S_\mathcal{B}} s_!(a_s)$ and the right adjoint is given by $R(F)=(\text{ev}_s(F))_{s\in S_\mathcal{B}}$. 
    Let $T$ be the induced comonad and let $T^\bullet:\Delta^\op\to \Fun^\ex_{\mathcal{A}^\omega}((\mathcal{B}^\omega)^\op,\mathcal{A})$ denote the associated simplicial object.
\end{remark}

The next result explains our interest in the comonad $T$.
\begin{lemma}\label{lemma: T equivalence}
    Let $f^\L:\mathcal{A}\to\mathcal{B}$ be a functor of tor-finite categories and let $S_\mathcal{B}\subseteq\mathcal{B}^\omega$ be generating. 
    Let $F\in\Fun^\ex_{\mathcal{A}^\omega}((\mathcal{B}^\omega)^\op,\mathcal{A})$.  
    Then $F\simeq |T^{\bullet+1}F|$.
\end{lemma}
\begin{proof}
    By the Barr-Beck theorem it is sufficient to check that $R$ is conservative and that it preserves geometric realization of $R$-split simplicial objects. 
    Since colimits in $\Fun^\ex_{\mathcal{A}^\omega}((\mathcal{B}^\omega)^\op,\mathcal{A})$ are computed pointwise, each evaluation $\text{ev}_s$ preserves colimits, so that $R$ does that too. 
    For conservativity, let $\alpha: G\to G'$ be natural transformations and let $\mathcal{U}$ denote the full subcategory of $\mathcal{B}^\omega$ spanned by those $x$ such that $\alpha_x: G(x)\to G'(x)$ is an equivalence. 
    Since $G$ and $G'$ are exact functors, $\mathcal{U}$ is closed under shifts, cofibers, and retracts.
    If $\alpha_s$ is an equivalence for all $s\in S_\mathcal{B}$, then $S_\mathcal{B}\subseteq \mathcal{U}$, and the thickness of $S_\mathcal{B}$ explained in \autoref{remark: thick generation}  implies $\mathcal{B}^\omega = \mathcal{U}$. 
    That is, $R$ is conservative. 
\end{proof}

The final ingredient needed for the main result of this section are projective classes.
\begin{remark}
    Let $\mathcal{C}\in\Pr^\L_\st$ be a presentable stable category. 
    Recall that a \emph{projective class} $(\mathcal{P},\mathcal{I})$ on $\mathcal{C}$ consists of a class of objects $\mathcal{P}$ of $\mathcal{C}$ and an ideal of morphisms $\mathcal{I}$ of $\mathcal{C}$ such that: 
    \begin{enumerate*}
        \item The ideal $\mathcal{I}$ is orthogonal to $\mathcal{P}$, in the sense that $(x\to y)\in\mathcal{I}$ if and only if $\pi_0\hom(p,x)\to\pi_0\hom(p,y)$ is zero for all $p\in \mathcal{P}$.
        \item There are enough projectives, in the sense that for every $x\in\mathcal{C}$ there exists a cofibre sequence $p\to x\to y$ where $p\in\mathcal{P}$ and $(x\to y)\in\mathcal{I}$.
    \end{enumerate*} 
    Projective classes are \enquote{over-determined}. 
    Indeed, given a class of objects $\mathcal{P}$ of $\mathcal{C}$, let $\mathcal{P}\text{-null}$ denote the class of morphisms $x\to y$ in $\mathcal{C}$ such that for every $p\in\mathcal{P}$ the induced map $\pi_0\Hom_{\mathcal{C}}(p,x)\to \pi_0\Hom_{\mathcal{C}}(p,y)$ is the zero homomorphism.
    Then $\mathcal{P}\text{-null}$  is an ideal and the pair $(\mathcal{P},\mathcal{P}{\text{-null}})$ is a projective class if and only if there are enough projectives.  
\end{remark}
The next result shows how to construct projective classes.
\begin{proposition}\label{proposition: construction of projective classes}
    Let $\mathcal{C}\in\Pr^\L_\st$ be a presentable stable category, and let $S\subseteq \mathcal{C}$ be an essentially small collection of objects.
    Choose a small skeleton  $\{s_i\}_{i\in I}$ in $S$ and let $\mathcal{P}=\text{Thick}_1(\mathbb{Z}S)$ be the smallest  subcategory  of $\mathcal{C}$ closed under small coproducts, retracts and shifts that contains $S$.
    Then the pair $(\mathcal{P},\mathcal{P}\text{-null})$ is a projective class on $\mathcal{C}$.
\end{proposition}
\begin{proof}
    To show that there are enough projectives, fix $x\in\mathcal{C}$.
    Since $I$ is small and since $\mathcal{C}$ is locally small, the disjoint union $J_x:=\amalg_{(i,n)\in I\times\mathbb{Z}} \pi_0\Hom_{\mathcal{C}}(\Sigma^n s_i,x)$ is a set. 
    Define $p_x:=\amalg_{(i,n,\alpha)\in J_x}\Sigma^n s_i \in \mathcal{P}$ and let $\epsilon: p_x\to x$ be the morphism whose component on the summand indexed by $(i,n,\alpha)$ is the chosen map $\alpha: \Sigma^n s_i\to x$.
    Let $y_x$ be the cofibre of $\epsilon$.
    Applying $\hom_{\mathcal{C}}(\Sigma^n s_i,-)$ yields an exact sequence on $\pi_0$,
    \[
    \pi_0\hom(\Sigma^n s_i,p_x)\to \pi_0\hom(\Sigma^n s_i,x)\to \pi_0\hom(\Sigma^n s_i,y_x).
    \]
    Since the first arrow is surjective by construction, the second map is zero for all $(i,n)$. Since $\hom_\mathcal{C}(-,x)$ sends coproducts in the first variable to products, and retracts to retracts, vanishing (or surjectivity) for each $\Sigma^n s_i$ implies the same for every object obtained from them by shifts, coproducts, and retracts.
    Therefore $x\to y_x$ lies in $\mathcal{P}\text{-null}$. 
\end{proof}
We will apply the previous construction to a specific essentially small collection of objects. 
\begin{remark}\label{remark: projective classes and inclusion}
    Let $\mathcal{C}\in\Pr^{\L,\omega}_\st$ be a compactly generated stable category and let $S\subseteq \mathcal{C}$ be an essentially small collection of objects that thickly generates the compact objects, in that $\text{thick}(S)=\mathcal{C}^\omega$. 
    Then \autoref{proposition: construction of projective classes} applied to $\text{thick}(S)=\mathcal{C}^\omega$ determines a projective class whose projective objects are given by retracts of arbitrary coproducts of compact objects.
    We will refer to this projective class as the \emph{projective class of compact objects}.
    Notice that this projective class depends only on the compact objects of $\mathcal{C}$. 
    In other terms, if $S_1$ and $S_2$ are two essentially small collections of objects that thickly generate the compact objects, then 
    \[
    \mathcal{P}_1
    :=\text{Thick}_1(\text{thick}(S_1))
    = \text{Thick}_1(\mathcal{C}^\omega)
    = \text{Thick}_1(\text{thick}(S_2))
    =:\mathcal{P}_2
    \]
    define the same projective class.
\end{remark}

\begin{remark}
    Let $\mathcal{C}\in\Pr^\L_\st$ be a presentable stable category. 
    Let $(\mathcal{P},\mathcal{I})$ be a projective class on $\mathcal{C}$.
    Let $\mathcal{P}_1=\mathcal{P}$ and inductively define $\mathcal{P}_n$ to be the class of all retracts of objects $y$ that sit in a cofibre sequence $x\to y\to p$ with $x\in\mathcal{P}_{n-1}$ and $p\in\mathcal{P}$. 
    Define also $\mathcal{I}^n$ by the $n$-fold composites of $\mathcal{I}$. 
    Then \cite[Theorem 1.1]{christensen2013idealstriangulatedcategoriesphantoms} implies that $(\mathcal{P}_n,\mathcal{I}^n)$ is a projective class.
\end{remark}

In the setting of tor-finite categories there is a natural projective class.

\begin{notation}
    Let $f^\L:\mathcal{A}\to\mathcal{B}$ be a functor of tor-finite categories and let $S_\mathcal{A}\subseteq\mathcal{A}^\omega$ and $S_\mathcal{B}\subseteq\mathcal{B}^\omega$ be generating.
    We will denote by $S_!$ the full subcategory of $\Fun^\ex_{\mathcal{A}^\omega}((\mathcal{B}^\omega)^\op,\mathcal{A})$ spanned by $s_!(a)$ for $s\in S_\mathcal{B}$ and $a\in S_\mathcal{A}$.
    Then $S_!$ thickly generates the compact objects of  $\Fun^\ex_{\mathcal{A}^\omega}((\mathcal{B}^\omega)^\op,\mathcal{A})$.
    In particular, by picking a small skeleton, \autoref{remark: projective classes and inclusion} equips $\Fun^\ex_{\mathcal{A}^\omega}((\mathcal{B}^\omega)^\op,\mathcal{A})$ with the projective class generated by compact objects. 
    We will denote it by $(\mathcal{P}_!,\mathcal{P}_!\text{-null})$.
    Notice that this projective class, under the equivalence $\mathcal{B}\simeq\Fun^\ex_{\mathcal{A}^\omega}((\mathcal{B}^\omega)^\op,\mathcal{A})$, corresponds to the projective class of compact objects of $\mathcal{B}$. 
\end{notation}

The following observation is the crucial ingredient in the proof of the main result of this section, since it allows us to move from infinite data to finite data. 
\begin{remark}\label{remark: crucial}
    Let $f^\L:\mathcal{A}\to\mathcal{B}$ be a functor of tor-finite categories and let $S_\mathcal{A}\subseteq\mathcal{A}^\omega$ and $S_\mathcal{B}\subseteq\mathcal{B}^\omega$ be generating.
    Assume now that there exists a finite subcategory $S_\text{fin}\subseteq S_\mathcal{B}$ consisting of compact generators of $\mathcal{B}$. 
    Let $\mathcal{P}_\text{fin}$ be the smallest  full subcategory closed under small coproducts and retracts  containing $\text{thick}(\{s_!(a) \text{ for }s\in S_\text{fin},a\in S_\mathcal{A}\})$. 
    Then $\mathcal{P}_\text{fin}= \mathcal{P}_!$ by \autoref{remark: projective classes and inclusion} since the thick closure of the $s_!(a)$ for $s\in S_\text{fin}$ and $a\in S_\mathcal{A}$ agrees with the compact objects. 
\end{remark}

\begin{lemma}\label{lemma: cofibre of counit is phantom}
    Let $f^\L:\mathcal{A}\to\mathcal{B}$ be a functor of tor-finite categories and let $F:(\mathcal{B}^\omega)^\op\to\mathcal{A}$ be an exact $\mathcal{A}^\omega$-linear functor. 
    Then the cofibre of the counit $TF\to F$ is $\mathcal{P}_!$-null.
\end{lemma}
\begin{proof}
    Consider the cofibre sequence $TF\to F\to D$ of the counit $TF\to F$. 
    Then for every $a\in S_\mathcal{A}$ and $s\in S_\mathcal{B}$ there is a fibre sequence
    \[
    \hom_{\Fun^\ex_{\mathcal{A}^\omega}((\mathcal{B}^\omega)^\op,\mathcal{A})}(s_!(a),TF)
    \to
    \hom_{\Fun^\ex_{\mathcal{A}^\omega}((\mathcal{B}^\omega)^\op,\mathcal{A})}(s_!(a),F)
    \to \hom_{\Fun^\ex_{\mathcal{A}^\omega}((\mathcal{B}^\omega)^\op,\mathcal{A})}(s_!(a),D)
    \]
    in $\Sp$ which, by adjunction, is equivalent to 
    \[
    \hom_{\mathcal{A}}(a,(TF)(s))
    \to
    \hom_{\mathcal{A}}(a,F(s))
    \to \hom_{\Fun^\ex_{\mathcal{A}^\omega}((\mathcal{B}^\omega)^\op,\mathcal{A})}(s_!(a),D).
    \]
    By \autoref{remark: component of counit has a section} the component of the counit has a section, so that the first map is split surjective. 
    In particular, the second map must be zero (as it follows by the universal property of pullbacks in a stable category), and hence zero on all homotopy groups. 

    Let $\mathcal U\subseteq\mathcal \Fun^\ex_{\mathcal{A}^\omega}((\mathcal{B}^\omega)^\op,\mathcal{A})^\omega$ be the full subcategory spanned by those compact objects $P$ such that
\[
\hom_{\Fun^\ex_{\mathcal{A}^\omega}((\mathcal{B}^\omega)^\op,\mathcal{A})}(P,T F)\to \hom_{\Fun^\ex_{\mathcal{A}^\omega}((\mathcal{B}^\omega)^\op,\mathcal{A})}(P,F)
\]
is surjective on all homotopy groups. 
    The subcategory $\mathcal U$ is thick. 
    Indeed, this follows from the long exact sequences on homotopy groups associated to cofibre sequences in the variable $p$. 
    By the previous paragraph, $\mathcal U$ contains the generators $s_!(a)$ for $s\in S_\mathcal B$ and $a\in S_\mathcal A$. 
    Since these objects thickly generate $\Fun^\ex_{\mathcal{A}^\omega}((\mathcal{B}^\omega)^\op,\mathcal{A})^\omega$, it follows that $\mathcal U=\Fun^\ex_{\mathcal{A}^\omega}((\mathcal{B}^\omega)^\op,\mathcal{A})^\omega$.
    Therefore, for every $P\in\Fun^\ex_{\mathcal{A}^\omega}((\mathcal{B}^\omega)^\op,\mathcal{A})^\omega$, the map
\[
\pi_0\hom_{\Fun^\ex_{\mathcal{A}^\omega}((\mathcal{B}^\omega)^\op,\mathcal{A})}(P,F)\to \pi_0\hom_{\Fun^\ex_{\mathcal{A}^\omega}((\mathcal{B}^\omega)^\op,\mathcal{A})}(P,D)
\]
    is zero. 
    Since $\mathcal P_!=\text{Thick}_1(\Fun^\ex_{\mathcal{A}^\omega}((\mathcal{B}^\omega)^\op,\mathcal{A})^\omega)$ is obtained from compact objects by small coproducts, retracts and shifts, the same holds for every $P\in\mathcal P_!$. Thus $F\to D$ is $\mathcal P_!$-null.
\end{proof}

\begin{lemma}\label{lemma: cofibre sequence pseudo-coherent}
    Let $f^\L:\mathcal{A}\to\mathcal{B}$ be a functor of tor-finite categories and assume that $S_\text{fin}\subseteq S_\mathcal{B}$ is a finite generating set. 
    Let $F:(\mathcal{B}^\omega)^\op\to\PCoh{A}$ be an exact $\mathcal{A}^\omega$-linear functor. 
    Then for every integer $n\in\N$ there exists a cofibre sequence $F_n\to F\to D_n$ where $F_n$ is pseudo-coherent and $(F\to D_n)\in (\mathcal{P}_!\text{-null})^n$.
\end{lemma}
\begin{proof}
    Since $S_\text{fin}\subseteq\mathcal{B}^\omega$ is a finite generating set, \autoref{remark: crucial} implies that $\mathcal{P}_\text{fin}= \mathcal{P}_!$ and $\mathcal{P}_\text{fin}\text{-null}= \mathcal{P}_!\text{-null}$.
    Let $T_\text{fin}=\bigoplus_{s\in S_\text{fin}}s_!\circ\text{ev}_s$ be the comonad of \autoref{remark: comonad} restricted to $S_\text{fin}$.

    The proof now goes by induction. 
    For the base step, consider the cofibre sequence $T_\text{fin}F\to F\to D_1$.
    Since  for every $s\in S_\mathcal{B}$ the object $\text{ev}_s(F)$ belongs to $\PCoh{A}$, \autoref{lemma: b_! sends pseudo to pseudo} implies that $s_!(\text{ev}_s(F))$ is pseudo-coherent.
    Since $S_\text{fin}$ is finite, the object $T_\text{fin}F=\bigoplus_{s\in S_\text{fin}}s_!(\text{ev}_s(F))$ is again pseudo-coherent. 
    To check that the map $F\to D_1$ is $\mathcal{P}_!\text{-null}$ it suffices to check that it is $\mathcal{P}_\text{fin}\text{-null}$, but this follows by adapting  proof of \autoref{lemma: cofibre of counit is phantom}: the cofibre of the counit $T_\text{fin}F\to F$ is $\mathcal{P}_\text{fin}\text{-null}$, hence $\mathcal{P}_!\text{-null}$.
    This proves the base case.
    Assume therefore that there exists a cofibre sequence $F_n\to F\to D_n$ where $F_n$ is pseudo-coherent and $(F\to D_n)\in (\mathcal{P}_!\text{-null})^n$. 
    Since $F_n$ and $F$ factor through  $\PCoh{A}$, the same happens for $D_n$. 
    In particular, the cofibre sequence $T_\text{fin}D_n\to D_n\to D_{n+1}$ is such that $T_\text{fin}D_n$ is pseudo-coherent and $D_n\to D_{n+1}$ is in $\mathcal{P}_!\text{-null}$. 
    Consider then the composite $F\to D_n\to D_{n+1}$, which lies in $(\mathcal{P}_!\text{-null})^{n+1}$, and compute its fibre  $F_{n+1}\to F\to D_{n+1}$.
    The octahedral axiom produces then a cofibre sequence $F_n\to F_{n+1}\to T_\text{fin}D_n$, which shows that $F_{n+1}$ is pseudo-coherent, proving the claim.
\end{proof}
We can finally prove the main result. 

\begin{theorem}[Functors out of $(\mathcal{B}^\omega)^\op$]\label{theorem: functors out of Cc giovanni proof}
    Let $f^\L:\mathcal{A}\to\mathcal{B}$ be a quasi-proper functor and assume that $S_\text{fin}\subseteq S_\mathcal{B}$ is a finite generating set. 
    Assume furthermore  that the family $(\text{ev}_s)_{s\in S_\mathcal{B}}$ detects coconnective objects. 
    Then there are equivalences of categories
    \[
    \PCoh{B}\to\Fun_{\mathcal{A}^\omega}^\ex((\mathcal{B}^\omega)^\op,\PCoh{A}),\qquad
    \Coh{B}\to\Fun_{\mathcal{A}^\omega}^\ex((\mathcal{B}^\omega)^\op,\Coh{A})
    \]
    induced by the $\mathcal{A}^\omega$-linear Yoneda embedding.
\end{theorem}
\begin{proof}
    By construction there are equivalences of categories
    \[
    \PCoh{B}\to\text{PCoh}(\Fun^\ex_{\mathcal{A}^\omega}((\mathcal{B}^\omega)^\op,\mathcal{A})),
    \qquad
    \Coh{B}\to\text{Coh}(\Fun^\ex_{\mathcal{A}^\omega}((\mathcal{B}^\omega)^\op,\mathcal{A}))
    \]
    induced by the $\mathcal{A}^\omega$-linear Yoneda embedding.
    In particular, the claim will be proved if
    \[
    \text{PCoh}(\Fun^\ex_{\mathcal{A}^\omega}((\mathcal{B}^\omega)^\op,\mathcal{A})) \simeq \Fun^\ex_{\mathcal{A}^\omega}((\mathcal{B}^\omega)^\op,\PCoh{A})
    \]
    are equivalent subcategories of $\Fun^\ex_{\mathcal{A}^\omega}((\mathcal{B}^\omega)^\op,\mathcal{A})$. 
    The inclusion $(\subseteq)$ follows by \autoref{lemma: quasi-proper functor and pseudo} since $f^\L$ is quasi-proper. 
    For $(\supseteq)$, let  $F:(\mathcal{B}^\omega)^\op\to\mathcal{A}$ be an exact $\mathcal{A}^\omega$-linear functor and assume that it factors through $\PCoh{A}$.  
    By \autoref{lemma: cofibre sequence pseudo-coherent} for every integer $n\in\N$ there exists a cofibre sequence $F_n\to F\to D_n$ where $F_n$ is pseudo-coherent and $F\to D_n$ is in $(\mathcal{P}_!\text{-null})^n$.
    In particular, there are integers $i_n\in\Z$ such that  each $\susp^{i_n}F_n$ is connective. 
    Thus \autoref{lemma: assumption S} implies that every $\susp^{i_n}F_n$  may be written as a sequential colimit $\susp^{i_n}F_n\simeq \colim_{i\in \N} B^n(i)$ where each $B^n(i)$ is in $S_\mathcal{B}$. 
    Hence every $F_n$  can be written  as a sequential colimit of objects in $\mathcal{P}_!$, thus showing that $F_n\in(\mathcal{P}_!)_2$, being it realized as the cofibre of two terms in $\mathcal{P}_!$.
    Similarly, since $F$ is the geometric realization of $T^{\bullet+1}F$, it follows that  $F\in (\mathcal{P}_!)_4$. 
    Since in the cofibre sequence $F_4\to F\to D_4$ the last map is in  $(\mathcal{P}_!\text{-null})^4$, it follows that  $F$ must be a direct summand of $F_4$, and since pseudo-coherent objects are thick,  $F\in\text{PCoh}(\Fun^\ex_{\mathcal{A}^\omega}((\mathcal{B}^\omega)^\op,\mathcal{A}))$ is pseudo-coherent. 
    In particular, the restricted Yoneda 
    \[
    \PCoh{B}\to \Fun^\ex_{\mathcal{A}^\omega}((\mathcal{B}^\omega)^\op,\PCoh{A})
    \]
    is an equivalence. 
    For the coherent version, it suffices to show that an exact $\mathcal{A}^\omega$-linear functor $F:(\mathcal{B}^\omega)^\op\to\mathcal{A}$ factors through $\Coh{A}$ if and only if it factors through $\PCoh{A}$ and it is eventually coconnective. 
    Since the family $(\text{ev}_s)_{s\in S_\mathcal{B}}$ detects coconnective objects, the second equivalence follows.
\end{proof}
\begin{remark}\label{remark: improvement of theorem Cc}
    It seems unlikely to be able to choose inside $S_\mathcal{B}$ a finite generating set. 
    Fortunately, this happens in applications: in the notation of \autoref{example: moral}, one generally picks $S_\mathcal{B}$ to be \enquote{$\Perf(X)\cap\QCoh(X)_{\geq0}$} and then picks a finite set of generators $S_\text{fin}$. 
    Indeed, in this case the thick closure of  $S_\text{fin}$ consists of $\Perf(X)$ and clearly $\Perf(X)\cap\QCoh(X)_{\geq0}\subseteq\Perf(X)$.
\end{remark}

We conclude with an observation that already appeared in \cite[Lemma 3.0.7]{ben2017integral}.
\begin{remark}\label{remark: NP}
    Let $f^\L:\mathcal{A}\to\mathcal{B}$ be a functor of tor-finite categories and let $\mathcal{A}_0\subseteq\mathcal{A}$ be a small idempotent-complete stable $\mathcal{A}^\omega$-module.
    Recall that the compact pullback of $\mathcal{A}_0$ along $f^\L$ is defined as the pullback
    \[\begin{tikzcd}[cramped]
	{f^\#(\mathcal{A}_0)} & {\Fun^\ex_{\mathcal{A}^\omega}((\mathcal{B}^\omega)^\op,\mathcal{A}_0)} \\
	{\mathcal{B}} & {\Fun^\ex_{\mathcal{A}^\omega}((\mathcal{B}^\omega)^\op,\mathcal{A})}
	\arrow[from=1-1, to=1-2]
	\arrow[hook, from=1-1, to=2-1]
	\arrow[hook, from=1-2, to=2-2]
	\arrow[from=2-1, to=2-2]
	\arrow[from=2-1, to=2-2]
    \end{tikzcd}\]
    in $\Cat^\st$. 
    Since the bottom horizontal map is an equivalence, it follows that the linear Yoneda provides an equivalence $\yo_{\mathcal{A}^\omega}:f^\#(\mathcal{A}_0)\to\Fun^\ex_{\mathcal{A}^\omega}((\mathcal{B}^\omega)^\op,\mathcal{A}_0)$ of small idempotent-complete stable $\mathcal{A}^\omega$-modules.
    In particular, the $\mathcal{A}^\omega$-linear Yoneda induces equivalences
    \[
    f^\#(\PCoh{A})\to \Fun^\ex_{\mathcal{A}^\omega}((\mathcal{B}^\omega)^\op,\PCoh{A}), 
    \quad
    f^\#(\Coh{A})\to \Fun^\ex_{\mathcal{A}^\omega}((\mathcal{B}^\omega)^\op,\Coh{A}).
    \]
    These equivalences may be regarded as a generalization of \autoref{theorem: functors out of Cc giovanni proof} for \enquote{properly supported (pseudo-)coherent complexes}.
    Moreover, the proof of \autoref{theorem: functors out of Cc giovanni proof} shows that there are always inclusions $f^\#(\PCoh{A})\subseteq\PCoh{B}$ and $f^\#(\Coh{A})\subseteq\Coh{B}$ (the second one if there exists a finite family of generators detecting  coconnectivity) and that the opposite inclusions are provided by the quasi-properness of $f^\L$. 
\end{remark}

\subsection{Morphism of universal descent}\label{subsection: morphism of universal descent and h-covers}
We now introduce the theory of universal descent associated to a functor of tor-finite categories. 
\begin{notation}
    We will denote by $\Delta$ the simplex category and by $\Delta_{\leq n}$ the full subcategory of $\Delta$ spanned by those objects $[m]\in\Delta$ such that $m\leq n$. 
    We will also denote by $\Delta^+$ the augmented simplex category, and we will denote by $[\text{-}1]$ the augmented object. 
    Let $\mathcal{D}$ be a category. 
    An \emph{augmented cosimplicial object} in $\mathcal{D}$ is a functor $\mathcal{C}^\bullet:\Delta^+\to\mathcal{D}$. 
\end{notation}
\begin{notation}
    Let $h^\L:\mathcal{B}\to\mathcal{C}$ be a functor in $\CAlg^\rig(\Pr^{\L,\omega}_\st)$. 
    We define the following augmented cosimplicial objects.
    \begin{enumerate*}
    \item First of all, we can regard $\mathcal{C}$ as $\mathcal{B}$-module in $\Pr^{\L,\omega}_\st$. 
    In particular, we can form the Cech nerve of $h^\L$ inside $\Mod_\mathcal{B}(\Pr^{\L,\omega}_\st)$. 
    Objectwise, this augmented cosimplicial object is given by $[n]\mapsto \mathcal{C}\otimes_\mathcal{B}\dots\otimes_\mathcal{B}\mathcal{C}$, where the tensor product is taken $(n+1)$-times if $n\geq0$. 
    If $n=\texttt{-}1$ we set $[\texttt{-}1]\mapsto\mathcal{B}$. 
    The augmentation map is given by $h^\L:\mathcal{B}\to\mathcal{C}$. 
    We will generally regard it as a cosimplicial object in $\Pr^{\L,\omega}_\st$ and we will denote it by $\mathcal{B}^\bullet:\Delta^+\to \Pr^{\L,\omega}_\st$.  
    \item Since $h^\L$ is rigid, it restricts to a functor $h^\L:\mathcal{B}^\omega\to\mathcal{C}^\omega$. 
    Since $\mathcal{C}^\omega$ is a $\mathcal{B}^\omega$-module in $\Cat^\perf$, the same construction as before produces an augmented cosimplicial object $(\mathcal{B}^\omega)^\bullet:\Delta^+\to \Cat^\perf$.
\end{enumerate*}
    Since the Ind-completion furnishes a symmetric monoidal equivalence $\Ind:\Cat^\perf\to\Pr^{\L,\omega}_\st$ which extends to modules, the two augmented cosimplicial objects contain the same amount of data. However, it is easier to impose assumptions on the second augmented cosimplicial object. 
\end{notation}
\begin{definition}
    Let $(\mathcal{B}^\omega)^\bullet:\Delta^+\to\Cat^\perf$ be an  augmented cosimplicial object. 
    We will say that $(\mathcal{B}^\omega)^\bullet$ satisfies the \emph{Beck–Chevalley condition} if for any morphism $\alpha:[m]\to[n]$ in $\Delta^+$ the square
    \[\begin{tikzcd}[cramped]
	{(\mathcal{B}^\omega)^m} & {(\mathcal{B}^\omega)^{m+1}} \\
	{(\mathcal{B}^\omega)^n} & {(\mathcal{B}^\omega)^{n+1}}
	\arrow["{d^0}", from=1-1, to=1-2]
	\arrow["\alpha"', from=1-1, to=2-1]
	\arrow["\alpha", from=1-2, to=2-2]
	\arrow["{d^0}"', from=2-1, to=2-2]
    \end{tikzcd}\]
    is horizontally right adjointable, that is, the horizontal maps admit right adjoints and the canonical morphism $\alpha\circ d^0_*\to d^0_*\circ \alpha$ is an equivalence. 
\end{definition}
The following result is an instance of the claim made above.
\begin{lemma}\label{lemma: BC implies pushfoward}
    Let $h^\L:\mathcal{B}\to\mathcal{C}$ be a functor in $\CAlg^\rig(\Pr^{\L,\omega}_\st)$, and $(\mathcal{B}^\omega)^\bullet:\Delta^+\to\Cat^\perf$ the associated augmented cosimplicial object. 
    If $(\mathcal{B}^\omega)^\bullet$ satisfies the Beck–Chevalley condition, then the right adjoint restricts to a functor $h^\R:\mathcal{C}^\omega\to\mathcal{B}^\omega$.
\end{lemma}
\begin{proof}
    Apply the definition to $\alpha=\text{id}_{[-1]}$ and use $d^0=h^\L$ to deduce the existence of a right adjoint $r:\mathcal{C}^\omega\to\mathcal{B}^\omega$. 
    By Ind-completing it follows that $r$ must be equivalent to $h^\R$ restricted to the compact objects. 
\end{proof}
\begin{remark}
    Let  $h^\L:\mathcal{B}\to\mathcal{C}$ be a functor in $\CAlg^\rig(\Pr^{\L,\omega}_\st)$.
    Let $H:\mathcal{B}\to\mathcal{B}$ be  the induced monad.
    Then there exists a filtration $\dots\to \phi_n\to \dots \to \phi_1\to\phi_0$ of small colimit (and finite limit) preserving functors $\mathcal{B}\to\mathcal{B}$ constructed as follows.
    For every integer $n\geq0$ let $\phi_n:\mathcal{B}\to\mathcal{B}$  be the $n$-th partial totalization $\Tot_n(H^\bullet)$ of the cosimplicial object determined by $H$. 
    In other terms,  
    \[
    \phi_n \simeq\lim_{\Delta_{\leq n}}(h^\R  h^\L\rightrightarrows h^\R h^\L h^\R h^\L\mathrel{\substack{\textstyle{\rightarrow}\\[-0.6ex]
                      \textstyle{\rightarrow} \\[-0.6ex]
                      \textstyle{\rightarrow}}}\dots ).
    \]
    The natural transformations $\phi_n\to\phi_{n-1}$ are induced by restricting the limit along the inclusion $\Delta_{\leq n-1}\into\Delta_{\leq n}$ and, intuitively, they forget the top degree $n$.
\end{remark}

In particular, a careful analysis of the natural transformations involved, together with the triangle identities, shows the following result.
\begin{lemma}
    Let  $h^\L:\mathcal{B}\to\mathcal{C}$ be a functor in $\CAlg^\rig(\Pr^{\L,\omega}_\st)$ and let $I\simeq\cofib(\id_\mathcal{B}\to H)$ be augmentation ideal.
    For every integer $n\geq1$ the cofibre of $\phi_n\to\phi_{n-1}$ is equivalent to $H^{n}(I)$. 
\end{lemma}
\begin{proof}
    This is standard Dold-Kan correspondence business. 
    Indeed, for a cosimplicial object $x^\bullet$ the cofibre of $\Tot_n(x^\bullet)\to\Tot_{n-1}(x^\bullet)$ is equivalent to the $n$-suspension of the normalized piece of $x^\bullet$.  
    If the cosimplicial object comes from a monad, then  this normalized piece is equivalent to $\susp^{-n}H^{n}(I)$. 
\end{proof}

Morphisms of universal descent are defined to be those for which the filtration becomes uninteresting for larger indexes. 
\begin{definition}\label{definition: morphism of universal descent}
    Let $h^\L:\mathcal{B}\to\mathcal{C}$ be a functor in $\CAlg^\rig(\Pr^{\L,\omega}_\st)$. 
    We will say that $h^\L$ is of \emph{universal descent} if:
    \begin{enumerate*}
        \item The associated augmented cosimplicial object $(\mathcal{B}^\omega)^\bullet$ satisfies the Beck–Chevalley condition. 
        \item There exists an integer $e\geq0$ such that the identity on $\mathcal{B}$ is a retract of $\phi_e$.
    \end{enumerate*}
    We will call $e$ the \emph{exponent of $h^\L$}.
\end{definition}

\begin{remark}
    Let $h^\L:\mathcal{B}\to\mathcal{C}$ be  of universal descent. 
    Since the  associated augmented cosimplicial object $(\mathcal{B}^\omega)^\bullet:\Delta^+\to\Cat^\perf$ satisfies the Beck–Chevalley condition, it follows that $h^\R$ preserves compact objects. 
    In particular,  the filtration $\dots\to \phi_n\to \dots \to \phi_1\to\phi_0$ restricts to a filtration of exact functors $\mathcal{B}^\omega\to\mathcal{B}^\omega$. 
    Furthermore, the identity on $\mathcal{B}^\omega$ is a retract of $\phi_e|_{\mathcal{B}^\omega}$ for some exponent $e\geq0$.
\end{remark}

Our next goal is to show that for a morphism of universal descent $h^\L:\mathcal{B}\to\mathcal{C}$ the category $\mathcal{B}^\omega$ can be reconstructed from the augmented cosimplicial object $(\mathcal{B}^\omega)^\bullet:\Delta^+\to\Cat^\perf$. 
By Ind-completing, the same result holds for $\mathcal{B}^\bullet:\Delta^+\to\Pr^{\L,\omega}_\st$.

In order to prove our main result we need first a technical lemma.
\begin{lemma}\label{lemma: conservativity of universal descent}
    Let $h^\L:\mathcal{B}\to\mathcal{C}$ be of universal descent. 
    Then $h^\L$ is conservative. 
\end{lemma}
\begin{proof}
    Let $\alpha:x\to x'$ in $\mathcal{B}$ be a morphism such that $h^\L(\alpha)$ is an equivalence in $\mathcal{C}$. 
    By assumption there exists some integer $e\geq0$ such that the identity $\text{id}_\mathcal{B}$ is a retract of $\phi_e$. 
    Since equivalences are stable under retracts, to show that $\alpha$ is an equivalence it suffices to show that $\phi_e(\alpha)$ is an equivalence. 
    This follows by the explicit definition of the $\phi_n$'s. 
\end{proof}

We then deduce the following. 
\begin{proposition}\label{proposition: universal descent and limits}
    Let $h^\L:\mathcal{B}\to\mathcal{C}$ be  of universal descent.
    Then $\mathcal{B}^\bullet$ is a limit diagram.
\end{proposition}
\begin{proof}
    Since $h^\L$ is conservative by \autoref{lemma: conservativity of universal descent}, by the Barr-Beck theorem it suffices to show that the category $\mathcal{B}^\omega$ admits geometric realizations of $h^\L$-split simplicial objects, and those geometric realizations are preserved by $h^\L$. 
    Let $\mathcal{U}\subseteq\Fun(\Delta,\mathcal{B}^\omega)$ be the full subcategory spanned by those  cosimplicial objects which admit limit in $\mathcal{B}^\omega$ and  which is preserved by $h^\L$. 
    Notice that $\mathcal{U}$ is a thick subcategory (since $\mathcal{B}^\omega$ is idempotent-complete). 
    Moreover, it contains those cosimplicial objects which admit splittings (since they admit a limit, and since applying $h^\L$ furnishes cosimplicial objects which admit splittings, hence that have a limit). 
    Let now $x^\bullet:\Delta\to\mathcal{B}$ be an $h^\L$-split cosimplicial object.
    Since $x^\bullet\in\mathcal{B}$ by construction, the claim  is proved. 
\end{proof}
\begin{remark}
    Let $h^\L:\mathcal{B}\to\mathcal{C}$ be  of universal descent.
    It follows from the previous result that the canonical map
    \[
    \mathcal{B}^\omega \to \lim_{\Delta}\, (\mathcal{C}^\omega)^{\otimes_{\mathcal{B}^\omega}\bullet} 
    \]
    is an equivalence of categories. 
\end{remark}

We can also add the input of a $t$-structure. 
\begin{definition}\label{definition: t-geometric of universal descent}
    Let $h^\L:\mathcal{B}\to\mathcal{C}$ be a  functor of tor-finite categories.
    We will say that $h^\L$ is of \emph{universal descent} if:
    \begin{enumerate*}
        \item It is of universal descent in the sense of \autoref{definition: morphism of universal descent}.
        \item  It is left $t$-exact up to a shift. 
    \end{enumerate*}
\end{definition}

\begin{remark}\label{remark: t-geometric of universal descent is quasi-perfect}
    Let $h^\L:\mathcal{B}\to\mathcal{C}$ be a functor tor-finite categories of universal descent.
    Since \autoref{lemma: BC implies pushfoward} implies that $h^\R:\mathcal{R}\to\mathcal{C}$ preserves compact objects, it follows that $h^\L$ is quasi-perfect. 
    In particular, it is quasi-proper. 
\end{remark}

We conclude this section with one last observation regarding $t$-geometric functors of universal descent.  
\begin{lemma}\label{lemma: universal descent detect and preserve compact objects, connective objects and pseudo-coherent objects}
    Let $h^\L:\mathcal{B}\to\mathcal{C}$ be a functor tor-finite categories of universal descent.
    Then: 
    \begin{enumerate*}
        \item An object $x\in\mathcal{B}^\omega$ is compact if and only if $h^\L(x)\in\mathcal{C}^\omega$ is compact.
        \item An object $x\in\mathcal{B}^-$ is eventually connective if and only if $h^\L(x)\in\mathcal{C}^-$ is eventually connective.
        \item An object $x\in\PCoh{B}$ is pseudo-coherent if and only if $h^\L(x)\in\PCoh{C}$ is pseudo-coherent.
        \item An object $x\in\PCoh{B}$ is coherent if and only if $h^\L(x)\in\PCoh{C}$ is coherent.
    \end{enumerate*}
\end{lemma}
\begin{proof}
    First of all, since $h^\L:\mathcal{C}\to\mathcal{R}$ is of universal descent, there exists an exponent $e\geq0$ such that the identity on $\mathcal{C}$ is a retract of  $\phi=\phi_e:\mathcal{C}\to\mathcal{C}$ defined as above.
    With that being said,  point $(1)$ is exactly \autoref{proposition: universal descent and limits}  and the remark that follows it.
    Consider $(2)$. 
    Since $h^\L$ is right $t$-exact the implication $(\Rightarrow)$ is always true.
    For  $(\Leftarrow)$, assume that $x\in\mathcal{B}$ is such that $h^\L(x)\in\mathcal{C}_{\geq n}$ for some integer $n\in\mb{Z}$. 
    Since  $x$ is retract of  $\phi(x)$, it suffices to prove that this object is connective. 
    However, $\phi(x)$ is obtained as a finite limit
    \[
    \phi(x) = \lim_{\Delta_{\leq e}}(h^\R h^\L(x)\rightrightarrows 
            h^\R h^\L h^\R h^\L(x)
            \mathrel{\substack{\textstyle\rightarrow\\[-0.6ex]
                      \textstyle\rightarrow \\[-0.6ex]
                      \textstyle\rightarrow}}\dots)
    \]
    so it suffices to show that each term is connective. This follows since $h^\L$ is right $t$-exact and $h^\R$ is right $t$-exact up to a shift.

    Consider $(3)$. 
    Since every functor of tor-finite categories preserves pseudo-coherent objects the implication $(\Rightarrow)$ is always true, so that it suffices to prove $(\Leftarrow)$. 
    Let $x\in\mathcal{B}$ such that $h^\L(x)\in\PCoh{C}$. 
    Since $\PCoh{B}$ is stable under retracts, it suffices to show that $\phi(x)$ is pseudo-coherent. 
    This is obvious, since $\phi(x)$ is a finite limit involving powers of $h^\R h^\L$ and both $h^\R$ and $h^\L$ preserve pseudo-coherent objects.

    For $(4)$, the implication $(\Leftarrow)$ follows since $h^\L$ is left $t$-exact up to a shift. For $(\Rightarrow)$, if $x\in\mathcal{B}$ is such that $h^\L(x)\in\Coh{C}$ is coherent then the same argument above shows that $x$ is a retract of $\phi(x)$, which is a finite limit of coherent objects. 
\end{proof}

\subsection{Functors out of $\Coh{B}$}\label{subsection: Functors out of Coh}

To prove the second duality result, we will deduce it from a class of tor-finite categories for which it is automatically satisfied. 

\begin{definition}\label{definition: regular category}
    Let $\mathcal{R}$ be a tor-finite category. 
    We will say that $\mathcal{R}$ is \emph{regular} if the compact objects coincide with the coherent ones, that is $\mathcal{R}^\omega=\Coh{R}$.
\end{definition}

\begin{example}\label{example: regular not good for spectra}
    Let $A$ be a connective $\mb{E}_\infty$-ring and consider  the category of modules $\Mod_A$.
    In general, every perfect objects of $\Mod_A$, that is, every compact object, is pseudo-coherent, thus providing an inclusion $\text{Perf}(A)\subseteq\text{PCoh}(A)$. 
    If $A$ is also coconnective, then every perfect object is coconnective, thus proving  refining the previous inclusion to $\text{Perf}(A)\subseteq\text{Coh}(A)$.
    If $\Mod_A$ is regular then the other inclusion holds, and by \cite[Lemma 11.3.3.3]{Lurie-SAG} it follows that $A$ must be discrete. 

    For this reason, regular $\mb{E}_\infty$-rings are defined by just asking $\text{Coh}(A)\subseteq\text{Perf}(A)$. 
\end{example}
This implies that our definition of regular categories is not optimal, and therefore that all the arguments which we now propose suffer of the same problem.
Anyway, the our next goal is to show that for regular categories the second abstract Neeman duality is a consequence of the first one. 

\begin{lemma}\label{lemma: regular inf-categories satisfy ND2}
    Let  $f^\L:\mathcal{B}\to\mathcal{R}$ be a quasi-proper functor to a regular category.
    Assume that $S_\text{fin}\subseteq S_\mathcal{R}$ is a finite generating set such that $(\text{ev}_r)_{r\in S_\mathcal{R}}$ detects coconnective objects. 
    Then there exists an equivalence of categories
    \[
    (\mathcal{R}^\omega)^\op\to\Fun^\ex_{\mathcal{B}^\omega}(\Coh{R},\Coh{B})
    \]
    induced by the dual $\mathcal{B}^\omega$-Yoneda embedding. 
\end{lemma}
\begin{proof}
    Since $\mathcal{R}^\omega=\Coh{R}$, the duality $\Delta_\mathcal{R}:(\mathcal{R}^\omega)^\op\to\mathcal{R}^\omega$ provides a commutative square
    \[\begin{tikzcd}[cramped]
	{(\mathcal{R}^\omega)^\op} & {\Fun_{\mathcal{B}^\omega}^\ex(\Coh{R},\Coh{B})} \\
	{\mathcal{R}^\omega} & {\Fun_{\mathcal{B}^\omega}^\ex((\mathcal{R}^\omega)^\op,\Coh{B})}
	\arrow["{\tilde{\yo}}", from=1-1, to=1-2]
	\arrow["{\Delta_\mathcal{R}}"', from=1-1, to=2-1]
	\arrow["{-\circ \Delta_\mathcal{R}}", from=1-2, to=2-2]
	\arrow["\yo"', from=2-1, to=2-2]
\end{tikzcd}\]
    Since the horizontal bottom map is an equivalence by \nameref{theorem: functors out of Cc giovanni proof} and the vertical maps are equivalence, the horizontal top map is an equivalence. 
\end{proof}

In order to prove the main result of this section we need some preliminary results. 
\begin{lemma}\label{lemma: dual yoneda restricts}
    Let $f^*:\mathcal{A}\to\mathcal{B}$ be a quasi-proper functor. 
    Then the dual $\mathcal{A}^\omega$-Yoneda embedding restricts to a functor
    \[
    \tilde{\yo}:(\mathcal{B}^\omega)^\op\to\Fun_{\mathcal{A}^\omega}^\ex(\PCoh{B},\PCoh{A}).
    \]
    The same is true with $\Coh{-}$ in place of $\PCoh{-}$ in the first  argument of the functor category.
\end{lemma}
\begin{proof}
    Let $x\in\mathcal{B}^\omega$ be a compact object and let $y\in\mathcal{B}$ be (pseudo-)coherent. 
    We wish to show that $\tilde{\yo}(x)(y)=\mathcal{B}(x,y)$ is (pseudo-)coherent. 
    Since $x$ is compact, this object can be  identified with $\mathcal{B}(x,y)= f^\R\IHom_\mathcal{B}(x,y) \simeq f^\R(x^\vee\otimes_\mathcal{B}y)$.
    Since  \autoref{lemma: pcoh and coh are Cc submodules} implies that the tensor product $x^v\otimes_\mathcal{B}y$ is (pseudo-)coherent, the quasi-properness of $f^\L$ concludes. 
\end{proof}

\begin{lemma}\label{lemma: universal descent to regular implies compact subsets of coehrent}
    Let $h^\L:\mathcal{B}\to\mathcal{R}$ be a functor of universal descent of tor-finite categories and assume that $\mathcal{R}$ is regular. 
    Then  $\mathcal{B}^\omega\subseteq\Coh{B}$. 
\end{lemma}
\begin{proof}
    Since $\mathcal{B}$ is tor-finite, $\mathcal{B}^\omega\subseteq\PCoh{B}$. 
    Let $x\in\mathcal{B}^\omega$. Since $h^\L(x)\in\mathcal{R}^\omega$ is coherent by the regularity of $\mathcal{R}$, \autoref{lemma: universal descent detect and preserve compact objects, connective objects and pseudo-coherent objects} implies that $x\in\Coh{B}$ is coherent. 
    Thus $\mathcal{B}^\omega\subseteq\Coh{B}$. 
\end{proof}

\begin{lemma}\label{lemma: retract of a representable}
    Let $\mathcal{A}\in\CAlg(\Cat^\perf)$ and let $\mathcal{B}_1,\mathcal{B}_2\in\Mod_\mathcal{A}(\Cat^\perf)$ with $\mathcal{C}\subseteq\mathcal{B}$ a full subcategory.
    Let $\alpha: \mathcal{B}^\op\to\mathcal{A}$ be an  exact and $\mathcal{A}$-linear functor.
    Assume that $\alpha$ is a retract of $\mathcal{B}(-,x)$ for $x\in\mathcal{C}$. 
    Then $\alpha$ is represented by an object of $\mathcal{C}$.
\end{lemma}
\begin{proof}
    Let $\alpha \xrightarrow{i} {\mathcal{B}(-,x)} \xrightarrow{r} \alpha$ be the given retract of functors $\mathcal{B}^\op\to\mathcal{A}$.
    Since the composition $i\circ r:\mathcal{B}(-,x)\to\mathcal{B}(x,-)$ is idempotent, there exists an idempotent morphism $f:x\to x$ in $\mathcal{C}$ such that $\mathcal{B}(f,-)\simeq i\circ r$.
    Since $\mathcal{C}$ is idempotent-complete, there exists a retract $y \xrightarrow{j} x \xrightarrow{q} y$, that is $f\simeq j\circ q$ and $q\circ j\simeq \id_y$ with $y\in\mathcal{C}$.
    By Yoneda this retract is sent to the one of $\alpha$, thus showing $\alpha\simeq \mathcal{B}(-,y)$ for $y\in\mathcal{C}$.
\end{proof}

We can now prove the main result of this section. 
\begin{theorem}[Functors out of $\Coh{B}$]\label{theorem: functors out of coh proof}
    Let $f^\L:\mathcal{A}\to\mathcal{B}$ be quasi-proper and assume that $\mathcal{B}$ admits a  morphism of universal descent $h^\L:\mathcal{B}\to\mathcal{R}$ to a regular category such that there exists a finite generating set $S_\text{fin}\subseteq S_\mathcal{R}$ for which $(\text{ev}_r)_{r\in S_\mathcal{R}}$ detects coconnective objects on $\mathcal{A}$.
    Then there exists an equivalence of categories
    \[
    (\mathcal{B}^\omega)^\op\to\Fun^\ex_{\mathcal{A}^\omega}(\Coh{B},\Coh{A})
    \]
    induced by the dual $\mathcal{A}^\omega$-linear Yoneda embedding.
\end{theorem}
\begin{proof}
    Since $h^\L:\mathcal{B}\to\mathcal{R}$ is of universal descent, there exists an exponent $e\geq0$ such that the identity on $\mathcal{B}$ is a retract of $\phi=\phi_e:\mathcal{B}\to\mathcal{B}$, defined as in the previous section.
    Notice that $\phi$ restricts to a functor $\Coh{B}\to\Coh{B}$ since $h^\L$ is left $t$-exact up to a shift. 
    Consider now the $\mathcal{A}$-linear Yoneda embedding 
    \[
    \tilde{\yo}: 
    \Coh{B}^\op \to\Fun^\ex_{\mathcal{A}^\omega}(\Coh{B},\mathcal{A})
    \]
    Restricting $\tilde{\yo}$ to $(\mathcal{B}^\omega)^\op$ and using \autoref{lemma: dual yoneda restricts} produces a fully faithful functor 
    \[
    \tilde{\yo}:(\mathcal{B}^\omega)^\op\to\Fun^\ex_{\mathcal{A}^\omega}(\Coh{B},\Coh{A}).
    \]
    In particular, to prove the claim it suffices to show that this functor is essentially surjective.
    Notice first of all that a computation with adjoints shows that the diagram
    \[\begin{tikzcd}[cramped]
	{(\mathcal{B}^\omega)^\op} & {\Fun_{\mathcal{A}^\omega}^\ex(\Coh{B},\Coh{A})} \\
	{(\mathcal{R}^\omega)^\op} & {\Fun_{\mathcal{A}^\omega}^\ex(\Coh{R},\Coh{A})}
	\arrow["{\tilde{\yo}}", from=1-1, to=1-2]
	\arrow["{(h^\L)^\op}"', from=1-1, to=2-1]
	\arrow["{-\circ h^\R}", from=1-2, to=2-2]
	\arrow["{\tilde{\yo}}"', from=2-1, to=2-2]
	\arrow["\simeq", from=2-1, to=2-2]
    \end{tikzcd}\]
    commutes.   
    Since $h^\L:\mathcal{B}\to\mathcal{R}$ is a  morphism of universal descent  such that there exists a finite generating set $S_\mathcal{R}\subseteq\mathcal{R}^\omega$ for which $(\text{ev}_r)_{r\in S_\mathcal{R}}$ detects coconnective objects on $\mathcal{A}$,  the horizontal bottom map is an equivalence  by \autoref{lemma: regular inf-categories satisfy ND2}.
    Let now $\alpha:\Coh{B}\to\Coh{A}$ be an exact and $\mathcal{A}^\omega$-linear functor.
    The above diagram shows that the composition $\alpha\circ h^\R$ is represented by a compact object $y\in(\mathcal{R}^\omega)^\op$ and the same happens for $\alpha\circ\phi$. 
    Indeed, since $h^\L$ is quasi-perfect, the abstract Grothendieck-Neeman duality (\autoref{theorem: Grothendieck Duality} and \autoref{example: internal left adjoint for rigidly-compactly generated}) implies the existence of adjunctions $h_{(1)}\dashv h^\L\dashv h^\R$ where the functor $h_{(1)}$ preserves compact objects (having a right adjoint that preserves filtered colimits), so  that the equivalences
    \[\begin{aligned}
    \alpha\circ (h^\R h^\L)^n
    &\simeq \alpha\circ h^\R \circ h^\L\circ( h^\R\circ h^\L)^{n-1}\\
    &\simeq \tilde{\yo}(y) \circ h^\L\circ (h^\R\circ h^\L)^{n-1}\\
    &\simeq f^\R h^\R\IHom_\mathcal{R}(y,  h^\L\circ (h^\R\circ h^\L)^{n-1}(-))
    \end{aligned}\]
    valid for every integer $1\leq n\leq e$, imply that $\alpha\circ(h^\R h^\L)^n$ is represented by $(h^\L h_{(1)})^n(y)$. 
    It follows from  the explicit definition of $\phi=\phi_e$  that  $\alpha\circ\phi$ is a finite limit of representable functors with compact representing object, and hence it must be representable (since the enriched Yoneda is exact, being the domain and codomain stable) by a compact object. 

    Now, since $\alpha$ is a retract of $\alpha\circ\phi$, that is, of a functor represented by a compact object, an application of the dual of \autoref{lemma: retract of a representable} implies that  $\alpha$ is represented by a compact object, thus proving essential surjectivity.
\end{proof}

\section{Examples}\label{section: examples}

We now present some applications, ranging from modules to  quasi-compact quasi-separated spectral algebraic spaces, passing through schemes.

\subsection{Modules}\label{subsection: modules}
We start by discussing a broad class of examples given by module categories.
\begin{remark}
    Let $\mathcal{C}\in\CAlg(\Pr^\L_\st)$ be a symmetric monoidal category and let $x\in\CAlg(\mathcal{C})$ be a commutative algebra. 
    Let $\Mod_x(\mathcal{C})$ denote the category of $x$-modules in $\mathcal{C}$ and regard it as a symmetric monoidal category under the relative tensor product $\otimes_x$. 
    Tensoring with $x$ determines a functor $x\otimes-:\mathcal{C}\to\Mod_x(\mathcal{C})$ which admits a conservative right adjoint $\text{res}:\Mod_x(\mathcal{C})\to\mathcal{C}$. 
\end{remark}
We now show that module objects are stable under many constructions introduced so far. 
Let us first analyze compact generation.
\begin{lemma}\label{lemma: rigidity of modules}
    Let $\mathcal{C}\in\CAlg^\rig(\Pr^{\L,\omega}_\st)$ be a rigid algebra  and let  $x\in\mathcal{C}$ be a commutative algebra. 
    Then the category $\Mod_x(\mathcal{C})\in\CAlg^\rig(\Pr^{\L,\omega}_\st)$ of $x$-module objects in $\mathcal{C}$ is a rigid algebra. 
    Moreover, $x\otimes-:\mathcal{C}\to\Mod_x(\mathcal{C})$ is rigid.
\end{lemma}
\begin{proof}
    Since $\mathcal{C}$ is presentable stable and the forgetful functor $\text{res}:\Mod_x(\mathcal{C})\to\mathcal{C}$ preserves and detects all limits and colimits, it follows that $\Mod_x(\mathcal{C})$ is presentable stble. 
    It also follows that the relative tensor product commutes with small colimits in each variable, making $\Mod_x(\mathcal{C})\in\CAlg(\Pr^\L_\st)$. 
    Regarding compact generation, it suffices to notice that the essential image of the compact generators of $\mathcal{C}$ compactly generates $\Mod_x(\mathcal{C})$.
    Finally, since tensoring with $x$ is a symmetric monoidal functor, we deduce that dualizable and compact objects in $\Mod_x(\mathcal{C})$ coincide. 
    Thus $\Mod_x(\mathcal{C})\in\CAlg^\rig(\Pr^{\L,\omega}_\st)$.
    The statement about rigidity of $x\otimes -$ is then trivial.
\end{proof}
Let us now analyze the $t$-structure.
\begin{remark}
    Let $\mathcal{C}\in\CAlg(\Pr^{\L,\omega}_\st)$ be a compactly generated stable category with a compatible symmetric monoidal structure.
    Assume that the monoidal unit of $\mathcal{C}$ is a compact generator. 
    Then \autoref{remark: summary 1} implies that the monoidal unit defines a right-complete presentable compactly generated $t$-structure on $\mathcal{C}$ which is compatible with $\otimes$.
    Furthermore, every compact object is eventually connective $\mathcal{C}^\omega\subseteq\mathcal{C}^-$.
    Left completeness of this $t$-structure is not automatic.
\end{remark}
The previous observation propagates to module categories.
\begin{lemma}\label{lemma: tensor t-structure for module categories}
    Let $\mathcal{C}\in\CAlg(\Pr^{\L,\omega}_\st)$ be a compactly generated stable category with a compatible symmetric monoidal structure and assume that the monoidal unit of $\mathcal{C}$ is a compact generator. 
    Regard $\mathcal{C}$ with the induced $t$-structure. 
    Let $x\in\mathcal{C}$ be a commutative algebra object. 
    Then:
    \begin{enumerate*}
        \item The category $\Mod_x(\mathcal{C})$ can be equipped with a right-complete presentable compactly generated $t$-structure compatible with $\otimes$ such that $x\otimes-:\mathcal{C}\to\Mod_x(\mathcal{C})$ is right $t$-exact.
        \item If $\mathcal{C}^\omega\subseteq\Tor{C}$, then the same happens for $\Mod_x(\mathcal{C})$.
    \end{enumerate*}
    Assume that $x$ is connective in $\mathcal{C}$. Then:
    \begin{enumerate*}
        \item[(3)] The restriction $\text{res}:\Mod_x(\mathcal{C})\to\mathcal{C}$ is $t$-exact. In particular, it detects connective and coconnective objects.
        \item[(4)] If the $t$-structure on $\mathcal{C}$ is left complete, then  the $t$-structure on $\Mod_x(\mathcal{C})$ is left complete. 
    \end{enumerate*}
\end{lemma}
\begin{proof} 
    Let $\mb{1}_\mathcal{C}$ be the monoidal unit of $\mathcal{C}$. 
    since $\mb{1}_\mathcal{C}$ is a generator, the adjunction $x\otimes-\dashv\text{res}$, coupled with fact that $\text{res}$ is conservative, implies that $x$ is a compact generator of $\Mod_x(\mathcal{C})$. 
    The previous remark implies then point $(1)$.
    Consider $(2)$. 
    By  thickness it suffices to show that the monoidal unit $x$ of $\Mod_x(\mathcal{C})$ is of finite tor-dimension. 
    Since it is connective in the $t$-structure on the module category, it suffices to show that for every $y\in(\Mod_x(\mathcal{C}))_{\leq0}$ the object $x\otimes_x y$ is eventually coconnective. 
    But  $x\otimes_x y\simeq y$.

    Assume now that $x$ is connective in $\mathcal{C}$.
    Assertion $(3)$ regarding the $t$-exactness of the restriction functor is clear and the detection statement follows from conservativity.
    Consider $(4)$ and assume therefore that the $t$-structure on $\mathcal{C}$ is left complete. 
    Since $\Mod_x(\mathcal{C})$  is presentable and the $t$-structure compatible with filtered colimits, for the $t$-structure being left complete coincides with being weakly left complete. 
    Now since $\text{res}:\Mod_x(\mathcal{C})\to\mathcal{C}$ is $t$-exact, it follows that $\text{res}_x(\tau_{\leq n}M)\simeq \tau_{\leq n}\text{res}_x(M)$, and the conservativity of $\text{res}_x$ allows to conclude that  the canonical map $M\to \lim_{n\in \Z}\tau_{\leq n}M$ is an  equivalence  for every $x$-module $M$, being it an equivalence because of the left completeness  of $\mathcal{C}$.
\end{proof}
\begin{corollary}
    Let $\mathcal{C}\in\CAlg^\rig(\Pr^{\L,\omega}_\st)$ be a rigid algebra and assume that the monoidal unit of $\mathcal{C}$ is a compact generator such that the pair $(\mathcal{C},\{\mb{1}\})$ determines a tor-finite category. 
    If $x\in\CAlg(\mathcal{C})\cap \mathcal{C}_{\geq0}$ is a connective commutative algebra, then the pair $(\Mod_x(\mathcal{C}),\{x\})$ determines a tor-finite category for which $x\otimes-:\mathcal{C}\to\Mod_x(\mathcal{C})$ is a functor of tor-finite categories.
\end{corollary}
\begin{notation}
    In the same setup of the previous corollary, we will denote by $\text{Perf}(x)$,  $\text{PCoh}(x)$ and  $\text{Coh}(x)$ the full subcategories of compact, pseudo-coherent and coherent objects of $\Mod_x(\mathcal{C})$.
\end{notation}
Let us explicitly note that coherence is not preserved by taking module objects. 
\begin{remark}
    Assume that $\mathcal{C}$ is coherent and let $x\in\mathcal{C}$ be a commutative algebra object. 
    Then it is \emph{not} always true that $\Mod_x(\mathcal{C})$ is coherent.
    Indeed, if $R$ is a connective commutative ring spectrum, then $\Mod_R=\Mod_R(\Sp)$ is coherent if and only if $R$ is a coherent ring in the sense of \cite[Definition 7.2.4.16]{Lurie-HA}, whereas $\Sp\simeq\Mod_\mb{S}(\Sp)$ is coherent. 
\end{remark}
To avoid awkward terminology, let us give the following.
\begin{definition}
    Let $\mathcal{C}\in\CAlg(\Pr^{\L,\omega}_\st)$ be a compactly generated stable category and assume that the monoidal unit is a compact generator. 
    We will say that a commutative algebra object $x\in\CAlg(\mathcal{C})$ in  $\mathcal{C}$ is \emph{coherent} if the module category $\Mod_x(\mathcal{C})$ is coherent when equipped with the $t$-structure induced by the monoidal unit.
\end{definition}
We can also adopt a relative point of view.
\begin{remark}
    Let $\mathcal{C}\in\CAlg^\rig(\Pr^{\L,\omega}_\st)$ be a rigid algebra.  
    Let $f:x\to y$ be a map of commutative algebra objects in $\mathcal{C}$ and regard $y$ as an $x$-module. 
    Then there exists a commutative diagram
    \[\begin{tikzcd}[cramped]
	{\Mod_x(\mathcal{C})} && {\Mod_y(\mathcal{C})} \\
	& {\mathcal{C}}
	\arrow["{f^*\simeq y\otimes_x-}", from=1-1, to=1-3]
	\arrow["{x\otimes-}", from=2-2, to=1-1]
	\arrow["{y\otimes-}"', from=2-2, to=1-3]
    \end{tikzcd}\]
    where the functor $f^\L\simeq y\otimes_x-$ is called the \emph{extension of scalars}. 
    The right adjoint $f^\R=\text{res}_f$ of $f^\L$ is given by the forgetful functor, and is called the \emph{restriction along $f$}. 
    Since the functor $f^\L$ is clearly symmetric monoidal and colimit preserving, it is rigid (since the module categories are rigid by \autoref{lemma: rigidity of modules}). 
    Assume now that the monoidal unit is a compact generator and equip all the categories with the induced $t$-structure. 
    Then $f^\L$ is also right  $t$-exact. 
    Moreover, if $y\in(\Mod_x(\mathcal{C}))_{\geq0}$ is connective, then $\text{res}_f$ is also $t$-exact. 
    In  particular, $f^\L$  is a functor of tor-finite categories.
\end{remark}

\begin{definition}
    Let $\mathcal{C}\in\CAlg^\rig(\Pr^{\L,\omega}_\st)$ be a rigid algebra and assume that the monoidal unit of $\mathcal{C}$ is a compact generator such that the pair $(\mathcal{C},\{\mb{1}\})$ determines a tor-finite category. 
    Let $f:x\to y$ be a map of connective commutative algebra objects. 
    We will say that $f$ is \emph{quasi-proper} if $\text{res}_f(y)$ is a pseudo-coherent object of $\Mod_x(\mathcal{C})$.
\end{definition}
\begin{lemma}
    Let $\mathcal{C}\in\CAlg^\rig(\Pr^{\L,\omega}_\st)$ be a rigid algebra and assume that the monoidal unit of $\mathcal{C}$ is a compact generator such that the pair $(\mathcal{C},\{\mb{1}\})$ determines a tor-finite category. 
    Let $f:x\to y$ be a quasi-proper map between coherent objects. 
    Then $y\otimes_x-$ is quasi-proper.
\end{lemma}
\begin{proof}
    Let $S_y$ denote the smallest full subcategory of $\Mod_y(\mathcal{C})_{\geq0}$ closed under finite colimits and extensions and containing $y$.  
    Since $\text{res}_f(y)$ is a pseudo-coherent in $\Mod_x(\mathcal{C})$, thickness implies that  every object  of $S_y$ is sent to a pseudo-coherent object of $x$.
    Let now $c\in\text{PCoh}(y)$ be pseudo-coherent, and assume, without loss of generality, that it is connective. 
    Then \autoref{lemma: assumption S} implies that $c$ can be written as a geometric realization of objects $c_n\in S_y$. 
    It follows from above that  each $\text{res}_f(c_n)$ is  pseudo-coherent in $x$. 
    The $t$-exactness of the restriction implies furthermore that each $\text{res}_f(c_n)$ is connective.
    In particular, $\text{res}_f(c)$ can be written as a geometric realization of pseudo-coherent conenctive objects. 
    Apply then \autoref{lemma: 0-connected pseudo-coheren are closed under the formation of geometric realization} to deduce  $\text{res}_f(c)\in\text{PCoh}(x)$.
\end{proof}

We deduce the following consequence of \autoref{theorem: functors out of Cc giovanni proof}.
\begin{corollary}
    Let $\mathcal{C}\in\CAlg^\rig(\Pr^{\L,\omega}_\st)$ be a rigid algebra and assume that the monoidal unit of $\mathcal{C}$ is a compact generator such that the pair $(\mathcal{C},\{\mb{1}\})$ determines a tor-finite category. 
    Let $f:x\to y$ be  quasi-proper.
    Then the restricted Yoneda embedding induces  equivalences
    \[
    \text{PCoh}(y)\to \Fun^\ex_{\text{Perf}(x)}(\text{Perf}(y)^\op,\text{PCoh}(x)),
    \qquad
    \text{Coh}(y)\to \Fun^\ex_{\text{Perf}(x)}(\text{Perf}(y)^\op,\text{Coh}(x))
    \]
    of categories.
\end{corollary}
\begin{proof}
    It follows from \autoref{theorem: functors out of Cc giovanni proof} after having identified the evaluation at the compact generator $y$ with the restriction along $f$.
\end{proof}
We deduce the following.
\begin{example}
    The above theorem is particularly useful when applied to the category of spectra $\Sp$ with the standard $t$-structure. 
    If $f:A\to B$ is a quasi-proper map between connective $\mb{E}_\infty$-rings, then it provides equivalences
    \[
    \text{PCoh}(B)\to \Fun^\ex_{\text{Perf}(A)}(\text{Perf}(B)^\op, \text{PCoh}(A)), \qquad
    \text{Coh}(B)\to \Fun^\ex_{\text{Perf}(A)}(\text{Perf}(B)^\op, \text{Coh}(A))
    \]
    of categories. 
    Furthermore, when $A$ and $B$ are classical, the map $f$ is quasi-proper if and only if the derived pushforward is quasi-proper in  the sense of Lipman-Neeman \cite[Page 3]{lipman2007quasi}.
\end{example}

\subsection{Schemes}\label{subsection: schemes}
We now  study  the case of schemes.
Let $X$ be a quasi-compact quasi-separated scheme and let us denote by $\QCoh(X)$ the derived category of quasi-coherent sheaves on $X$. 
Then
\[
\QCoh(X) = \lim_{\Spec(R)\subseteq X} \Mod_{HR}
\]
where the limit is taken over the poset of open affine subsets of $X$. Here $\Mod_{HR}$ is the category of modules over the Eilenberg-MacLane spectrum $HR$, or equivalently the (unbounded) derived category of $R$-modules. 
\begin{remark}
    Let $X$ be a quasi-compact quasi-separated scheme. 
    It  follows formally that $\QCoh(X)\in\CAlg(\Pr^\L_\st)$  under the (derived) tensor product.
    Furthermore, the monoidal unit $\mathcal{O}_X$ of $\QCoh(X)$ is compact.  
    In particular, every dualizable object is compact. 
    Since dualizable objects coincide with the perfect objects $\text{Perf}(X)$, that is with those complexes which, locally, are quasi-isomorphic to bounded chain complexes of finitely generated projective modules (see \cite[\href{https://stacks.math.columbia.edu/tag/08JJ}{Lemma 08JJ}]{stacks-project} and \cite[\href{https://stacks.math.columbia.edu/tag/0FPV}{Lemma 0FPV}]{stacks-project}) then \cite[\href{https://stacks.math.columbia.edu/tag/09M1}{Proposition 09M1}]{stacks-project} implies  that any perfect object is compact, so that  dualizable and compact objects coincide. 
    Thus $\QCoh(X)\in\CAlg^\rig(\Pr^\L_\st)$.
\end{remark}
The compact generation of $\QCoh(X)$ is well understood.   
\begin{remark}
    A beautiful result of Bondal and van den Bergh \cite[Theorem 3.1.1]{bondal2002generators} shows that $\QCoh(X)$ is generated by a single compact object.
\end{remark}
We  continue the ideas of \autoref{example: moral}.
\begin{remark}
    Let $X$ be a quasi-compact quasi-separated scheme and equip $\QCoh(X)$ with the standard t-structure.
    Every quasi-compact quasi-separated scheme $X$ satisfies connective perfect generation by \cite[Proposition 9.6.1.2]{Lurie-SAG}. 
    This means that the connective half  $\QCoh(X)_{\geq0}$ is compactly generated by connective perfect complexes $\Perf(X)_{\geq0}= \Perf(X)\cap\QCoh(X)_{\geq0}$. 
    Furthermore, since $\QCoh(X)_{\geq0}\into\QCoh(X)$ is symmetric monoidal, it follows that $\Perf(X)_{\geq0}$ inherits a symmetric monoidal structure. 
    An application of \autoref{lemma: conditions of t-structure} and \autoref{remark: lurie construction of t} shows that the standard $t$-structure on $\QCoh(X)$ is  accessible, compatible with filtered colimits and right complete, as well as being compatible with $\otimes$.
    It is also left complete\footnote{A proof may be given via \autoref{lemma: excellent t-structure are preserved by equivalence class}.}.
    Finally, the compact generator can be choosen to be connective.
\end{remark}
\begin{remark}
    Let $X$ be a quasi-compact quasi-separated scheme. 
    Since the monoidal unit is coconnective, an application of  \autoref{lemma: tor amplitude vs classical} shows that being of finite tor-dimension in the  sense of \autoref{definition: tor dimension} reduces to the classical notion of being of finite tor-dimension. 
    In particular, every perfect object on $X$ is of finite tor-dimension. 
\end{remark}

We deduce the following. 
\begin{corollary}
    Let $X$ be a quasi-compact quasi-separated scheme. 
    The pair $(\QCoh(X), \Perf(X)_{\geq0})$ determines a tor-finite category. 
    Moreover, if $G\in\Perf(X)_{\geq0}$ is a single compact connective generator, then $\{G\}\subseteq\Perf(X)_{\geq0}$ is a finite generating family. 
\end{corollary}

We wish now to compute our (pseudo)-coherent objects (which we will denote by $\text{Coh}(X)\subseteq \text{PCoh(X)}$) for the standard $t$-structure on $\QCoh(X)$. 

\begin{remark}
    Let $X$ be a quasi-compact quasi-separated scheme. 
    Let $G$ be the single compact generator of $\QCoh(X)$. 
    Then one can show that $G$ satisfies the assumptions of \autoref{proposition: PCoh = C-c}. 
    It follows that  the pseudo-coherent objects coincide with the bounded pseudo-compact objects.
    Furthermore, if $X$ is coherent, then the standard $t$-structure on $\QCoh(X)$ is coherent in the sense of \autoref{definition: coherent t-structure}. See, for example, \cite[Proposition 8.7]{scholze2025sixfunctorformalisms}.
    Hence \autoref{proposition: pseudo coherent and pi_n} allows us to compute explicitly $\text{Coh}(X)\subseteq \text{PCoh(X)}$ for a coherent scheme $X$: they turn out to be the classical $D^b_{\text{coh}}(X)\subseteq D^-_{\text{coh}}(X)$.   
\end{remark}

Let us now consider a map $f:X\to Y$ between quasi-compact quasi-separated schemes. 
Consider the pullback functor $f^*:\QCoh(Y)\to\QCoh(X)$. 
This functor is rigid and  right $t$-exact with respect to the standard $t$-structure. 
Since the pushforward $f_*:\QCoh(X)\to\QCoh(Y)$ is of finite cohomological dimension, it follows that $f^*$ may be identified as  a  functor of tor-finite categories.
\begin{remark}
    In general the pushforward $f_*$ does \emph{not} preserve compact objects, nor does it send pseudo-coherent objects to pseudo-coherent objects. 
    It does when it is \emph{quasi-perfect} or \emph{quasi-proper}, respectively, in the schematic sense.
\end{remark}
We need the following result (see \cite[Theorem 1.10]{alonso2023relative}).
\begin{lemma}\label{lemma: yoneda for schemes}
    Let $f:X\to Y$  be a morphism of quasi-compact quasi-separated schemes, and let $G$ be a compact generator for $\QCoh(X)$. Then $\QCoh(X)(G,-):\QCoh(X)\to\QCoh(Y)$ detects and preserves eventually being connective and coconnective objects. 
\end{lemma}
\begin{proof}
    Let $M\in\QCoh(X)$ and assume that it is connective (respectively, coconnective). 
    Since  $\QCoh(X)(G,-)\simeq f_*\IHom_{\QCoh(X)}(G,-)\simeq f_*(G^\vee\otimes)$ is $t$-exact up to a shift (since $f_*$ is and since $G^\vee$ is of finite tor-dimension), it follows that $\QCoh(X)(G,M)$ is connective (respectively, coconnective) up to a shift.
    Assume conversely that $\QCoh(X)(G,M)$ is connective (respectively, coconnective) up to a shift. 
    An argument by thickness implies the same claim with  $G$ replaced by any perfect complex, which  implies the claim.
\end{proof}

We can now apply \autoref{theorem: functors out of Cc giovanni proof} to deduce the following result. 
\begin{corollary}
    Let $f:X\to Y$ be a quasi-proper map of quasi-compact quasi-separated schemes. Assume that $Y$ is coherent. 
    Then there are equivalences of categories
    \[
    \text{PCoh}(X)\to \Fun^\ex_{\text{Perf}(Y)}(\text{Perf}(X)^\op,\text{PCoh}(Y)), 
    \qquad 
    \text{Coh}(X)\to \Fun^\ex_{\text{Perf}(Y)}(\text{Perf}(X)^\op,\text{Coh}(Y)).
    \]
    induced by the $\text{Perf}(Y)$-Yoneda embedding.  
\end{corollary}
\begin{proof}
   Since $f^*:\QCoh(Y)\to\QCoh(X)$ is a functor of tor-finite categories and since $\QCoh(X)$ may be generated by a singe connective object, an application of \autoref{remark: improvement of theorem Cc} and \autoref{lemma: yoneda for schemes} (after having identified the evaluation at the compact generator with the linear Yoneda) concludes. 
\end{proof}
We  obtain a more classical result by means of  Kiehl’s Finiteness Theorem \cite[Theorem 2.9']{kiehl1972descente}. 
It shows that every proper pseudo-coherent map is quasi-proper, and, in particular, it implies that every finite-type separated map $f : X \to Y$ over a noetherian base is quasi-proper if and only if it is proper. This observation, coupled with the previous result, proves the following  generalization of \cite[Corollary 0.5]{neeman2025triangulatedcategoriessinglecompact}.
\begin{corollary}\label{corollary: ND1 for schemes}
   Let $f:X\to Y$ be a proper map and assume that $Y$ is noetherian. 
   Then there are equivalences of categories
    \[
    D^-_{\text{coh}}(X)\to \Fun^\ex_{\text{Perf}(Y)}(\text{Perf}(X)^\op,D^-_{\text{coh}}(Y)), \qquad D^b_{\text{coh}}(X)\to \Fun^\ex_{\text{Perf}(Y)}(\text{Perf}(X)^\op,D^b_{\text{coh}}(Y)).
    \]
    induced by  the $\text{Perf}(Y)$-Yoneda embedding.  
\end{corollary}

We now turn our attention to the second duality result. 
\begin{remark}
    Let $X$ be a noetherian scheme. 
    Recall that a \emph{regular alteration} of $X$ is a proper surjective morphism $h : R\to X$ such that:
    \begin{enumerate*}
        \item $R$ is regular and finite dimensional.
        \item There is a dense open set $U\subseteq X$ over which $h$ is finite.
    \end{enumerate*}
    A great number  of examples of regular alterations were provided by de Jong in \cite{de1996smoothness} and \cite{jong1997families}:  every separated of finite type scheme over an excellent scheme of dimension $\leq 2$ admits a regular alteration. 
\end{remark}
\begin{lemma}
    Let $X$ be a noetherian scheme and let $h : R\to X$ be a quasi-perfect regular alteration.
    Then $h^*:\QCoh(X)\to\QCoh(R)$ is of universal descent in the sense of \autoref{definition: t-geometric of universal descent}.
\end{lemma}
\begin{proof}
    Since $X$ is noetherian and every morphism of finite type  over a locally noetherian base is of finite presentation, \cite[\href{https://stacks.math.columbia.edu/tag/0ETW}{Lemma 0ETW}]{stacks-project}  implies that every regular alteration $h : R\to X$ is an \emph{h-cover} in the sense of Voevodsky \cite[\href{https://stacks.math.columbia.edu/tag/0ETS}{Definition 0ETS}]{stacks-project}. 
    Now by \cite[Proposition 11.25]{bhatt2017projectivity} every $h$-cover $h:R\to X$  of noetherian schemes is \emph{descendable} in the sense of \cite[Definition 3.18]{mathew2016galois}. 
    In particular, the derived pullback $h^*:\QCoh(X)\to\QCoh(R)$ is of universal descent in the sense of \autoref{definition: morphism of universal descent}. 
    Indeed, point $(1)$ follows by the quasi-perfection of $h$, whereas point $(2)$ follows from Bhatt and Scholze's result. 
    By regarding $h^*$ as a functor between tor-finite categories, it immediately follows that it is also of universal descent in the sense of \autoref{definition: t-geometric of universal descent}.
\end{proof}

We deduce the following consequence of \autoref{theorem: functors out of coh proof}.
\begin{corollary}\label{corollary: ND2 for schemes}
    Let $f:X\to Y$ be a proper map and assume that $Y$ is noetherian. 
    Assume that $X$ admits a quasi-perfect regular alteration.
    Then there is an equivalence of categories
    \[
    \text{Perf}(X)^\op\to\Fun_{\text{Perf}(Y)}^\ex(D^b_{\text{coh}}(X),D^b_{\text{coh}}(Y))
    \]
    induced by the dual $\text{Perf}(Y)$-Yoneda embedding.
\end{corollary}

\subsection{Spectral algebraic spaces}\label{subsection: spectral deligne-mumford stacks}

We can generalize the previous example by considering spectral Deligne-Mumford stacks. Let us denote by $\widehat{\Spc}$ the very large category of large spaces. 
Denote also by $\CAlg^\cn(\Sp)$ the category of connective $\mb{E}_\infty$-rings. 

\begin{definition}[\protect{\cite[Definition 9.1.0.1]{Lurie-SAG}}]
    Let  $X : \CAlg^\cn(\Sp)\to \widehat{\Spc}$ be a functor. We will say that $X$ is \emph{quasi-geometric stack} if:
    \begin{enumerate*}
        \item The functor $X$ satisfies descent with respect to the fpqc topology.
        \item The diagonal map $\delta : X \times X\to X$ is quasi-affine
        \item There exists a connective $\mb{E}_\infty$-ring $A$ and a faithfully flat morphism $\Spec(A)\to X$.
    \end{enumerate*}
\end{definition}
The class of quasi-geometric stacks contains many algebro-geometric objects that arise in practice. 
In \cite[Section 6.2]{Lurie-SAG}, Lurie assigns to every quasi-geometric stack $X$ a category of quasi-coherent sheaves $\QCoh(X)$. 
This category defines an algebra $\QCoh(X)\in\CAlg(\Pr^\L_\st)$ and admits a $t$-structure which is right-complete presentable $t$-structure, left complete and compatible with $\otimes$.
However, in general, $\QCoh(X)$  is not compactly generated for two reasons.
The first obstruction happens since the monoidal unit $\mathcal{O}_X$ is \emph{not} a compact object of $\QCoh(X)$. 
The second obstruction is in the compact generation, since $\QCoh(X)$ does not have enough perfect complexes in general.
Restricting to  \emph{perfect stacks} fixes this issue, since, roughly speaking, a quasi-geometric stack $X$ is perfect if the canonical map $\Ind(\text{Perf}(X))\to\QCoh(X)$ is an equivalence of categories (see \cite[Proposition 9.4.4.5]{Lurie-SAG}). 
\begin{definition}[\protect{\cite[Definition 9.4.4.1]{Lurie-SAG}}]
    Let $X : \CAlg^\cn(\Sp)\to\widehat{\Spc}$ be a functor. 
    We will say that $X$ is a \emph{perfect stack} if it satisfies the following conditions:
    \begin{enumerate*}
    \item The functor $X$ is a quasi-geometric stack. 
    \item The structure sheaf $\mathcal{O}_X$ is a compact object of $\QCoh(X)$.
    \item Every quasi-coherent sheaf $\mathcal{F}\in\QCoh(X)$ can be obtained as the colimit of a filtered diagram $\{\mathcal{F}_i\}_{i\in I}$, where each $\mathcal{F}_i$ is a perfect object of $\QCoh(X)$.
    \end{enumerate*}
\end{definition}
\begin{remark}
    Since for a perfect stack $X$ the compact objects  coincide with the perfect ones, it follows that $\QCoh(X)\in\CAlg^\rig(\Pr^\L_\st)$ is rigidly-compactly generated.
\end{remark}

Compact generation by a single object is more subtle and needs some restriction. 
\begin{remark}
    Recall that a \emph{spectral algebraic space} is a spectral Deligne-Mumford stack $X$ such that the mapping space $\Hom(\text{Spét}(R), X)$ is discrete for every commutative ring $R$. See \cite[Definition 1.6.8.1]{Lurie-SAG}. 
    Now \cite[Proposition 9.6.1.1]{Lurie-SAG} shows that if $X$ is a quasi-compact, quasi-separated spectral algebraic space then its functor of points defines a perfect stack, allowing us to deduce that $\QCoh(X)$ is rigidly-compactly generated. 
    Its compact generation by a single object follows then by \cite[Corollary 9.6.3.2]{Lurie-SAG}. 
\end{remark}
To sum up, we have the following.
\begin{corollary}
    Let $X$ be a quasi-compact and quasi-separated spectral algebraic space.
    Then the category of quasi-coherent sheaves $\QCoh(X)$ on $X$, equipped with the standard $t$-structure, is tor-finite. 
    Moreover, $X$ comes equipped with a compact connective generator $G$ such that $\{G\}\subseteq\Perf(X)_{\geq0}$ determines a finite generating family.  
    Finally, if $X$ is noetherian,  then $\QCoh(X)$ is coherent.
\end{corollary}
\begin{proof}
    It suffices to understand the objects of finite tor dimension, which, by construction, are designed to  be the objects of finite tor dimension in the geometric sense.  
    Since every perfect complex is such, the claim follows. 
\end{proof}
We now study morphisms of quasi-compact quasi-separated spectral algebraic spaces. 
\begin{remark}
    Let $f:X\to Y$ be a morphism of quasi-compact quasi-separated spectral algebraic spaces. 
    In this case $f$ determines a symmetric monoidal and colimit preserving functor $f^*:\QCoh(Y)\to\QCoh(X)$ which is also right $t$-exact. 
    Furthermore:
    \begin{enumerate*}
        \item The pushforward $f_*$ is right $t$-exact up  to  a shift precisely when $f$ is of finite cohomological dimension. 
        In this case, $f^*$ defines a functor of tor-finite categories. 
        \item Thanks to the direct image theorem, \cite[Theorem 5.6.02]{Lurie-SAG},  any proper and locally almost of finite presentation of spectral Deligne-Mumford stacks $f : X \to Y$  is such that $f_* : \QCoh(X)\to\QCoh(Y)$ carries almost perfect objects to almost perfect objects.
    \end{enumerate*}
    The definitions of proper  and locally almost of finite presentation morphisms are in \cite[Definition 5.1.2.1 and Definition 4.2.0.1]{Lurie-SAG}. 
\end{remark}
Since pseudo-coherent objects are defined to coincide with Lurie's almost perfect objects on a quasi-compact spectral Deligne-Mumford stack, we deduce the following.
\begin{corollary}
    Let $f: X\to Y$ be a morphism of finite cohomological dimension of quasi-compact quasi-separated spectral algebraic spaces which is proper  and locally almost of finite presentation. 
    Then $f^*:\QCoh(Y)\to\QCoh(X)$ is quasi-proper. 
\end{corollary}

We now deduce the following consequence of  \autoref{theorem: functors out of Cc giovanni proof}.
\begin{corollary}\label{corollary: ND1 for spectral schemes}
    Let $f: X\to Y$ be a morphism of finite cohomological dimension of quasi-compact quasi-separated spectral algebraic spaces which is proper  and locally almost of finite presentation. 
    Assume also that $Y$ is noetherian. 
    Then there are equivalences of categories
    \[
    \text{PCoh}(X)\to \Fun^\ex_{\text{Perf}(Y)}(\text{Perf}(X)^\op,\text{PCoh}(Y)), \qquad \text{Coh}(X)\to \Fun^\ex_{\text{Perf}(Y)}(\text{Perf}(X)^\op,\text{Coh}(Y)).
    \]
    induced by the $\text{Perf}(Y)$-Yoneda embedding. 
\end{corollary}
\begin{proof}
    Let $G$ be a compact generator of $\QCoh(X)$.
    It suffices  to show that, for a perfect complex $G$, generating $\QCoh(X)$ then the  Yoneda embedding $\QCoh(X)(G,-):\QCoh(X) \to\QCoh(Y)$ detects eventually coconnective objects. 
    By reducing to $Y$ affine, the proof follows by \cite[Proposition 7.0.2 and Remark 7.0.3]{ben2017integral}. 
    Their argument, which is for algebraic spaces, can be  carried without any modification also for spectral algebraic spaces and works for morphisms of  quasi-compact quasi-separated spectral algebraic spaces.
\end{proof}

Unfortunately, we do not know any application  of \autoref{theorem: functors out of coh proof} in the realm of spectral algebraic geometry. 
Actually, \autoref{example: regular not good for spectra} shows that  we are far from proving an honest converse to \autoref{theorem: functors out of Cc giovanni proof}.

\bibliographystyle{alpha}
\bibliography{Bibliography}

\end{document}

%% file: Comands.tex
\theoremstyle{definition}
\newtheorem{theorem}{Theorem}[subsection]
\newtheorem{definition}[theorem]{Definition}
\newtheorem{proposition}[theorem]{Proposition}
\newtheorem{corollary}[theorem]{Corollary}
\newtheorem{lemma}[theorem]{Lemma}
\newtheorem{remark}[theorem]{Remark}
\newtheorem{notation}[theorem]{Notation}
\newtheorem{example}[theorem]{Example}
\newtheorem{warning}[theorem]{Warning}
\newtheorem{construction}[theorem]{Construction}

\newtheorem*{theorem*}{Theorem}
\newtheorem*{corollary*}{Corollary}
\newtheorem*{definition*}{Definition}
\newtheorem*{remark*}{Remark}
\newtheorem*{question*}{Question}

\newcommand{\rom}[1]{\uppercase\expandafter{\romannumeral #1\relax}}

\makeatletter
\@ifpackageloaded{hyperref}%
  {\newcommand{\mylabel}[2]
    {\protected@write\@auxout{}{\string\newlabel{#1}{{#2}{\thepage}%
      {\@currentlabelname}{\@currentHref}{}}}}}%
  {\newcommand{\mylabel}[2]
    {\protected@write\@auxout{}{\string\newlabel{#1}{{#2}{\thepage}}}}}
\makeatother

\usepackage{enumitem}
\newenvironment{enumerate*}
  {\begin{enumerate}[label=(\arabic*),topsep=-\parskip+\jot,itemsep=-\parskip+\jot]}
  {\end{enumerate}}
  
\newenvironment{enumerate**}
  {\begin{enumerate}[label=(\alph*),topsep=-\parskip+\jot,itemsep=-\parskip+\jot]}
  {\end{enumerate}}
  
\newenvironment{enumerate***}
  {\begin{enumerate}[label=(*),topsep=-\parskip+\jot,itemsep=-\parskip+\jot]}
  {\end{enumerate}}

 \usepackage{bbm}
\let\mb\mathbbm 

\let\mc\mathcal

\def\N{\mathbb{N}}
\def\Z{\mathbb{Z}}

\DeclareFontFamily{U}{dmjhira}{}
\DeclareFontShape{U}{dmjhira}{m}{n}{ <-> dmjhira }{}
\DeclareRobustCommand{\yo}{\text{\usefont{U}{dmjhira}{m}{n}\symbol{"48}}}

\def\id{\text{id}}
\DeclareMathOperator\Fun{Fun}
\DeclareMathOperator\Hom{Hom}

\def\IHom{\underline{\text{Hom}}}

\def\Ind{\text{Ind}}

\def\Mod{\text{Mod}}

\def\op{\text{op}}

\def\perf{\text{perf}}
\def\st{\text{st}}

\def\Spc{\text{Spc}}
\def\Sp{\text{Sp}}
\def\Cat{\text{Cat}}
\def\Pr{\text{Pr}}
\def\CAlg{\text{CAlg}}

\def\ex{\text{ex}}

\def\Tot{\text{Tot}}

\def\rig{\text{rig}}

\newcommand{\Coh}[1]{\text{Coh}(\mc{#1})}
\newcommand{\PCoh}[1]{\text{PCoh}(\mc{#1})}

\newcommand{\Tor}[1]{\text{Tor}(\mc{#1})}

\def\into{\hookrightarrow}

\def\iso{\cong}

\def\susp{\Sigma}

\def\iff{\Leftrightarrow}
\def\L{\text{L}}
\def\R{\text{R}}

\def\fib{\text{fib}}
\def\cofib{\text{cofib}}

\DeclareMathOperator\colim{colim}

\def\phi{\varphi}
\def\epsilon{\varepsilon}

\newcommand{\ts}[1]{(\mc{#1}_{\geq0}, \mc{#1}_{\leq0})}
\newcommand{\Acc}[1]{\text{Acc}^\otimes(\mathcal{#1})}

\def\QCoh{\text{QCoh}}
\def\Perf{\text{Perf}}
\def\Spec{\text{Spec}}
\def\cn{\text{cn}}